\documentclass[11pt,a4paper]{article}
\usepackage{pict2e}
\usepackage{booktabs}
\usepackage{color}
\usepackage{epsfig}
\usepackage{epstopdf}
\usepackage{longtable}
\usepackage[T1]{fontenc}
\usepackage{amsmath}
\usepackage{cases}
\usepackage{geometry}
\usepackage{amsbsy,latexsym,amsfonts, epsfig, color, authblk, amssymb, graphics, bm}
\usepackage{epsf,slidesec,epic,eepic}
\usepackage{fancybox}
\usepackage{fancyhdr}
\usepackage{setspace}
\usepackage{nccmath}
\usepackage{diagbox}
\usepackage[colorlinks, citecolor=blue]{hyperref}

\newtheorem{theorem}{Theorem}[section]
\newtheorem{lemma}{Lemma}[section]
\newtheorem{algorithm}{Algorithm}[section]

\newtheorem{proposition}{Proposition}[section]

\newtheorem{assumption}{Assumption}[section]
\newtheorem{remark}{Remark}[section]

\newenvironment{proof}{{\noindent \bf Proof:}}{\hfill$\Box$\medskip}

\definecolor{lred}{rgb}{1,0.8,0.8}
\definecolor{lblue}{rgb}{0.8,0.8,1}
\definecolor{dred}{rgb}{0.6,0,0}
\definecolor{dblue}{rgb}{0,0,0.5}
\definecolor{dgreen}{rgb}{0,0.5,0.5}

 \title{Inexact indefinite proximal ADMMs for 2-block separable convex programs and applications to 4-block DNNSDPs}

\author{Li Shen\footnote{Department of Mathematics, South China University of Technology, Guangzhou, 510641, China (shen.li@mail.scut.edu.cn).}
 \ \ {\rm and}\ \ Shaohua Pan\footnote{Corresponding author. Department of Mathematics, South China University of Technology, Tianhe District of Guangzhou City, China
 (shhpan@scut.edu.cn).}}

 \date{July 8, 2015}

\begin{document}

 \maketitle

 \begin{abstract}
  This paper is concerned with two-block separable convex minimization problems with linear constraints,
  for which it is either impossible or too expensive to obtain the exact solutions of the subproblems
  involved in the proximal ADMM (alternating direction method of multipliers). Such structured convex
  minimization problems often arise from the two-block regroup of three or four-block separable
  convex optimization problems with linear constraints, or from the constrained total-variation superresolution
  image reconstruction problems in image processing. For them, we propose an inexact indefinite proximal
  ADMM of step-size $\tau\in\!(0,\frac{\sqrt{5}+1}{2})$ with two easily implementable inexactness
  criteria to control the solution accuracy of subproblems, and establish the convergence under
  a mild assumption on indefinite proximal terms. We apply the proposed inexact indefinite
  proximal ADMMs to the three or four-block separable convex minimization problems with linear constraints,
  which come from the important class of doubly nonnegative semidefinite programming
  (DNNSDP) problems with many linear equality and/or inequality constraints.
  Numerical results indicate that the inexact indefinite proximal ADMM with the absolute error criterion
  has a comparable performance with the directly extended multi-block ADMM of step-size $\tau=1.618$
  without convergence guarantee, whether in terms of the number of iterations or the computation time.


  \end{abstract}

\section{Introduction}\label{sec1}

 Let $\mathbb{X},\mathbb{Y}$ and $\mathbb{Z}$ be the finite dimensional vector spaces
 endowed with the inner product $\langle \cdot,\cdot\rangle$ and its induced norm $\|\cdot\|$.
 Given closed proper convex functions $f\!:\mathbb{X}\to(-\infty,+\infty]$ and $g\!:\mathbb{Y}\!\to(-\infty,+\infty]$,
 we are concerned with the separable convex optimization problem
\begin{align}\label{prob}
 &\min_{x\in\mathbb{X},y\in\mathbb{Y}} f(x)+g(y)\nonumber\\
 &\quad {\rm s.t.}\ \ \mathcal{A}^*x+ \mathcal{B}^*y = c,
 \end{align}
 where $\mathcal{A}\!: \mathbb{Z}\to\mathbb{X}$ and $\mathcal{B}\!: \mathbb{Z}\to\mathbb{Y}$
 are the given linear operators, $\mathcal{A}^*$ and $\mathcal{B}^*$ denote the adjoint operators
 of $\mathcal{A}$ and $\mathcal{B}$, respectively, and $c\in\mathbb{Z}$ is a given vector.

 \medskip

 As well known, there are many important cases with the form of (\ref{prob}), which include
 the covariance selection problems and semidefinite least squares problems in statistics
 \cite{BGA08,SMG10,Yuan12}, the sparse plus low-rank recovery problem arising from the so-called robust PCA
 (principle component analysis) with noisy and incomplete data \cite{WGRPM09,TYuan11},
 the constrained total-variation image restoration and reconstruction problems \cite{NWY10,ROF92},
 the simultaneous minimization of the nuclear norm and $\ell_1$-norm of a matrix arising from
 the low-rank and sparse representation for image classification and subspace clustering \cite{ZJD13,WXL13},
 and so on.

 \medskip

 For the structured convex minimization problem \eqref{prob}, the alternating direction method of multipliers
 (ADMM for short), first proposed by Glowinski and Marrocco \cite{GM75} and Gabay and Mercier \cite{GM76},
 is one of the most popular methods. For any given $\sigma>0$, let $L_{\sigma}\!:\mathbb{X}\times\mathbb{Y}\times\mathbb{Z}\to\!(-\infty,+\infty]$
 denote the augmented Lagrangian function of problem \eqref{prob}
 \[
   L_{\sigma}(x,y,z):=f(x)+g(y) + \langle z,\mathcal{A}^*x+ \mathcal{B}^*y -c\rangle + \frac{\sigma}{2}\|\mathcal{A}^*x+ \mathcal{B}^*y -c\|^2.
 \]
 The ADMM, from an initial point $(x^0,y^0,z^0)\in {\rm dom}\,f\times{\rm dom}\,g\times\mathbb{Z}$,
 consists of the steps
  \begin{subnumcases}{}\label{subprobx}
   x^{k+1}\in\mathop{\arg\min}_{x\in\mathbb{X}}\ L_{\sigma}(x,y^k,z^{k}),\\
      \label{subproby}
   y^{k+1}\in\mathop{\arg\min}_{y\in\mathbb{Y}}\ L_{\sigma}(x^{k+1},y,z^{k}),\\
   z^{k+1}=z^{k} +\tau\sigma\big(\mathcal{A}^*x^{k+1}+\mathcal{B}^*y^{k+1}-c\big),
  \label{multiplier1}
  \end{subnumcases}
  where $\tau\in(0,\frac{1+\sqrt{5}}{2})$ is a constant to control the step-size in
  \eqref{multiplier1}. The iterative scheme of ADMM actually embeds a Gaussian-Seidel
  decomposition into each iteration of the classical augmented Lagraigan method of
  Hestenes-Powell-Rockafellar \cite{Henstenes76,Powell69,Roc76}, so that the challenging task
  (i.e., the exact solution or the approximate solution with a high precision of the Lagrangian
  minimization problem) is relaxed to several easy ones.

  \medskip

  Notice that the subproblems \eqref{subprobx} and \eqref{subproby} in the ADMM may have
  no closed-form solutions or even be difficult to solve. When the functions $f$ and $g$ enjoy
  a closed-form Moreau envelope, one usually introduces the proximal terms
  $\frac{1}{2}\|x-x^k\|_{\mathcal{P}_{\!f}}$ and $\frac{1}{2}\|y-y^k\|_{\mathcal{P}_{\!g}}$
  respectively into the subproblems \eqref{subprobx} and \eqref{subproby} to cancel the operators
  $\mathcal{A}\mathcal{A}^*$ and $\mathcal{B}\mathcal{B}^*$ so as to get the exact solutions of
  proximal subproblems. This is the so-called proximal-ADMM which,
  for a chosen initial point $(x^0,y^0,z^0)\in {\rm dom}\,f\times{\rm dom}\,g\times\mathbb{Z}$,
  consists of
  \begin{subnumcases}{}
   x^{k+1}=\mathop{\arg\min}_{x\in\mathbb{X}}\ L_{\sigma}(x,y^k,z^{k})+\frac{1}{2}\|x-x^k\|_{\mathcal{P}_{\!f}},\\
   \label{psubprobx}
   y^{k+1}=\mathop{\arg\min}_{y\in\mathbb{Y}}\ L_{\sigma}(x^{k+1},y,z^{k})+\frac{1}{2}\|y-y^k\|_{\mathcal{P}_{\!g}},\\
    \label{psubproby}
   z^{k+1}=z^{k} +\tau\sigma\big(\mathcal{A}^*x^{k+1}+\mathcal{B}^*y^{k+1}-c\big).
  \label{pmultiplier1}
  \end{subnumcases}
  The existing works on the proximal ADMM mostly focus on the positive definite proximal terms
  (see, e.g., \cite{HLHY02,WY12,ZBO11}). It is easy to see that the proximal subproblems with
  the positive definite proximal terms will have a big difference from the original subproblems of ADMM.
  In fact, as pointed out in the conclusion remarks of \cite{HLHY02}, ``large and positive definite
  proximal terms will lead to easy solution of subproblems, but the number of iterations will increase.
  Therefore, for subproblems which are not extremely ill-posed, the proximal parameters should be small.''
  In view of this, some researchers recently develop the semi-proximal or indefinite proximal ADMM
  \cite{XW11,FPST13,LST2014} by using the positive semidefinite even indefinite proximal terms.
  The numerical experiments in \cite{FPST13} show that such tighter proximal terms display
  better numerical performance. In addition, it is worthwhile to emphasize that the ADMM itself
  is a semi-proximal (of course an indefinite proximal) ADMM, but is not in the family of
  positive definite proximal ADMMs.

  \medskip

  In this paper we are interested in problem \eqref{prob} in which the functions $f$ and/or $g$
  may not have a closed-form Moreau envelope or the linear operators $\mathcal{A}$ and/or
  $\mathcal{B}$ have a large spectral norm (now the proximal subproblems with a positive definite
  proximal term are bad surrogates for those of the ADMM), for which it is impossible or
  too expensive to achieve the exact solutions of the proximal subproblems though they are unique.
  Such separable convex optimization problems arise directly from the constrained total-variation
  superresolution image reconstruction problems \cite{CHAN07,NWY11} in image processing,
  and the two-block regroup of three or four-block separable convex minimization problems.
  Indeed, for the following four-block separable convex minimization problem
  \begin{align}\label{gprob}
   &\min_{x_i\in\mathbb{X}_i}\ {\textstyle \sum_{i=1}^4}f_i(x_i)\nonumber\\
   &\  {\rm s.t.}\ \ {\textstyle \sum_{i=1}^4}\mathcal{A}_i^*x_i= c
  \end{align}
  where $f_i\!:\mathbb{X}_i\to(-\infty,+\infty]$ for $i=1,2,3,4$ are closed proper convex functions,
  and $\mathcal{A}_i\!: \mathbb{Z}\to\mathbb{X}_i$ for $i=1,2,3,4$ are linear operators,
  since the directly extended multi-block ADMM does not have the convergence guarantee
  (see the  counterexamples in \cite{CHYY14}), one may rearrange it as the form of \eqref{prob} by
  reorganizing any two groups of variables into one group, and then apply the classical ADMM for
  solving the two-block regrouped problem. Clearly, the exact solution of each subproblem of ADMM
  for the two-block regrouped problem is difficult to obtain due to the cross of two classes of variables.
  In particular, the two-block regroup resolving of multi-block separable convex optimization
  also has a separate study value.


  \medskip

  To resolve this class of difficult two-block separable convex minimization problems,
  we propose an inexact indefinite proximal ADMM with a step-size $\tau\in\!(0,\frac{\sqrt{5}+1}{2})$,
  in which the proximal subproblems are solved to a certain accuracy with two easily implementable
  inexactness criteria to control the accuracy. Here, an indefinite proximal term, instead of a positive definite
  proximal term, is introduced into each subproblem of the ADMM to guarantee that each proximal subproblem
  has a unique solution as well as becomes a good surrogate for the original subproblem of the ADMM.
  For the proposed inexact indefinite proximal ADMM, we establish its convergence under a mild assumption
  on the indefinite proximal terms. To the best of our knowledge, this is the first convergent
  inexact proximal ADMM in which step-size $\tau$ may take the value in the interval $(1,\frac{\sqrt{5}+1}{2})$.
  We notice that a few existing research papers on inexact versions of the ADMM all focus on
  the unit step-size; see \cite{EB92,HLHY02,NWY11,GHY2014,CWY2014}, and moreover, only the papers
  \cite{NWY11,GHY2014,CWY2014} develop truly implementable inexactness criteria in the exact solutions are not required.
  Our inexact indefinite proximal ADMM is using the same absolute error criterion and
  and a little different relative error from the one used in \cite{NWY11}.
  It is well known that the ADMM with $\tau=1.618$ requires less $20\%$ to $50\%$ iterations
  than the one with $\tau=1$, especially for those difficult SDP problems \cite{WGY10}. Thus, the proposed
  inexact indefinite proximal ADMMs with a large step-size is expected to have better performance.

  \medskip

  In this work, we apply the inexact indefinite proximal ADMMs to the three and four-block separable
  convex minimization problems with linear constraints, coming from the duality of the doubly
  nonnegative semidefinite programming (DNNSDP) problems with many linear equality
  and/or inequality constraints. Specifically, we solve the two-block regroupment for the dual problems
  of DNNSDPs with the inexact indefinite proximal ADMM. Observe that the iterates yielded by
  solving each subproblem in an alternating way can satisfy the optimality condition approximately.
  Hence, in the implementation of the inexact indefinite proximal ADMMs, we get the inexact solution
  of each subproblem by minimizing the two group of variables alternately. Numerical results indicate
  that the inexact indefinite proximal ADMM with the absolute error criterion is comparable with
  the directly extended multi-block ADMM with step-size $\tau=1.618$ whether in terms of
  the number of iterations or the computation time, while the one with the relative error criterion
  requires less outer-iterations but more computation time since the error criterion is more restrictive
  and requires more inner-iterations. Thus, the inexact indefinite proximal ADMM with
  the absolute error criterion provides an efficient tool
  for handling the three and four-block separable convex minimization problems.

  \medskip

  We observe that there are several recent works \cite{WHML2013,HXY14,HTY14,FHWY2014} to regroup
  the multi-block separable convex minimization problems into two-block or several subblocks,
  and then solve each subblock simultaneously by introducing a positive definite
  proximal term related to the numbers of subproblems. Such procedures lead to easily
  solvable subproblems, but their performance becomes worse due to larger proximal terms.

  \medskip

  The rest of this paper is organized as follows. Section \ref{sec2} gives some notations
  and the main assumption. Section \ref{sec3} describes the inexact indefinite proximal ADMMs
  and analyzes the properties of the sequence generated. The convergence of the inexact indefinite
  proximal ADMMs is established in Section \ref{sec4}. Section \ref{sec5} applies the inexact
  indefinite proximal ADMMs for solving the duality of the doubly DNNSDPs with many linear
  equality and/or inequality constraints. Some concluding remarks are given in Section \ref{sec6}.

 \section{Notations and assumption}\label{sec2}

  Notice that the functions $f\!:\mathbb{X}\to\!(-\infty,+\infty]$ and $g\!:\mathbb{Y}\to\!(-\infty,+\infty]$
  are closed proper convex, and the subdifferential mappings of closed proper convex functions are
  maximal monotone \cite[Theorem 12.17]{RW98}. Hence, there exist self-adjoint operators $\Sigma_{\!f}\succeq 0$
  and $\Sigma_g\succeq 0$ such that for all $x,\widetilde{x}\in{\rm dom}\,f,u\in\partial f(x)$
  and $\widetilde{u}\in\partial f(\widetilde{x})$,
  \begin{equation}\label{prop-f}
   f(x)\ge f(\widetilde{x})+\langle \widetilde{u},x-\widetilde{x}\rangle + \frac{1}{2}\|x-\widetilde{x}\|^2_{\Sigma_{f}}
   \ \ {\rm and}\ \ \langle u-\widetilde{u},x-\widetilde{x}\rangle\ge \|x-\widetilde{x}\|_{\Sigma_{\!f}}^2;
  \end{equation}
  and for all $y,\widetilde{y}\in{\rm dom}\,g,v\in\partial g(y)$ and $\widetilde{v}\in\partial g(\widetilde{y})$,
  \begin{equation}\label{prop-g}
   g(x)\ge g(\widetilde{y})+\langle \widetilde{v},y-\widetilde{y}\rangle + \frac{1}{2}\|y-\widetilde{y}\|^2_{\Sigma_{g}}
   \ \ {\rm and}\ \ \langle v-\widetilde{v},y-\widetilde{y}\rangle\ge \|y-\widetilde{y}\|_{\Sigma_g}^2.
  \end{equation}

 For a self-adjoint linear operator $\mathcal{T}\!:\mathbb{X}\to\mathbb{X}$, the notation
 $\mathcal{T}\succeq 0$ (respectively, $\mathcal{T}\succ0$) means that $\mathcal{T}$ is positive
 semidefinite (respectively, positive definite), that is, $\langle x,\mathcal{T}x\rangle\ge 0$
 for all $x\in\mathbb{X}$ (respectively, $\langle x,\mathcal{T}x\rangle>0$ for all $x\in\mathbb{X}\backslash\{0\}$).
 Given a self-adjoint positive semidefinite linear operator $\mathcal{T}\!:\mathbb{X}\to\mathbb{X}$,
 we denote by $\|\cdot\|_{\mathcal{T}}$ the norm induced by $\mathcal{T}$, i.e.,
 \[
   \|x\|_{\mathcal{T}}:=\sqrt{\langle x,\mathcal{T}x\rangle}\quad\ \forall x\in\mathbb{X}.
 \]
 Given a self-adjoint positive definite linear operator, we denote by $\lambda_{\rm max}(\mathcal{T})$ and
 $\lambda_{\rm min}(\mathcal{T})$ the largest eigenvalue and the smallest eigenvalue of $\mathcal{T}$,
 respectively, and by $D_{\mathcal{T}}(x,\Omega)$  the distance induced by $\mathcal{T}$ from $x$
 to a closed set $\Omega$, that is,
  \(
    D_{\mathcal{T}}(x,\Omega):=\min_{z\in\Omega}\|z-x\|_{\mathcal{T}}.
  \)
 When $\mathcal{T}$ is the identity operator, we suppress the notation $\mathcal{T}$ in $D_{\mathcal{T}}(x,\Omega)$
 and write simply $D(x,\Omega)$. Clearly, for any positive definite linear operator
 $\mathcal{T}\!:\mathbb{X}\to\mathbb{X}$ and $\gamma>0$,
  \begin{equation}\label{Cauchy}
    2|\langle u,v\rangle|\le \gamma^{-1}\|u\|_{\mathcal{T}}^2 +\gamma\|v\|_{\mathcal{T}^{-1}}^2
    \quad\ \forall\,u,v\in\mathbb{X}.
 \end{equation}
 In addition, for any $u,v\in\mathbb{X}$ and any self-adjoint linear operator $\mathcal{T}\!:\mathbb{X}\to\mathbb{X}$,
 the following two identities will be frequently used in the subsequent analysis:
  \begin{align}\label{identity}
    2\langle u,\mathcal{T}v\rangle
    &=\langle u,\mathcal{T}u\rangle +\langle v,\mathcal{T}v\rangle
                     -\langle u-v,\mathcal{T}(u-v)\rangle\nonumber\\
    &=\langle u+v,\mathcal{T}(u+v)\rangle-\langle u,\mathcal{T}u\rangle-\langle v,\mathcal{T}v\rangle.
  \end{align}

  Throughout this paper, we make the following assumption for problem \eqref{prob}:
 \begin{assumption}\label{assump}
  Problem \eqref{prob} has an optimal solution, to say $(x^*,y^*)\in{\rm dom}\,f\times{\rm dom}\,g$,
  and there exists a point $(\widehat{x},\widehat{y})\in{\rm ri}({\rm dom}\,f\times{\rm dom}\,g)$
  such that $\mathcal{A}^*\widehat{x}+ \mathcal{B}^*\widehat{y} = b$.
 \end{assumption}

 Under Assumption \ref{assump}, from \cite[Corollary 28.2.2 \& 28.3.1]{Roc70} and \cite[Theorem 6.5 \& 23.8]{Roc70},
 it follows that there exists a Lagrange multiplier $z^*\in\mathbb{Z}$ such that
 \begin{equation}\label{optimal-cond}
  -\!\mathcal{A}z^*\in\partial f(x^*),\ -\mathcal{B}z^*\in\partial g(x^*)\ \ {\rm and}\ \
  \mathcal{A}^*x^*+\mathcal{B}^*y^*-c=0
 \end{equation}
 where $\partial f$ and $\partial g$ are the subdifferential mappings of $f$ and $g$, respectively.
 Moreover, any $z^*\in\mathbb{Z}$ satisfying (\ref{optimal-cond}) is an optimal solution to
 the dual problem of \eqref{prob}. In the sequel, we call $(x^*,y^*,z^*)\in{\rm dom}\,f\times{\rm dom}\,g\times\mathbb{Z}$
 a primal-dual solution pair of problem \eqref{prob}.

 \section{Inexact indefinite proximal ADMMs}\label{sec3}

  In this section, we describe the iteration steps of the inexact indefinite proximal ADMMs
  for solving problem (\ref{prob}), and then analyze the properties of the sequence generated.

  \medskip

  The iteration steps of our inexact indefinite proximal ADMMs are stated as follows.

  \bigskip

  \setlength{\fboxrule}{0.8pt}
  \noindent
  \fbox{
  \parbox{0.96\textwidth}
  {{\bf IEIDP-ADMM (Inexact indefinite proximal ADMM for (\ref{prob}))}\

   \noindent
   \begin{description}
   \item[(S.0)] Let $\sigma,\tau>0$ be given. Choose self-adjoint linear operators
                $\mathcal{P}_{\!f}\!:\mathbb{X}\to\mathbb{X}$ and \hspace*{0.05cm} $\mathcal{P}_{\!g}\!:\mathbb{Y}\to\mathbb{Y}$
                such that $\mathcal{T}_{\!f}\!:=\mathcal{P}_{\!f}+\Sigma_{f}+\sigma\mathcal{A}\mathcal{A}^* \succ 0$ and
                $\mathcal{T}_{\!g}\!:=\mathcal{P}_{\!g}+\Sigma_{g}+\sigma\mathcal{B}\mathcal{B}^* \succ 0$.
                \hspace*{0.05cm} Choose an initial point $(x^0,y^0,z^0)\in{\rm dom}\,f\times{\rm dom}\,g\times\mathbb{Z}$.
                Set $k:=0$.

   \item[(S.1)] Find $x^{k+1}\approx\mathop{\arg\min}_{x\in\mathbb{X}} \phi_k(x):=L_{\sigma}(x,y^k,z^k)+\frac{1}{2}\|x-x^k\|_{\mathcal{P}_{\!f}}^2$.

   \item[(S.2)]  Find $y^{k+1}\approx\mathop{\arg\min}_{y\in\mathbb{Y}} \psi_k(y):=L_{\sigma}(x^{k+1},y,z^k)+\frac{1}{2}\|y-y^k\|_{\mathcal{P}_{\!g}}^2$.

   \item[(S.3)]  Update the Lagrange multiplier $z^{k+1}$ via the following formula
                \[
                  z^{k+1} = z^{k}+\tau\sigma(\mathcal{A}^*x^{k+1}+ \mathcal{B}^*y^{k+1} - c).
                \]
  \item[(S.4)]  Let $k\leftarrow k+1$, and go to Step (S.1).

  \end{description}
   }
   }

  \bigskip
  \noindent
  The approximate optimality in (S.1) and (S.2) is measured by the following criteria:
  \begin{itemize}
   \item[\bf (C1)] $D\big(0,\partial\phi_k(x^{k+1})\big)\!\le\!\mu_{k+1},D\big(0,\partial\psi_k(y^{k+1})\big)\!\le\!\nu_{k+1}$ and
         $\sum_{k=0}^{\infty}\max(\mu_{k+1},\nu_{k+1})<\infty$;
  \end{itemize}
 \begin{itemize}
   \item[\bf (C2)] $D_{\!\mathcal{F}}\big(0,\partial\phi_k(x^{k+1})\big)\!\le\!\mu_{k+1}\|x^{k+1}\!-x^k\|_{\mathcal{T}_{\!f}}$,
                    $D_{\mathcal{G}}\big(0,\partial\psi_k(y^{k+1})\big)\!\le\!\nu_{k+1}\|y^{k+1}\!-y^k\|_{\mathcal{T}_{\!g}}$
                    and $\sum_{k=0}^{\infty}\max(\mu_{k+1},\nu_{k+1})<\infty$, where $\mathcal{F}\!:\mathbb{X}\!\to\mathbb{X}$
                    and $\mathcal{G}\!:\mathbb{Y}\to\mathbb{Y}$ are self-adjoint positive definite linear operators with
                    $\mathcal{F}^{-1}\preceq \mathcal{T}_{\!f}$ and $\mathcal{G}^{-1}\preceq  \mathcal{T}_{\!g}$;
 \end{itemize}
 \begin{itemize}
   \item[\bf (C2')] $D_{\!\mathcal{F}}\big(0,\partial\phi_k(x^{k+1})\big)\!\le\!\mu_{k+1}\|x^{k+1}\!-x^k\|_{\mathcal{T}_{\!f}}$,
                    $D_{\mathcal{G}}\big(0,\partial\psi_k(y^{k+1})\big)\!\le\!\nu_{k+1}\|y^{k+1}\!-y^k\|_{\mathcal{T}_{\!g}}$
                    and $\sum_{k=0}^{\infty}\max(\mu_{k+1}^2,\nu_{k+1}^2)<\infty$, where $\mathcal{F}$
                    and $\mathcal{G}$ are same as the one in (C2).
 \end{itemize}

 Notice that (C1) is an absolute error criterion, while (C2) and (C2') are a relative error criterion.
 Clearly, when the approximate optimality of $x^{k+1}$ and $y^{k+1}$ is measured by (C1),
 (S.1) and (S.2) are equivalent to finding $(x^{k+1},\xi^{k+1})$ and $(y^{k+1},\eta^{k+1})$ such that
  \begin{align}\label{xieta-C1}
 \left\{\begin{array}{c}
  \xi^{k+1}\in\partial\phi_k(x^{k+1}),\ \|\xi^{k+1}\|\le \mu_{k+1}\ \ {\rm with}\ \ \sum_{k=0}^{\infty}\mu_{k+1}\!<\!\infty,\\
  \!\eta^{k+1}\in\partial\psi_k(y^{k+1}),\ \|\eta^{k+1}\|\le \nu_{k+1}\ \ {\rm with}\ \ \sum_{k=0}^{\infty}\nu_{k+1}\!<\!\infty.\\
  \end{array}\right.
 \end{align}
 If the approximate optimality of $x^{k+1}$ and $y^{k+1}$ is measured by (C2) or (C2'),
 (S.1) and (S.2) are equivalent to finding $(x^{k+1},\xi^{k+1})$ and $(y^{k+1},\eta^{k+1})$ such that
 with $p=1$ or $2$,
 \begin{align}\label{xieta-C23}
  \left\{\!\begin{array}{c}
  \xi^{k+1}\in\partial\phi_k(x^{k+1}),\ \|\xi^{k+1}\|_{\mathcal{F}}\le \mu_{k+1}\|x^{k+1}\!-x^k\|_{\mathcal{T}_f}
  \ \ {\rm with}\ \sum_{k=0}^{\infty}\mu_{k+1}^p<\!\infty,\\
  \!\eta^{k+1}\in\partial\psi_k(y^{k+1}),\ \|\eta^{k+1}\|_{\mathcal{G}}\le \nu_{k+1}\|y^{k+1}\!-y^k\|_{\mathcal{T}_g}
  \ \ {\rm with}\ \sum_{k=0}^{\infty}\nu_{k+1}^p<\!\infty.
  \end{array}\right.
 \end{align}

 \vspace{-0.5cm}
 \begin{remark}\label{remark1}
  (a) When the proximal operators $\mathcal{P}_{\!f}$ and $\mathcal{P}_{\!g}$ are chosen as
  $\beta\mathcal{I}$ for a constant $\beta>\!0$ and the step-size $\tau$ is set to be $1$,
  the IEIDP-ADMM with (C1) reduces to the IADM1 in \cite{NWY11}. If, in addition,
  taking $\mathcal{F}=\mathcal{G}=\frac{1}{\beta}\mathcal{I}$, the IEIDP-ADMM with (C2') requires
  \begin{align*}
  \left\{\!\begin{array}{c}
  \xi^{k+1}\in\partial\phi_k(x^{k+1}),\ \|\xi^{k+1}\|\le \mu_{k+1}\sqrt{\beta}\|x^{k+1}\!-x^k\|_{\sigma\mathcal{A}\mathcal{A}^*+\beta\mathcal{I}}
  \ \ {\rm with}\ \sum_{k=0}^{\infty}\mu_{k+1}^2<\!\infty,\\
  \!\eta^{k+1}\in\partial\psi_k(y^{k+1}),\ \|\eta^{k+1}\|\le \nu_{k+1}\sqrt{\beta}\|y^{k+1}\!-y^k\|_{\sigma\mathcal{B}\mathcal{B}^*+\beta\mathcal{I}}
  \ \ {\rm with}\ \sum_{k=0}^{\infty}\nu_{k+1}^2<\!\infty,
  \end{array}\right.
  \end{align*}
  whereas the LADM2 in \cite{NWY11} is actually requiring that $\xi^{k+1}$ and $\eta^{k+1}$ satisfy
  \begin{align*}
  \left\{\!\begin{array}{c}
  \xi^{k+1}\in\partial\phi_k(x^{k+1}),\ \|\xi^{k+1}\|\le \mu_{k+1}\beta\|x^{k+1}\!-x^k\|
  \ \ {\rm with}\ \sum_{k=0}^{\infty}\mu_{k+1}^2<\!\infty,\\
  \!\eta^{k+1}\in\partial\psi_k(y^{k+1}),\ \|\eta^{k+1}\|\le \nu_{k+1}\beta\|y^{k+1}\!-y^k\|
  \ \ {\rm with}\ \sum_{k=0}^{\infty}\nu_{k+1}^2<\!\infty.
  \end{array}\right.
  \end{align*}
  Since $\sqrt{\beta}\|x^{k+1}\!-x^k\|_{\sigma\mathcal{A}\mathcal{A}^*+\beta\mathcal{I}}\ge\beta\|x^{k+1}\!-x^k\|$
  and $\sqrt{\beta}\|y^{k+1}\!-y^k\|_{\sigma\mathcal{B}\mathcal{B}^*+\beta\mathcal{I}}\ge\beta\|y^{k+1}\!-y^k\|$,
  the above inexact criterion (C2') is looser than Criterion 2 used in \cite{NWY11}.

  \medskip
  \noindent
  (b) When $\mathcal{P}_{\!f}$ and $\mathcal{P}_{\!g}$ are chosen to be self-adjoint positive semidefinite operators,
  the IEIDP-ADMMs with $\mu_{k}\equiv\nu_{k}\equiv0$ reduce to the semi-proximal ADMM in \cite{XW11,FPST13}.

  \medskip
  \noindent
  (c) For the self-adjoint positive definite linear operators $\mathcal{F}$ and $\mathcal{G}$
  in (C2) and (C2'), an immediate choice is $\mathcal{F}=\frac{1}{\lambda_{\rm min}(\mathcal{T}_{\!f})}\mathcal{I}$
  and $\mathcal{G}=\frac{1}{\lambda_{\rm min}(\mathcal{T}_{\!g})}\mathcal{I}$.
  Since $\lambda_{\rm min}(\mathcal{T}_{\!f})$ and $\lambda_{\rm min}(\mathcal{T}_{\!g})$
  are easy to estimate, such a choice is convenient for the numerical implementation.
 \end{remark}

 Next we study the properties of the sequence generated by the IEIDP-ADMMs.
 For convenience, we let $h(x,y):=\mathcal{A}^*x+\mathcal{B}^*y-c$ for $(x,y)\in\mathbb{X}\times\mathbb{Y}$,
 and for each $k\ge 1$ write
 \begin{align*}
   x_e^k:=x^k-x^*,\ \ y_e^k:=y^k-y^*,\ \ z_e^k:=z^k-z^*;\qquad\\
    \Delta y^k\!:=y^k-y^{k-1},\ \Delta x^k\!:=x^k-x^{k-1},\
   \Delta z^k\!:=z^k-z^{k-1}.
 \end{align*}
 Using these notations and noting that $h(x^*,y^*)=0$, we can rewrite Step (S.3) as
 \begin{equation}\label{zk}
   z^{k} =z^{k+1}\!-\tau\sigma h(x^{k+1},y^{k+1})=z^{k+1}\!-\tau\sigma(\mathcal{A}^*x_e^{k+1}+\mathcal{B}^*y_e^{k+1}).
 \end{equation}

  \vspace{-0.5cm}
\begin{lemma}\label{Lemma1}
 Let $\big\{(x^{k},y^{k},z^{k})\big\}$ be the sequence generated by the IEIDP-ADMMs with
 $(x^{k},\xi^{k})$ and $(y^{k},\eta^{k})$ satisfying equation \eqref{xieta-C1} or \eqref{xieta-C23}.
 Suppose that Assumption \ref{assump} holds and the operator $\mathcal{P}_{\!g}$ also satisfies
 $\mathcal{P}_{\!g}+\frac{3}{8}\Sigma_{g}\succeq 0$. Then, for all $k\ge 0$ we have
 \begin{align*}
  & (2\!-\!\tau)\sigma\|h(x^{k+1},y^{k+1})\|^2
    +(\tau\sigma)^{-1}\big(\|z_e^{k+1}\|^2\!-\|z_e^k\|^2\big)
    +\|y_e^{k+1}\|^2_{\mathcal{T}_{g}}\!-\|y_e^{k}\|^2_{\mathcal{T}_{g}}
    \nonumber\\
  &  +\|x_e^{k+1}\|^2_{\mathcal{P}_{\!f}+\Sigma_{\!f}}\!-\|x_e^{k}\|^2_{\mathcal{P}_{\!f}
     +\Sigma_{\!f}}+\|\Delta y^{k+1}\|^2_{\mathcal{P}_{\!g}+\frac{3}{4}\Sigma_{g}}
     -\|\Delta y^k\|^2_{\mathcal{P}_{\!g}+\frac{3}{4}\Sigma_{g}}
     \nonumber\\
  &\le 2(1\!-\!\tau)\sigma\langle h(x^k,y^k), \mathcal{B}^*\Delta y^{k+1}\rangle
      +r^{k+1}\!-\|\Delta x^{k+1}\|_{\mathcal{P}_{\!f}+\!\frac{1}{2}\Sigma_{\!f}}^2\!-\!\|\Delta y^{k+1}\|_{\mathcal{T}_g}^2
 \end{align*}
  where $ r^{k+1}\!:=2\langle x_e^{k+1},\xi^{k+1}\rangle+ 2\langle y_e^{k+1},\eta^{k+1}\rangle
       +2\langle \eta^{k+1}\!-\!\eta^{k}, \Delta y^{k+1}\rangle$.
 \end{lemma}
 \begin{proof}
  From the expressions of $\phi_k$ and $\psi_k$ and equations \eqref{xieta-C1} and \eqref{xieta-C23}, it follows that
  \begin{align}\label{opt-x}
   \xi^{k+1}-\mathcal{A}z^{k}-\sigma\mathcal{A}(\mathcal{A}^*x^{k+1}+ \mathcal{B}^*y^{k}\!-\! c)
                                   -\mathcal{P}_{\!f}\Delta x^{k+1} &\in \partial f(x^{k+1}),\\
   \eta^{k+1}-\mathcal{B}z^{k}-\sigma\mathcal{B}(\mathcal{A}^*x^{k+1}+ \mathcal{B}^*y^{k+1}\!-\!c)
                                      -\mathcal{P}_{\!g}\Delta y^{k+1} &\in \partial g(y^{k+1}).
   \label{opt-y}
  \end{align}
  Substituting the first identity in \eqref{zk} into equations (\ref{opt-x}) and (\ref{opt-y}) respectively yields
 \begin{align*}
  & (\tau\!-\!1)\sigma\mathcal{A}h(x^{k+1},y^{k+1})
    \!-\!\mathcal{A}z^{k+1}\!+\!\sigma\mathcal{A}\mathcal{B}^*\Delta y^{k+1}
    -\mathcal{P}_{\!f}\Delta x^{k+1}\!+\xi^{k+1}\!\in\!\partial f(x^{k+1}),\nonumber\\
  &\qquad (\tau\!-\!1)\sigma\mathcal{B}h(x^{k+1},y^{k+1})
  -\mathcal{B}z^{k+1}-\mathcal{P}_{\!g}\Delta y^{k+1}+\eta^{k+1}\in \partial g(y^{k+1}).\nonumber
 \end{align*}
  In view of inequalities \eqref{prop-f} and \eqref{prop-g},
  from the last two inclusions and equation \eqref{optimal-cond} we have
  \begin{align*}
   \big\langle x_e^{k+1},(\tau\!-\!1)\sigma\mathcal{A}h(x^{k+1},y^{k+1})
    \!-\!\mathcal{A}z_e^{k+1}\!+\!\sigma\mathcal{A}\mathcal{B}^*\Delta y^{k+1}
    \!-\mathcal{P}_{\!f}\Delta x^{k+1}\!+\xi^{k+1}\big\rangle\!\ge \|x_e^{k+1}\|_{\Sigma_{\!f}}^2,\\
    \big\langle y_e^{k+1}, (\tau\!-\!1)\sigma\mathcal{B}h(x^{k+1},y^{k+1})-\mathcal{B}z_e^{k+1}
    \!-\mathcal{P}_{\!g}\Delta y^{k+1}\!+\eta^{k+1}\big\rangle\ge \|y_e^{k+1}\|_{\Sigma_{g}}^2.\qquad\quad
  \end{align*}
  Adding the last two inequalities together and using equation \eqref{zk} yields that
 \begin{align}\label{temp-equa20}
  & (\tau\!-\!1)\sigma\|h(x^{k+1},y^{k+1})\|^2-(\tau\sigma)^{-1}\langle \Delta z^{k+1},z_e^{k+1}\rangle
    + \sigma\langle h(x^{k+1},y^{k+1}), \mathcal{B}^*\Delta y^{k+1}\rangle\\
  &  -\langle x_e^{k+1},\mathcal{P}_{\!f}\Delta x^{k+1}\!-\xi^{k+1}\rangle
   -\langle y_e^{k+1},(\mathcal{P}_{\!g}\!+\!\sigma\mathcal{B}\mathcal{B}^*)\Delta y^{k+1}\!-\eta^{k+1}\rangle
  \ge \|x_e^{k+1}\|_{\Sigma_{\!f}}^2 \!+\!\|y_e^{k+1}\|_{\Sigma_{g}}^2.\nonumber
 \end{align}
 Next we deal with the term $\sigma\langle h(x^{k+1},y^{k+1}), \mathcal{B}^*\Delta y^{k+1}\rangle$
 in inequality \eqref{temp-equa20}. Notice that
 \begin{align}\label{temp-equa21}
  \sigma\langle h(x^{k+1},y^{k+1}),\mathcal{B}^*\Delta y^{k+1}\rangle
   &= (1\!-\!\tau)\sigma\langle h(x^{k+1},y^{k+1})\!-\!h(x^k,y^k),\mathcal{B}^*\Delta y^{k+1}\rangle\\
   &\quad +\langle \Delta z^{k+1},\mathcal{B}^*\Delta y^{k+1}\rangle
    +(1\!-\!\tau)\sigma\langle h(x^{k},y^{k}), \mathcal{B}^*\Delta y^{k+1}\rangle.\nonumber
 \end{align}
  We first bound the first two terms in \eqref{temp-equa20}. From equations \eqref{opt-y} and \eqref{zk},
  it follows that
 \begin{align*}
  -\mathcal{B}z^{k+1}+(\tau\!-\!1)\sigma\mathcal{B}h(x^{k+1},y^{k+1})
   -\mathcal{P}_{\!g}\Delta y^{k+1}+\eta^{k+1}\in \partial g(y^{k+1}), \\
  -\mathcal{B}z^{k}+(\tau\!-\!1)\sigma\mathcal{B}h(x^{k},y^{k})
  -\mathcal{P}_{\!g}\Delta y^{k}+\eta^{k}\in \partial g(y^{k}).\qquad\quad
 \end{align*}
 Combining the last two inclusions with the second inequality in \eqref{prop-g} yields that
 \begin{align}\label{temp-equa22}
  & (\tau\!-\!1)\sigma\big\langle h(x^{k+1},y^{k+1})-h(x^{k},y^{k}),\mathcal{B}^*\Delta y^{k+1}\big\rangle
  -\langle \Delta z^{k+1},  \mathcal{B}^*\Delta y^{k+1}\rangle \nonumber\\
  & -\!\langle \Delta y^{k+1}\!-\!\Delta y^k, \mathcal{P}_{\!g}\Delta y^{k+1}\rangle
   +\langle \eta^{k+1}\!-\!\eta^{k},\Delta y^{k+1}\rangle\ge\|\Delta y^{k+1}\|_{\Sigma_{g}}^2.
   \end{align}
  Using equation \eqref{identity} and the given assumption $\mathcal{P}_{\!g}+\frac{3}{8}\Sigma_{g} \succeq 0$,
  we have that
 \begin{align}\label{temp-equa23}
  \langle \Delta y^k\!-\Delta y^{k+1}, \mathcal{P}_{\!g}\Delta y^{k+1}\rangle
  &= \frac{1}{2}\|y^{k+1}\!-\!y^ {k-1}\|^2_{\mathcal{P}_{\!g}} - \frac{1}{2}\|\Delta y^k\|^2_{\mathcal{P}_{\!g}}
     -\frac{1}{2}\|\Delta y^{k+1}\|^2_{\mathcal{P}_{\!g}}-\|\Delta y^{k+1}\|^2_{\mathcal{P}_{g}} \nonumber\\
  &\le \frac{1}{2}\|y^{k+1}\!-\!y^{k-1}\|^2_{\mathcal{P}_{\!g}+\frac{3}{8}\Sigma_{g}} - \frac{1}{2}\|\Delta y^k\|^2_{\mathcal{P}_{\!g}}
      -\frac{3}{2}\|\Delta y^{k+1}\|^2_{\mathcal{P}_{\!g}} \nonumber\\
  &\le \frac{1}{2}\|\Delta y^k\|^2_{\mathcal{P}_{\!g}+\frac{3}{4}\Sigma_{g}}
      -\frac{1}{2}\|\Delta y^{k+1}\|^2_{\mathcal{P}_{\!g}+\frac{3}{4}\Sigma_{g}}+\frac{3}{4}\|\Delta y^{k+1}\|^2_{\Sigma_{g}}
 \end{align}
 where the last inequality is using
 \(
  \frac{1}{2}\big\|y^{k+1}\!- y^{k-1}\big\|^2_{\mathcal{P}_{\!g}+\frac{3}{8}\Sigma_{g}}
  \!\le\!\|\Delta y^k\|^2_{\mathcal{P}_{\!g}+\frac{3}{8}\Sigma_{g}}
     +\|\Delta y^{k+1}\|^2_{\mathcal{P}_{\!g}+\frac{3}{8}\Sigma_{g}}.
 \)
 Combining inequalities \eqref{temp-equa23} and \eqref{temp-equa22} with equation \eqref{temp-equa21},
 we immediately obtain
 \begin{align}\label{temp-ineq24}
  \sigma\langle h(x^{k+1},y^{k+1}),\mathcal{B}^*\Delta y^{k+1}\rangle
  &\le (1\!-\!\tau)\sigma\langle h(x^k,y^k), \mathcal{B}^*\Delta y^{k+1}\rangle
    +\langle \eta^{k+1}\!-\!\eta^{k},\Delta y^{k+1}\rangle \nonumber\\
  &\quad \!+\frac{1}{2}\|\Delta y^k\|^2_{\mathcal{P}_{\!g}+\frac{3}{4}\Sigma_{g}}
             \!-\frac{1}{2}\|\Delta y^{k+1}\|^2_{\mathcal{P}_{\!g}+\frac{3}{4}\Sigma_{g}}
             \!-\!\frac{1}{4}\|\Delta y^{k+1}\|^2_{\Sigma_g}\!.
 \end{align}
 Now substituting inequality \eqref{temp-ineq24} into equation \eqref{temp-equa20},
 we immediately obtain that
 \begin{align}\label{temp-equa24}
 & (\tau\!-\!1)\sigma\|h(x^{k+1},y^{k+1})\|^2-(\tau\sigma)^{-1}\langle \Delta z^{k+1},z_e^{k+1}\rangle
    - \langle x_e^{k+1},\mathcal{P}_{\!f}\Delta x^{k+1}\rangle\nonumber\\
 &  +(1\!-\!\tau)\sigma\langle h(x^{k},y^{k}), \mathcal{B}^*\Delta y^{k+1}\rangle
    +\frac{1}{2}\|\Delta y^k \|^2_{\mathcal{P}_{\!g}+\frac{3}{4}\Sigma_{g}}
     \!-\frac{1}{2}\|\Delta y^{k+1}\|^2_{\mathcal{P}_{\!g}+\frac{3}{4}\Sigma_{g}}\nonumber\\
 & - \langle y_e^{k+1},(\mathcal{P}_{\!g}\!+\!\sigma\mathcal{B}\mathcal{B}^*)\Delta y^{k+1}\rangle
    +\frac{1}{2}r^{k+1}
  \ge \|x_e^{k+1}\|_{\Sigma_{\!f}}^2 + \|y_e^{k+1}\|_{\Sigma_{g}}^2 + \frac{1}{4}\|\Delta y^{k+1}\|^2_{\Sigma_g}.
 \end{align}
 By the first equality of \eqref{identity} and equation \eqref{zk}, the term $\langle \Delta z^{k+1},z_e^{k+1}\rangle$
 can be written as
 \[
   \langle \Delta z^{k+1},z_e^{k+1}\rangle = \frac{1}{2}\|z_e^{k+1}\|^2-\frac{1}{2}\|z_e^k\|^2+\frac{(\tau\sigma)^2}{2}\|h(x^{k+1},y^{k+1})\|^2.
 \]
 Applying equation \eqref{identity} to $\langle x_e^{k+1},\mathcal{P}_{\!f}\Delta x^{k+1}\rangle$
 and $\langle y_e^{k+1},(\mathcal{P}_{\!g}\!+\sigma\mathcal{B}\mathcal{B}^*)\Delta y^{k+1}\rangle$ yields
 \begin{align*}
 &\langle x_e^{k+1},\mathcal{P}_{\!f}\Delta x^{k+1}\rangle
   =\frac{1}{2}\|x_e^{k+1}\|_{\mathcal{P}_{\!f}}^2-\frac{1}{2}\|x_e^k\|_{\mathcal{P}_{\!f}}^2
    +\frac{1}{2}\|\Delta x^{k+1}\|_{\mathcal{P}_{\!f}}^2,\\
  &\langle y_e^{k+1},(\mathcal{P}_{\!g}\!+\sigma\mathcal{B}\mathcal{B}^*)\Delta y^{k+1}\rangle
   =\frac{1}{2}\|y_e^{k+1}\|_{\mathcal{P}_{\!g}\!+\sigma\mathcal{B}\mathcal{B}^*}^2
   \!-\!\frac{1}{2}\|y_e^k\|_{\mathcal{P}_{\!g}\!+\sigma\mathcal{B}\mathcal{B}^*}^2
    \!+\! \frac{1}{2}\|\Delta y^{k+1}\|_{\mathcal{P}_{\!g}\!+\sigma\mathcal{B}\mathcal{B}^*}^2.
 \end{align*}
 Substituting the last three equalities into inequality \eqref{temp-equa24}, we have that
 \begin{align*}
  & (\tau\!-\!2)\sigma\|h(x^{k+1},y^{k+1})\|^2+(\tau\sigma)^{-1}\big(\|z_e^k\|^2\!-\!\|z_e^{k+1}\|^2\big)
     + \big(\|y_e^{k}\|^2_{\mathcal{T}_{\!g}}\!-\!\|y_e^{k+1}\|^2_{\mathcal{T}_{g}}\big)
   \nonumber\\
  & +\big(\|x_e^{k}\|^2_{\mathcal{P}_{\!f}+\Sigma_{\!f}}\!-\!\|x_e^{k+1}\|^2_{\mathcal{P}_{\!f}+\Sigma_{\!f}}\big)
    +\big(\|\Delta y^k\|^2_{\mathcal{P}_{\!g}+\frac{3}{4}\Sigma_{g}}
    \!-\|\Delta y^{k+1}\|^2_{\mathcal{P}_{\!g}+\frac{3}{4}\Sigma_{g}}\big)
     \nonumber\\
  &      + 2(1\!-\!\tau)\sigma\langle h(x^k,y^k), \mathcal{B}^*\Delta y^{k+1}\rangle +r^{k+1}-\|\Delta y^{k+1}\|_{\mathcal{T}_g}^2
     \nonumber\\
  &\ge \|x_e^{k+1}\|_{\Sigma_{\!f}}^2+\|x_e^{k}\|_{\Sigma_{\!f}}^2 +\|\Delta x^{k+1}\|_{\mathcal{P}_{\!f}}^2
       +\|y_e^{k+1}\|_{\Sigma_{g}}^2+\|y_e^{k}\|_{\Sigma_{g}}^2 -\frac{1}{2}\|\Delta y^{k+1}\|_{\Sigma_g}^2.\nonumber
 \end{align*}
 Notice that
 \(
 \|x_e^{k+1}\|_{\Sigma_{f}}^2\!+\!\|x_e^{k}\|_{\Sigma_{f}}^2 \geq \frac{1}{2}\|\Delta x^{k+1}\|_{\Sigma_{f}}^2
 \)
 and
 \(
  \|y_e^{k+1}\|_{\Sigma_{g}}^2\!+\!\|y_e^{k}\|_{\Sigma_{g}}^2 \geq \frac{1}{2}\|\Delta y^{k+1}\|_{\Sigma_{g}}^2.
 \)
 The last inequality implies the desired result. The proof is completed.
 \end{proof}

 \medskip

 The following lemma provides an upper bound for the term $r^{k+1}$ in Lemma \ref{Lemma1}.
 \begin{lemma}\label{lemma2}
  If (C1) is used in (S.1) and (S.2), then for any given $\gamma>0$ we have
  \begin{align}\label{cross-C1}
   |r^{k+1}|&\le \gamma^{-1}\big(\mu_{k+1}\|x_e^{k+1}\|_{\mathcal{T}_f}^2\!+\!\nu_{k+1}\|y_e^{k+1}\|_{\mathcal{T}_g}^2\big)
       +\gamma^{-1}(\nu_{k+1}\!+\nu_k)\|\Delta y^{k+1}\|^2_{\mathcal{T}_{g}}\nonumber\\
  &\quad +\gamma\|\mathcal{T}_{f}^{-1}\|\mu_{k+1}
         +\gamma\|\mathcal{T}_{g}^{-1}\|(2\nu_{k+1}\!+\nu_{k})\quad\ {\rm for}\ k\ge 1;
  \end{align}
   if the criterion (C2) is used for the minimization in (S.1) and (S.2), then
  \begin{align}\label{cross-C2}
  |r^{k+1}|&\le \mu_{k+1}\|x_e^{k+1}\|_{\mathcal{T}_f}^2+\nu_{k+1}\|y_e^{k+1}\|_{\mathcal{T}_g}^2
       +\mu_{k+1}\|\Delta x^{k+1}\|^2_{\mathcal{T}_{f}}\nonumber\\
  &\quad +(\nu_k+3\nu_{k+1})\|\Delta y^{k+1}\|^2_{\mathcal{T}_{g}}
         +\nu_k\|\Delta y^{k}\|^2_{\mathcal{T}_{g}}\quad\ {\rm for}\ k\ge 1;
  \end{align}
  and if (C2') is used for the minimization in (S.1) and (S.2), then for any given $\gamma>0$,
  \begin{align}\label{cross-C3}
  |r^{k+1}|&\le \gamma\big(\mu_{k+1}^2\|x_e^{k+1}\|_{\mathcal{T}_f}^2+\nu_{k+1}^2\|y_e^{k+1}\|_{\mathcal{T}_g}^2\big)
       +\gamma^{-1}\|\Delta x^{k+1}\|^2_{\mathcal{T}_{f}}\nonumber\\
  &\quad +\!\big(2\gamma^{-1}\!+\gamma\nu^2_{k}\!+\gamma\nu^2_{k+1}\big)\|\Delta y^{k+1}\|^2_{\mathcal{T}_{g}}
         \!+\!\gamma^{-1}\|\Delta y^{k}\|^2_{\mathcal{T}_{g}}\ \ {\rm for}\ k\ge 1.
  \end{align}
 \end{lemma}
 \begin{proof}
  When the criterion (C1) is used in (S.1) and (S.2), for any given $\gamma>0$ we have
  \begin{align}\label{xexik}
   2\big|\langle x_e^{k+1},\xi^{k+1}\rangle\big|
    \le \frac{\mu_{k+1}}{\gamma}\|x_e^{k+1}\|_{\mathcal{T}_{\!f}}^2\!+\!\gamma\|\mathcal{T}_{\!f}^{-1}\|\mu_{k+1},\qquad\qquad\quad\nonumber\\
   2\big|\langle y_e^{k+1},\eta^{k+1}\rangle\big|
    \le \frac{\nu_{k+1}}{\gamma}\|y_e^{k+1}\|_{\mathcal{T}_{g}}^2\!+\!\gamma\|\mathcal{T}_{\!g}^{-1}\|\nu_{k+1},\qquad\qquad\quad\\
    2\big|\langle\eta^{k+1}\!-\!\eta^k,\Delta y^{k+1}\rangle\big|
     \le\frac{\nu_{k+1}\!+\nu_{k}}{\gamma}\|\Delta y^{k+1}\|^2_{\mathcal{T}_{g}}
      +\gamma\|\mathcal{T}_{\!g}^{-1}\|(\nu_{k+1}\!+\nu_{k})\nonumber
  \end{align}
  for all $k\ge 1$. Indeed, when $\mu_{k+1}=0$, the first inequality in \eqref{xexik} holds since now $\xi^{k+1}=0$;
  and when $\mu_{k+1}>0$, from equations \eqref{Cauchy} and \eqref{xieta-C1} we have
  \[
    2\big|\langle x_e^{k+1},\xi^{k+1}\rangle\big|
    \le \frac{\mu_{k+1}}{\gamma}\|x_e^{k+1}\|_{\mathcal{T}_{\!f}}^2\!+\!\frac{\gamma}{\mu_{k+1}}\|\mathcal{T}_{\!f}^{-1}\|\|\xi^{k+1}\|^2
    \le \frac{\mu_{k+1}}{\gamma}\|x_e^{k+1}\|_{\mathcal{T}_{\!f}}^2\!+\!\gamma\|\mathcal{T}_{\!f}^{-1}\|\mu_{k+1}.
  \]
  Similarly, we can prove that the last two inequalities hold for all $k\ge 1$.
  Adding the three inequalities in \eqref{xexik} yields \eqref{cross-C1}.
  When (C2) is used in (S.1) and (S.2), for all $k\ge 1$
  \begin{align}\label{etayk}
   2\big|\langle x_e^{k+1},\xi^{k+1}\rangle\big|
      &\le \mu_{k+1}\|x_e^{k+1}\|_{\mathcal{T}_{\!f}}^2\!+\!\mu_{k+1}\|\Delta x^{k+1}\|_{\mathcal{T}_{\!f}}^2,\nonumber\\
   2\big|\langle y_e^{k+1},\eta^{k+1}\rangle\big|
    &\le \nu_{k+1}\|y_e^{k+1}\|_{\mathcal{T}_{\!g}}^2\!+\!\nu_{k+1}\|\Delta y^{k+1}\|_{\mathcal{T}_{\!g}}^2,\\
   2\big|\langle\eta^{k+1}\!-\!\eta^k,\Delta y^{k+1}\rangle\big|
    &\le (\nu_{k}+2\nu_{k+1})\|\Delta y^{k+1}\|^2_{\mathcal{T}_{\!g}}+\nu_k\|\Delta y^{k}\|^2_{\mathcal{T}_{g}}.\nonumber
 \end{align}
  Indeed, when $\nu_{k+1}=0$, the second inequality in \eqref{etayk} holds since now $\eta^{k+1}=0$;
  and when $\nu_{k+1}\ne 0$, from equation \eqref{Cauchy} and
  $\mathcal{F}^{-1}\preceq\mathcal{T}_{\!f}$ and $\mathcal{G}^{-1}\preceq\mathcal{T}_{\!g}$, it follows that
 \[
   2\big|\langle y_e^{k+1},\eta^{k+1}\rangle\big|
   \le \nu_{k+1}\|y_e^{k+1}\|_{\mathcal{G}^{-1}}^2\!+\!\frac{1}{\nu_{k+1}}\|\eta^{k+1}\|_{\mathcal{G}}^2
   \le \nu_{k+1}\|y_e^{k+1}\|_{\mathcal{T}_{\!g}}^2\!+\!\nu_{k+1}\|\Delta y^{k+1}\|_{\mathcal{T}_{\!g}}^2.
 \]
 Similarly, we can prove that another two inequalities hold for all $k\ge 1$.
 Summing up the three inequalities in \eqref{etayk} yields \eqref{cross-C2}.
  When the criterion (C2') is used, for any $\gamma>0$,
 \begin{align}\label{C2-etayk}
  &2\big|\langle x_e^{k+1},\xi^{k+1}\rangle\big|
    \le \gamma\mu_{k+1}^2\|x_e^{k+1}\|_{\mathcal{T}_{f}}^2\!+\!\frac{1}{\gamma}\|\Delta x^{k+1}\|_{\mathcal{T}_{f}}^2,\nonumber\\
  & 2\big|\langle y_e^{k+1},\eta^{k+1}\rangle\big|
   \le \gamma\nu_{k+1}^2\|y_e^{k+1}\|_{\mathcal{T}_{g}}^2\!+\!\frac{1}{\gamma}\|\Delta y^{k+1}\|_{\mathcal{T}_{g}}^2,\\
  & 2\big|\langle\eta^{k+1}\!-\!\eta^k,\Delta y^{k+1}\rangle\big|
     \le \big(\gamma^{-1}+\gamma\nu^2_{k}+\gamma\nu^2_{k+1}\big)\|\Delta y^{k+1}\|^2_{\mathcal{T}_{g}}
        \!+\!\gamma^{-1}\|\Delta y^{k}\|^2_{\mathcal{T}_{g}}\nonumber
 \end{align}
 for all $k\ge 1$. Indeed, when $\nu_{k}=0$, the third inequality in \eqref{C2-etayk} holds since now $\eta^{k}=0$;
  and when $\nu_{k}\ne 0$, from equation \eqref{Cauchy} and
  $\mathcal{F}^{-1}\preceq\mathcal{T}_{\!f}$ and $\mathcal{G}^{-1}\preceq\mathcal{T}_{\!g}$, it follows that
  \begin{align*}
  & 2\big|\langle\eta^{k+1}\!-\!\eta^k,\Delta y^{k+1}\rangle\big|
  \le (\gamma\nu^2_{k}+\gamma\nu_{k+1}^2)\|\Delta y^{k+1}\|^2_{\mathcal{G}^{-1}}
      +\frac{1}{\gamma\nu^2_{k}}\|\eta^{k}\|^2_{\mathcal{G}}
      +\frac{1}{\gamma\nu_{k+1}^2}\|\eta^{k+1}\|^2_{\mathcal{G}}\nonumber\\
  &\qquad\qquad\qquad\qquad\quad\
    \le \big(\gamma^{-1}+\gamma\nu^2_{k}+\gamma\nu^2_{k+1}\big)\|\Delta y^{k+1}\|^2_{\mathcal{T}_{g}}
        \!+\!\gamma^{-1}\|\Delta y^{k}\|^2_{\mathcal{T}_{g}}.
  \end{align*}
  Similarly, one can prove that another two inequalities hold for all $k\ge 1$.
 Summing up the three inequalities in \eqref{C2-etayk} yields \eqref{cross-C3}.
 The proof is completed.
 \end{proof}

 \medskip

 Based on the results of Lemmas \ref{Lemma1} and \ref{lemma2},
 we obtain the following proposition.
 \begin{proposition}\label{prop21}
  Let $\{(x^{k},y^{k},z^{k})\}_{k\ge 1}$ be the sequence generated by the IEIPD-ADMMs. Suppose that
  Assumption \ref{assump} holds and the operator $\mathcal{P}_{g}$ also satisfies $\mathcal{P}_{g}+\frac{3}{8}\Sigma_{g}\succeq 0$.
  \begin{description}
  \item[(a)] If the criterion (C1) is used in (S.1) and (S.2), then for any given $\gamma>0$ we have
  \begin{align*}
   & \big(1\!-\!\frac{\mu_{k+1}}{\gamma}\big)\|x_e^{k+1}\|^2_{\mathcal{P}_{\!f}+\Sigma_{\!f}}
     \!-\!\|x_e^k\|_{\mathcal{P}_{\!f}+\Sigma_{\!f}}^2
     \!+\!\big(1\!-\!\frac{2\mu_{k+1}\!+\!\nu_{k+1}}{\gamma}\big)\|y_e^{k+1}\|^2_{\mathcal{T}_g}\nonumber\\
  &  \!-\!\|y_e^k\|_{\mathcal{T}_g}^2\!+(\tau\sigma)^{-1}\big(\|z_e^{k+1}\|^2-\|z_e^k\|^2\big)
     \!+\!\|\Delta y^{k+1}\|^2_{\mathcal{P}_{\!g}+\frac{3}{4}\Sigma_{g}}
     \!-\!\|\Delta y^k\|^2_{\mathcal{P}_{\!g}+\frac{3}{4}\Sigma_{g}}\nonumber\\
  &  + \big(2\!-\!\tau-\!2\gamma^{-1}\mu_{k+1}\big)\sigma\|h(x^{k+1},y^{k+1})\|^2\nonumber\\
 \le&\ 2(1\!-\!\tau)\sigma\langle h(x^k,y^k),\mathcal{B}^*\Delta y^{k+1}\rangle
        \!-\big(1\!-\!\gamma^{-1}\nu_{k+1}\!-\!\gamma^{-1}\nu_{k}\big)\|\Delta y^{k+1}\|^2_{\mathcal{T}_g} \nonumber\\
    &\ -\!\|\Delta x^{k+1}\|^2_{\mathcal{P}_{\!f}+\frac{1}{2}\Sigma_{\!f}}
       \!+\!\gamma\|\mathcal{T}_{f}^{-1}\|\mu_{k+1}\!+\!\gamma\|\mathcal{T}_{g}^{-1}\|(2\nu_{k+1}+\nu_{k})
       \quad{\rm for}\ k\ge 1.
 \end{align*}
  \item[(b)] If the criterion (C2) is used for (S.1) and (S.2), then we have
 \begin{align}\label{result-C2}
  &\big(1\!-\!\mu_{k+1}\big)\|x_e^{k+1}\|_{\mathcal{P}_{\!f}+\Sigma_{\!f}}^2\!-\!\|x_e^{k}\|_{\mathcal{P}_{\!f}+\Sigma_{\!f}}^2
      +\big(1\!-\!\nu_{k+1}-5\mu_{k+1}\big)\|y_e^{k+1}\|^2_{\mathcal{T}_g}-\|y_e^{k}\|^2_{\mathcal{T}_{g}}\nonumber\\
 & +(\tau\sigma)^{-1}\big(\|z_e^{k+1}\|^2\!-\|z_e^k\|^2\big)
   \!+\nu_{k+1}\|\Delta y^{k+1}\|^2_{\mathcal{T}_{g}}\!-\nu_{k}\|\Delta y^{k}\|^2_{\mathcal{T}_{g}}
   \!-6\mu_{k+1}\sigma\|h(x^{k},y^{k})\|^2\nonumber\\
 & +\big(2\!-\!\tau-3\mu_{k+1}\big)\sigma\|h(x^{k+1},y^{k+1})\|^2  +\|\Delta y^{k+1}\|^2_{\mathcal{P}_{g}+\frac{3}{4}\Sigma_{g}}
    \!-\|\Delta y^k\|^2_{\mathcal{P}_{\!g}+\frac{3}{4}\Sigma_{g}}\nonumber\\
 \le&\ 2(1\!-\!\tau)\sigma\langle h(x^{k},y^{k}), \mathcal{B}^*\Delta y^{k+1}\rangle
      -\big(1\!-\nu_{k}\!-4\nu_{k+1}\!-4\mu_{k+1}\big)\|\Delta y^{k+1}\|^2_{\mathcal{T}_{g}}\nonumber\\
 &\ -\|\Delta x^{k+1}\|_{\mathcal{P}_{\!f}+\frac{1}{2}\Sigma_{\!f}-\mu_{k+1}(\mathcal{P}_{\!f}+\Sigma_{\!f})}^2
    \qquad{\rm for}\ k\ge 1.
 \end{align}
 \item[(c)] If the criterion (C2') is used for (S.1) and (S.2), then for any given $\gamma>0$ we have
 \begin{align}\label{result-C3}
  & \big(1\!-\!\gamma\mu^2_{k+1}\big)\|x_e^{k+1}\|^2_{\mathcal{P}_{\!f}+\Sigma_{\!f}}\!-\!\|x_e^{k}\|^2_{\mathcal{P}_{\!f}+\Sigma_{\!f}}
    +\big(1\!-\!\gamma\nu^2_{k+1}-3\gamma\mu_{k+1}^2\big)\|y_e^{k+1}\|^2_{\mathcal{T}_{g}}\!-\!\|y_e^{k}\|^2_{\mathcal{T}_{g}} \nonumber\\
  & +(\tau\sigma)^{-1}\big(\|z_e^{k+1}\|^2\!-\!\|z_e^k\|^2\big)
    \!+\big(2\!-\!\tau-2.5\gamma^{-1}-1.5\gamma\mu_{k+1}^2\big)\sigma\|h(x^{k+1},y^{k+1})\|^2 \nonumber\\
  & +\|\Delta y^{k+1}\|^2_{\mathcal{P}_{g}+\frac{3}{4}\Sigma_{g}}
    -\|\Delta y^k\|^2_{\mathcal{P}_{g}+\frac{3}{4}\Sigma_{g}}
    +\gamma^{-1}\|\Delta y^{k+1}\|^2_{\mathcal{T}_{g}}
    -\gamma^{-1}\|\Delta y^{k}\|^2_{\mathcal{T}_{g}}\nonumber\\
 \le&\ 2(1\!-\!\tau)\sigma\langle h(x^k,y^k), \mathcal{B}^*\Delta y^{k+1})\rangle-\|\Delta x^{k+1}\|_{\mathcal{P}_{\!f}
      +\frac{1}{2}\Sigma_{\!f}-\gamma^{-1}(\mathcal{P}_{\!f}+\Sigma_{\!f})}^2 \nonumber\\
 &\ \!-\!\big(1\!-\!6\gamma^{-1}\!-\!\gamma\nu^2_{k}-\gamma\nu^2_{k+1}\big)
     \|\Delta y^{k+1}\|_{\mathcal{T}_g}^2+4\gamma^{-1}\sigma\|h(x^{k},y^{k})\|^2\quad{\rm for}\ k\ge 1.
 \end{align}
 \end{description}
 \end{proposition}
 \begin{proof}
  (a) From inequality \eqref{cross-C1} and the result of Lemma \ref{Lemma1},
  it follows that
  \begin{align*}
   & \big(2\!-\!\tau\big)\sigma\|h(x^{k+1},y^{k+1})\|^2
      \!+\frac{1}{\tau\sigma}\big(\|z_e^{k+1}\|^2-\|z_e^k\|^2\big)
     \!+\!\big(1\!-\!\frac{\nu_{k+1}}{\gamma}\big)\|y_e^{k+1}\|^2_{\mathcal{T}_g}
   \!-\!\|y_e^k\|_{\mathcal{T}_g}^2\nonumber\\
  &       +\big(1\!-\!\frac{\mu_{k+1}}{\gamma}\big)\|x_e^{k+1}\|^2_{\mathcal{P}_{\!f}+\Sigma_{\!f}}
     \!-\!\|x_e^k\|_{\mathcal{P}_{\!f}+\Sigma_{\!f}}^2
     \!+\!\|\Delta y^{k+1}\|^2_{\mathcal{P}_{\!g}+\frac{3}{4}\Sigma_{g}}\!-\!\|\Delta y^k\|^2_{\mathcal{P}_{\!g}+\frac{3}{4}\Sigma_{g}}
     \nonumber\\
 &\le 2(1\!-\!\tau)\sigma\langle h(x^k,y^k),\mathcal{B}^*\Delta y^{k+1}\rangle
     \!-\!\|\Delta x^{k+1}\|^2_{\mathcal{P}_{\!f}+\frac{1}{2}\Sigma_{\!f}}
      \!-\!\big(1\!-\!\frac{\nu_{k+1}\!+\!\nu_{k}}{\gamma}\big)\|\Delta y^{k+1}\|^2_{\mathcal{T}_g}\nonumber\\
 &\ +\frac{\mu_{k+1}}{\gamma}\|x_e^{k+1}\|_{\sigma\mathcal{A}\mathcal{A}^*}^2
   \!+\!\gamma\|\mathcal{T}_{\!f}^{-1}\|\mu_{k+1}\!+\! \gamma\|\mathcal{T}_{\!g}^{-1}\|(2\nu_{k+1}\!+\nu_{k}).
 \end{align*}
 Since
 \(
   \frac{1}{2}\|x_e^{k+1}\|_{\sigma\mathcal{A}\mathcal{A}^*}^2 \leq \sigma\|h(x^{k+1},y^{k+1})\|^2
   \!+\!\|y_e^{k+1}\|_{\sigma\mathcal{B}\mathcal{B}^*}^2
   \le \sigma\|h(x^{k+1},y^{k+1})\|^2 \!+\!\|y_e^{k+1}\|_{\mathcal{T}_{\!g}}^2
 \)
 where the second inequality is due to $\sigma\mathcal{B}\mathcal{B}^*\preceq\mathcal{T}_{\!g}$,
 the last inequality implies part (a).

  \medskip
 \noindent
 (b) From inequality \eqref{cross-C2} and the result of Lemma \ref{Lemma1},
  it follows that
  \begin{align}\label{temp-equa27}
 &\big(2\!-\!\tau\big)\sigma\|h(x^{k+1},y^{k+1})\|^2
  +(\tau\sigma)^{-1}\big(\|z_e^{k+1}\|^2\!-\|z_e^k\|^2\big)
  +\big(1\!-\!\nu_{k+1}\big)\|y_e^{k+1}\|^2_{\mathcal{T}_g}
    \nonumber\\
 &   -\|y_e^{k}\|^2_{\mathcal{T}_{g}}
   +\big(1\!-\!\mu_{k+1}\big)\|x_e^{k+1}\|_{\mathcal{P}_f+\Sigma_{f}}^2\!-\!\|x_e^{k}\|_{\mathcal{P}_f+\Sigma_{f}}^2
   +\nu_{k+1}\|\Delta y^{k+1}\|^2_{\mathcal{T}_{g}}\nonumber\\
 & -\nu_{k}\|\Delta y^{k}\|^2_{\mathcal{T}_{g}}+\|\Delta y^{k+1}\|^2_{\mathcal{P}_{g}+\frac{3}{4}\Sigma_{g}}
    \!-\|\Delta y^k\|^2_{\mathcal{P}_{g}+\frac{3}{4}\Sigma_{g}}\nonumber\\
 &\le 2(1\!-\!\tau)\sigma\langle h(x^{k},y^{k}), \mathcal{B}^*(y^{k+1}\!-\!y^{k})\rangle
      -\big(1\!-\!\nu_{k}\!-\!4\nu_{k+1}\big)\|\Delta y^{k+1}\|^2_{\mathcal{T}_{g}}\nonumber\\
 &\quad +\mu_{k+1}\sigma\|x_e^{k+1}\|_{\mathcal{A}\mathcal{A}^*}^2 +\mu_{k+1}\|\Delta x^{k+1}\|_{\sigma\mathcal{A}\mathcal{A}^*}^2
       -\|\Delta x^{k+1}\|_{\mathcal{P}_{\!f}+\frac{1}{2}\Sigma_{\!f}-\mu_{k+1}(\mathcal{P}_{\!f}+\Sigma_{\!f})}^2.
 \end{align}
  For the terms $\|x_e^{k+1}\|_{\mathcal{A}\mathcal{A}^*}$ and $\|\Delta x^{k+1}\|_{\mathcal{A}\mathcal{A}^*}$,
  using equation \eqref{Cauchy} yields that
  \begin{align}\label{xeAA1}
  \|x_e^{k+1}\|_{\mathcal{A}\mathcal{A}^*}^2
   &=\|h(x^{k+1},y^{k+1})\|^2+\|y_e^{k+1}\|_{\mathcal{B}\mathcal{B}^*}^2-2\langle h(x^{k+1},y^{k+1}),\mathcal{B}^*y_e^{k+1}\rangle\nonumber\\
  &\le \|h(x^{k+1},y^{k+1})\|^2\!+\|y_e^{k+1}\|_{\mathcal{B}\mathcal{B}^*}^2
        \!+\frac{1}{4}\|h(x^{k+1},y^{k+1})\|^2\!+4\|y_e^{k+1}\|_{\mathcal{B}\mathcal{B}^*}^2,\\
  \|\Delta x^{k+1}\|_{\mathcal{A}\mathcal{A}^*}^2
   &=\|h(x^{k+1},y^{k+1})\|^2+\|h(x^{k},y^{k})\|^2  + \|\Delta y^{k+1}\|_{\mathcal{B}\mathcal{B}^*}^2 \nonumber\\
   &\quad -2\langle h(x^{k+1},y^{k+1}), h(x^{k},y^{k})\rangle
                 + 2\langle h(x^{k},y^{k})\!-h(x^{k+1},y^{k+1}), \mathcal{B}^*\Delta y^{k+1}\rangle \nonumber\\
   &\le \|h(x^{k+1},y^{k+1})\|^2+\|h(x^{k},y^{k})\|^2  + \|\Delta y^{k+1}\|_{\mathcal{B}\mathcal{B}^*}^2
                         +\frac{1}{4}\|h(x^{k+1},y^{k+1})\|^2\nonumber\\
   &\quad \!+4\|h(x^{k},y^{k})\|^2+\|h(x^{k},y^{k})\|^2 +\|\Delta y^{k+1}\|_{\mathcal{B}\mathcal{B}^*}^2\nonumber\\
   &\quad +\frac{1}{2}\|h(x^{k+1},y^{k+1})\|^2 + 2\|\Delta y^{k+1}\|_{\mathcal{B}\mathcal{B}^*}^2.
      \label{xeAA2}
  \end{align}
  Combining the last inequalities with \eqref{temp-equa27} and using
  $\sigma\mathcal{B}\mathcal{B}^*\preceq \mathcal{T}_g$ yields part (b).

  \medskip
  The proof of Part (c) is similar to that of part (b), we here omit it.
 \end{proof}

 \section{Convergence analysis of the IEIDP-ADMMs}\label{sec4}

   In this section we analyze the convergence of the IEIPD-ADMMs with the approximation criterion
  (C1) and (C2) respectively chosen for the minimization in (S.1) and (S.2).
  \subsection{Convergence of the IEIDP-ADMM with (C1)}\label{subsec4.1}

  For convenience, we write
  \(
     w^{k}:=(x_e^{k};y_e^{k};z_e^{k};\Delta y^{k};\Delta z^{k})
  \)
  for $k\ge 1$, and let $\mathcal{H}\!:\mathbb{X}\times\mathbb{Y}\times\mathbb{Z}\times\mathbb{Y}\times\mathbb{Z}
 \to \mathbb{X}\times\mathbb{Y}\times\mathbb{Z}\times\mathbb{Y}\times\mathbb{Z}$ denote
  the block diagonal operator defined by
  \[
    \mathcal{H}:={\rm Diag}\Big((\mathcal{P}_{\!f}\!+\!\Sigma_{\!f})^{1/2},\ (\mathcal{T}_g)^{1/2},\
    \frac{1}{\sqrt{\tau\sigma}}\mathcal{I},\ (\mathcal{P}_{\!g}\!+\!{\textstyle\frac{3}{4}}\Sigma_g)^{1/2},
    \ \frac{1}{\tau\sqrt{\sigma}}\mathcal{I}\Big)
  \]
  with the proximal operators $\mathcal{P}_{\!f}$ and $\mathcal{P}_{\!g}$ satisfying $\mathcal{P}_{\!f}+\Sigma_{\!f}\succeq 0$
  and $\mathcal{P}_{\!g}\!+\!\frac{3}{4}\Sigma_g\succeq 0$.
 \begin{lemma}\label{lemma31-C1}
  Let $\{(x^{k},y^{k},z^{k})\}_{k\ge 1}$ be the sequence generated by the IEIDP-ADMM with the criterion (C1) and
  $\max(\mu_{k},\nu_{k})\le \gamma\min(\frac{1}{6},\frac{2-\tau}{4})$ for some constant $\gamma>0$.
  Suppose that Assumption \ref{assump} holds and $\mathcal{P}_{\!f}$ and $\mathcal{P}_{g}$
  also satisfy $\mathcal{P}_{\!f}\!+\!\frac{1}{2}\Sigma_{\!f} \succeq 0$ and $\mathcal{P}_{\!g}\!+\!\frac{3}{8}\Sigma_{g}\succeq 0$.
  Then, when $\tau\in \big(0,2\big)$,
  there exists an absolute constant $c>0$ such that for all $k\ge 1$
  \begin{align*}
  &\|\mathcal{H}w^{k+1}\|^2_{\mathcal{W}_{k+1}}\le
   \Big[1\!+\frac{4\nu_k}{\gamma(2-\tau)}\Big]\Big[1\!+\frac{2(\nu_k\!+\!\nu_{k+1})}{\gamma}\Big]\|\mathcal{H}w^{k}\|^2_{\mathcal{W}_{k}}
     +c\gamma\big(\nu_k\!+\nu_{k+1}\!+\mu_{k+1}\big)\nonumber\\
  &\qquad\qquad\quad -\!\min\!\Big[\frac{2}{3}\min(\tau,1\!+\!\tau\!-\!\tau^2),\min(\tau,1\!+\!\tau\!-\!\tau^2)\Big]\Big(\frac{3\sigma}{2\tau}\|h(x^{k},y^{k})\|^2
      \!+\!\|\Delta y^{k+1}\|^2_{\mathcal{T}_g}\Big),
  \end{align*}
  where $\mathcal{W}_{k}\!:\mathbb{X}\times\mathbb{Y}\times\mathbb{Z}\times\mathbb{Y}\times\mathbb{Z}
 \to \mathbb{X}\times\mathbb{Y}\times\mathbb{Z}\times\mathbb{Y}\times\mathbb{Z}$ is the block diagonal linear operator
  \[
   \mathcal{W}_k:={\rm Diag}\Big(\big(1\!-\!\gamma^{-1}\mu_k\big)\mathcal{I},\,\big(1\!-\!2\gamma^{-1}\mu_k\!-\!\gamma^{-1}\nu_k\big)\mathcal{I},\,
                   \mathcal{I},\,\mathcal{I},\, (2\!-\!\tau\!-\!2\gamma^{-1}\mu_k)\mathcal{I}\Big).
  \]
 \end{lemma}
 \begin{proof}
 For each $k\ge 1$, let $\mathcal{V}_{k}\!:\mathbb{X}\times\mathbb{Y}\times\mathbb{Z}\times\mathbb{Y}\times\mathbb{Z}
 \to \mathbb{X}\times\mathbb{Y}\times\mathbb{Z}\times\mathbb{Y}\times\mathbb{Z}$ be defined by
 \[
   \mathcal{V}_k={\rm Diag}\Big(\mathcal{I},\ \mathcal{I},\ \mathcal{I},\ \mathcal{I},\ \frac{2\!-\!\tau}{1-\gamma^{-1}\alpha_k}\mathcal{I}\Big)
   \quad {\rm with}\ \alpha_k=\nu_{k}+\nu_{k-1}.
 \]
 With the notations $\mathcal{W}_k$ and $\mathcal{V}_k$, we first establish the following important inequality:
 \begin{align}\label{aim-ineq1}
  &\|\mathcal{H}w^{k+1}\|^2_{\mathcal{W}_{k+1}}\!-\!\|\mathcal{H}w^{k}\|^2_{\mathcal{V}_{k+1}}\nonumber\\
  &\le -\min(\tau,1+\tau\!-\!\tau^2)\left[\frac{\tau^{-1}\sigma}{1\!-\!\gamma^{-1}\alpha_{k+1}}\| h(x^{k},y^{k})\|^2
       +\Big(1\!-\!\frac{\alpha_{k+1}}{\gamma}\Big)\|\Delta y^{k+1}\|_{\mathcal{T}_g}^2\right]\nonumber\\
  &\quad\ +\max(\|\mathcal{T}_{f}^{-1}\|,2\|\mathcal{T}_{g}^{-1}\|)\gamma\big(\alpha_{k+1}\!+\mu_{k+1}\big).
 \end{align}
  Indeed, when $\tau\in (0,1]$, since $\gamma^{-1}\alpha_{k+1}=\gamma^{-1}(\nu_{k+1}+\nu_{k})<1$,
  by equation \eqref{Cauchy} we have
  \begin{align*}
   2(1\!-\!\tau)\sigma\langle h(x^{k},y^{k}), \mathcal{B}^*\Delta y^{k+1}\rangle
   \le \frac{\sigma(1\!-\!\tau)}{1-\!\gamma^{-1}\alpha_{k+1}}\| h(x^{k},y^{k})\|^2
       +\Big(1-\!\frac{\alpha_{k+1}}{\gamma}\Big)\|\Delta y^{k+1}\|^2_{\sigma(1-\tau)\mathcal{B}\mathcal{B}^*}.
 \end{align*}
  Substituting this inequality into Proposition \ref{prop21}(a) and using \eqref{zk}, we obtain that
  \begin{align}\label{temp-equa30}
   & \Big(1\!-\!\frac{\mu_{k+1}}{\gamma}\Big)\|x_e^{k+1}\|^2_{\mathcal{P}_{\!f}+\Sigma_{\!f}}
     \!-\!\|x_e^k\|_{\mathcal{P}_{\!f}+\Sigma_{\!f}}^2
     \!+\!\Big(1\!-\!\frac{\nu_{k+1}}{\gamma}\!-\!\frac{2\mu_{k+1}}{\gamma}\Big)\|y_e^{k+1}\|^2_{\mathcal{T}_g}
   \!-\!\|y_e^k\|_{\mathcal{T}_g}^2\nonumber\\
  &  \!+\frac{1}{\tau\sigma}\big(\|z_e^{k+1}\|^2-\|z_e^k\|^2\big)
     \!+\!\|\Delta y^{k+1}\|^2_{\mathcal{P}_{g}+\frac{3}{4}\Sigma_{g}}
     \!-\!\|\Delta y^k\|^2_{\mathcal{P}_{g}+\frac{3}{4}\Sigma_{g}}\nonumber\\
  &  + \Big(2\!-\!\tau-\!\frac{2\mu_{k+1}}{\gamma}\Big)\frac{1}{\tau^2\sigma}\|\Delta z^{k+1}\|^2
      -\frac{(2\!-\!\tau)\sigma}{1-\gamma^{-1}\alpha_{k+1}}\cdot\frac{1}{\tau^2\sigma}\|\Delta z^{k}\|^2\nonumber\\
 \le&\ \frac{-\sigma}{1-\gamma^{-1}\alpha_{k+1}}\| h(x^{k},y^{k})\|^2
        \!-\Big(1\!-\!\frac{\alpha_{k+1}}{\gamma}\Big)\|\Delta y^{k+1}\|^2_{\mathcal{T}_g-(1\!-\!\tau)\sigma\mathcal{B}\mathcal{B}^*} \nonumber\\
  &\ -\!\|\Delta x^{k+1}\|^2_{\mathcal{P}_{\!f}+\frac{1}{2}\Sigma_{\!f}}
       \!+\!\gamma\|\mathcal{T}_{f}^{-1}\|\mu_{k+1}\!+\!\gamma\|\mathcal{T}_{g}^{-1}\|(2\nu_{k+1}+\nu_{k})\nonumber\\
 \le&\ \frac{-\sigma}{1-\gamma^{-1}\alpha_{k+1}}\| h(x^{k},y^{k})\|^2
        \!-\tau\Big(1\!-\!\frac{\alpha_{k+1}}{\gamma}\Big)\|\Delta y^{k+1}\|^2_{\mathcal{T}_g}\nonumber\\
 &\ +\max(\|\mathcal{T}_{f}^{-1}\|,2\|\mathcal{T}_{g}^{-1}\|)\gamma\big(\alpha_{k+1}\!+\mu_{k+1}\big),
 \end{align}
  where the second inequality is using
  $\mathcal{T}_g-(1\!-\!\tau)\sigma\mathcal{B}\mathcal{B}^*\succeq \tau\mathcal{T}_g$
  and $\mathcal{P}_{\!f}+\frac{1}{2}\Sigma_{\!f}\succeq 0$.
  For the case where $\tau\in(1,2)$, from $\gamma^{-1}\alpha_{k+1}=\gamma^{-1}(\nu_{k+1}+\nu_{k})<1$ and equation \eqref{Cauchy}
  we have
  \begin{align*}
  & 2(1\!-\!\tau)\sigma\langle h(x^{k},y^{k}), \mathcal{B}^*\Delta y^{k+1}\rangle \nonumber\\
  \le&\ \frac{\sigma(\tau\!-\!1)}{\tau(1-\gamma^{-1}\alpha_{k+1})}\| h(x^{k},y^{k})\|^2
      +\tau\big(1\!-\!\frac{\alpha_{k+1}}{\gamma}\big)\|\Delta y^{k+1}\|^2_{\sigma(\tau\!-\!1)\mathcal{B}\mathcal{B}^*}.
 \end{align*}
 Substituting this inequality into Proposition \ref{prop21}(a) and using \eqref{zk} yields that
 \begin{align}\label{temp-equa31}
   &\big(1-\frac{\mu_{k+1}}{\gamma}\big)\|x_e^{k+1}\|^2_{\mathcal{P}_{\!f}+\Sigma_{\!f}}
     -\|x_e^k\|_{\mathcal{P}_{\!f}+\Sigma_{\!f}}^2
         +\big(1\!-\!\frac{\nu_{k+1}}{\gamma}\!-\!\frac{2\mu_{k+1}}{\gamma}\big)\|y_e^{k+1}\|^2_{\mathcal{T}_g}
        -\|y_e^k\|_{\mathcal{T}_g}^2\nonumber\\
   &\ + (\tau\sigma)^{-1}\big(\|z_e^{k+1}\|^2 - \|z_e^k\|^2\big)
      +\big(\|\Delta y^{k+1}\|^2_{\mathcal{P}_{g}+\frac{3}{4}\Sigma_{g}}
        \!-\!\|\Delta y^k\|^2_{\mathcal{P}_{g}+\frac{3}{4}\Sigma_{g}}\big)\nonumber\\
   &\ +\big(2\!-\!\tau-\frac{2\mu_{k+1}}{\gamma}\big)\frac{1}{\tau^2\sigma}\|\Delta z^{k+1}\|^2
       -\frac{2-\tau}{1\!-\!\gamma^{-1}\alpha_{k+1}}\cdot\frac{1}{\tau^2\sigma}\|h(x^{k},y^{k})\|^2 \nonumber\\
  \le&\ \frac{-(1\!+\!\tau^{-1}\!-\!\tau)}{1\!-\!\gamma^{-1}\alpha_{k+1}}\sigma\|h(x^k,y^{k})\|^2
     -\big(1\!-\!\frac{\alpha_{k+1}}{\gamma}\big)\|\Delta y^{k+1}\|^2_{\mathcal{T}_g-(\tau\!-\!1)\tau\sigma\mathcal{B}\mathcal{B}^*}\nonumber\\
  &\quad -\|\Delta x^{k+1}\|^2_{\mathcal{P}_{\!f}+\frac{1}{2}\Sigma_{\!f}}
        \!+\!\gamma\|\mathcal{T}_{f}^{-1}\|\mu_{k+1}\!+\! \gamma\|\mathcal{T}_{g}^{-1}\|(2\nu_{k+1}+\nu_{k})\nonumber\\
  \le&\ -\frac{(1\!+\!\tau^{-1}\!-\!\tau)\sigma}{1\!-\!\gamma^{-1}\alpha_{k+1}}\|h(x^k,y^{k})\|^2
     -(1+\tau-\tau^2)\big(1\!-\!\frac{\alpha_{k+1}}{\gamma}\big)\|\Delta y^{k+1}\|^2_{\mathcal{T}_g}\nonumber\\
  &\ +\max(\|\mathcal{T}_{f}^{-1}\|,2\|\mathcal{T}_{g}^{-1}\|)\gamma\big(\alpha_{k+1}\!+\mu_{k+1}\big),
 \end{align}
  where the second inequality is using $\mathcal{P}_{\!f}+\frac{1}{2}\Sigma_{\!f}\succeq 0$ and
 \(
  \mathcal{T}_g-(\tau\!-\!1)\tau\sigma\mathcal{B}\mathcal{B}^*
   \succeq (1\!+\tau\!-\tau^2)\mathcal{T}_g.
 \)
  Notice that $\mathcal{W}_k\succ 0$ and $\mathcal{V}_k\succ 0$ for all $k\ge 1$
  due to $\max(\mu_{k},\nu_{k})\le \gamma\min(\frac{1}{6},\frac{2-\tau}{4})$.
  From the definition of $w^{k+1}$ and the expressions of $\mathcal{H}, \mathcal{W}_{k}$ and
  $\mathcal{V}_k$, the left hand side of \eqref{temp-equa30} and \eqref{temp-equa31}
  equals $\|\mathcal{H}w^{k+1}\|^2_{\mathcal{W}_{k+1}}\!-\|\mathcal{H}w^{k}\|^2_{\mathcal{V}_{k+1}}$.
  Along with \eqref{temp-equa30} and \eqref{temp-equa31}, we get \eqref{aim-ineq1}.

  \medskip

 Now by the assumption $\max(\mu_{k},\nu_{k})\le\gamma\min(\frac{1}{6},\frac{2-\tau}{4})$,
 it is not difficult to verify that
 \[
    \frac{1}{1\!-\!\frac{\nu_k}{\gamma}\!-\!\frac{2\mu_k}{\gamma}}\le 1\!+\frac{2(\nu_{k}\!+\!2\mu_k)}{\gamma}
    \ {\rm and}\
    \frac{1}{(1\!-\!\frac{\alpha_{k+1}}{\gamma})(1\!-\!\frac{2}{2-\tau}\frac{\nu_{k}}{\gamma})}
   \le \big(1\!+\frac{2\alpha_{k+1}}{\gamma}\big)\Big[1\!+\frac{4\nu_{k}}{\gamma(2-\tau)}\Big].
 \]
 By the expressions of $\mathcal{V}_{k+1}$ and $\mathcal{W}_{k}$, we have
 \(
   \mathcal{V}_{k+1}\mathcal{W}_{k}^{-1} \preceq \big(1\!+\frac{2\alpha_{k+1}}{\gamma}\big)\Big[1\!+\frac{4\nu_{k}}{\gamma(2-\tau)}\Big]\mathcal{I}.
 \)
 Then
  \begin{equation}\label{HGk}
  \|\mathcal{H}w^{k}\|^2_{\mathcal{V}_{k+1}}
  \!=\langle \mathcal{H}w^{k},\mathcal{V}_{k+1}\mathcal{W}_{k}^{-1}\mathcal{W}_{k}\mathcal{H}w^{k}\rangle
   \le \big(1\!+\frac{2\alpha_{k+1}}{\gamma}\big)\Big[1\!+\frac{4\nu_{k}}{\gamma(2-\tau)}\Big]\|\mathcal{H}w^{k}\|^2_{\mathcal{W}_{k}}.
 \end{equation}
  In addition, since $\max(\mu_{k},\nu_{k})\le\gamma\min(\frac{1}{6},\frac{2-\tau}{4})$ for all $k\ge 1$,
  we have $\frac{2}{3}\le 1\!-\!\frac{\alpha_{k+1}}{\gamma}\le 1$.
  Let $c=\max(\|\mathcal{T}_{g}^{-1}\|,2\|\mathcal{T}_{f}^{-1}\|)$.
  Combining \eqref{HGk} with \eqref{aim-ineq1} yields the desired result.
 \end{proof}

 \medskip

 By Lemma \ref{lemma31-C1}, we may establish the convergence of the IEIDP-ADMM with (C1).
 \begin{theorem}\label{theorem-C1}
  Let $\{(x^{k},y^{k},z^{k})\}_{k\ge 1}$ be the sequence generated by the IEIDP-ADMM with
  (C1) and $\max(\mu_{k},\nu_{k})\le\gamma\min(\frac{1}{6},\frac{2-\tau}{4})$,
  where $\gamma$ is same as that of Lemma \ref{lemma31-C1}.
  Suppose that Assumption \ref{assump} holds and $\mathcal{P}_{\!f}$ and $\mathcal{P}_{\!g}$
  also satisfy $\mathcal{P}_{\!f}+\frac{1}{2}\Sigma_{\!f} \succeq 0$ and
  $\mathcal{P}_{\!g}+\frac{3}{8}\Sigma_{g} \succeq 0$. Then, for
  (a) $\tau\in \big(0,\frac{1+\sqrt{5}}{2}\big)$ or (b) $\tau\in[\frac{1+\sqrt{5}}{2},2)$ but
  $\sum_{k=0}^{\infty}\big(\frac{3\sigma}{2\tau}\|h(x^{k},y^{k})\|^2
      \!+\!\|\Delta y^{k+1}\|^2_{\mathcal{T}_g}\big)<\infty$,
  the sequence $\{(x^{k},y^{k})\}$ converges to an optimal solution of \eqref{prob} and
  the sequence $\{z^{k}\}$ converges to an optimal solution to the dual problem of \eqref{prob}.
 \end{theorem}
 \begin{proof}
  We write $\vartheta_k=[1\!+\frac{4\nu_k}{\gamma(2-\tau)}][1\!+\frac{2(\nu_k+\nu_{k+1})}{\gamma}]$,
  $\varpi_k=c\gamma(\nu_k\!+\nu_{k+1}\!+\mu_{k+1})$ and
  $R_k\!=\frac{3\sigma}{2\tau}\|h(x^{k},y^{k})\|^2\!+\!\|\Delta y^{k+1}\|^2_{\mathcal{T}_g}$
  for $k\ge 1$. By Lemma \ref{lemma31-C1} we have that
  \begin{align*}
  \|\mathcal{H}w^{k+1}\|^2_{\mathcal{W}_{k+1}}\le
   & -\!\min\!\Big[\frac{2}{3}\min(\tau,1\!+\!\tau\!-\!\tau^2),\min(\tau,1\!+\!\tau\!-\!\tau^2)\Big]\Big(\sum_{l=0}^{k-1}\prod_{j=l+1}^{k}\vartheta_jR_l+R_k\Big)\\
   &\ \ +\prod_{j=0}^{k}\vartheta_j\|\mathcal{H}w^{0}\|^2_{\mathcal{W}_{0}}
    +\sum_{l=0}^{k-1}\prod_{j=l+1}^{k}\vartheta_j\varpi_{l}+\varpi_k.
  \end{align*}
  Since $\sum_{k=1}^{\infty}\nu_k<\!\infty$, we have
  \(
     1\le\prod_{k =1}^{\infty}\vartheta_k\le K_1
  \)
  for some $K_1\!\ge 1$. Hence, we have that
  \[
   \left\{\begin{array}{cl}
   \!\|\mathcal{H}w^{k+1}\|^2_{\mathcal{W}_{k+1}}\!\le
    K_1\big(\|\mathcal{H}w^{0}\|^2_{\mathcal{W}_{0}}\!+\!\sum_{l=0}^{k}\varpi_l\big)\!-\frac{2}{3}\min(\tau,1\!+\!\tau\!-\!\tau^2)\sum_{l=0}^{k}R^l
    & {\rm if}\ \tau\in(0,\frac{1+\sqrt{5}}{2}),\\
   \|\mathcal{H}w^{k+1}\|^2_{\mathcal{W}_{k+1}}\le
   K_1\big(\|\mathcal{H}w^{0}\|^2_{\mathcal{W}_{0}}+\!\sum_{l=0}^{k}\varpi_l\big)+K_1(\tau^2\!-\!1\!-\!\tau)\sum_{l=0}^{k}R^l
   &{\rm if}\ \tau\in[\frac{1+\sqrt{5}}{2},2).
   \end{array}\right.
  \]
  Notice that $\sum_{k=1}^{\infty}\max(\mu_k,\nu_{k}) < \infty$ implies
  $\sum_{l=0}^{\infty} \varpi_l\leq K_2$ for some $K_2\ge 0$. Then,
 \[
    \left\{\begin{array}{cl}
    \!\|\mathcal{H}w^{k+1}\|^2_{\mathcal{W}_{k+1}}\!+\frac{2}{3}\min(\tau,1\!+\!\tau\!-\!\tau^2)\sum_{l=0}^{k}R^l\le
   K_1\big(\|\mathcal{H}w^{0}\|^2_{\mathcal{W}_{0}}+K_2\big) & {\rm if}\ \tau\in(0,\frac{1+\sqrt{5}}{2}),\\
   \|\mathcal{H}w^{k+1}\|^2_{\mathcal{W}_{k+1}}\le K_1\big(\|\mathcal{H}w^{0}\|^2_{\mathcal{W}_{0}}+K_2\big)
    +K_1(\tau^2\!-\!1\!-\!\tau)\sum_{l=0}^{k}R^l& {\rm if}\ \tau\in[\frac{1+\sqrt{5}}{2},2).
    \end{array}\right.
  \]
  Notice that $\max(\mu_{k},\nu_k)\le \frac{\gamma}{2\max(3,\frac{2}{2-\tau})}$ implies
  $\mathcal{W}_{k}\succeq\overline{\mathcal{W}}$ for some $\overline{\mathcal{W}}\succ0$.
  Then, under conditions (a) and (b), the last two inequalities imply that
  the sequence $\{\mathcal{H}w^{k}\}$ is bounded and $\sum_{l=0}^{\infty}R_l<+\infty$.
  The latter implies that $\lim_{k\to\infty}R_k=0$, and consequently,
  \begin{align}\label{yzk}
   \lim_{k \rightarrow \infty}\|\Delta y^{k+1}\|_{\mathcal{T}_g} = 0\ \ {\rm and}\ \
   \lim_{k \rightarrow \infty}\|\Delta z^{k}\| = 0.
  \end{align}
  By equality \eqref{zk}, the limits in \eqref{yzk} imply that the sequence
  $\{\|x_e^{k+1}\|_{\mathcal{A}\mathcal{A}^*}\}$ is bounded.
  By the definition of $\mathcal{H}$ and the boundedness of $\{\mathcal{H}w^{k}\}$,
  the sequence $\{\|x_e^{k+1}\|_{\mathcal{P}_{\!f}+\Sigma_{\!f}}\}$ is also bounded.
  Thus, the sequence $\{\|x_e^{k+1}\|_{\mathcal{T}_{\!f}}\}$ is bounded. In addition,
  the boundedness of $\{\mathcal{H}w^{k}\}$ also implies that the sequences $\{\|y_e^{k+1}\|_{\mathcal{T}_{\!g}}\}$
  and $\{\|z_e^{k+1}\|\}$ are bounded. Together with the positive definiteness
  of $\mathcal{T}_{\!f}$ and $\mathcal{T}_{\!g}$, it follows that $\{(x^{k},y^{k},z^{k})\}$ is bounded.
  So, there exists a convergent subsequence, to say $\{(x^{k},y^k,z^{k})\}_{\mathcal{K}}$.
  Without loss of generality, we assume $\{(x^{k},y^k,z^{k})\}_{\mathcal{K}}\to (x^{\infty},y^{\infty},z^{\infty})$.
  Since $\lim_{k \rightarrow \infty}\|z^{k}\!-\!z^{k-1}\| = 0$, we have
  $h(x^{\infty},y^{\infty})=0$. In addition, taking the limit $k\to\infty$ with $k\in\mathcal{K}$
  on the both sides of \eqref{opt-x} and \eqref{opt-y}, and using the closedness of the graphs
  of $\partial f$ and $\partial g$ (see \cite{Roc70}), we have
  \[
    -\!\mathcal{A}^*z^{\infty}\in\partial f(x^{\infty})\ \ {\rm and}\ \
    -\!\mathcal{B}^*z^{\infty}\in\partial g(y^{\infty}).
  \]
  Along with \eqref{optimal-cond}, $(x^{\infty},y^{\infty})$ is an optimal solution of \eqref{prob}
  and $z^{\infty}$ is the associated multiplier.

 \medskip

  Finally, we argue that $(x^{\infty},y^{\infty},z^{\infty})$  is actually the unique limit point
  of $\{(x^k,y^k,z^k)\}$. Recall that $(x^{\infty},y^{\infty})$ is an optimal solution to \eqref{prob}
  and $z^{\infty}$ is the associated multiplier. Hence, we could replace $(x^{*},y^{*},z^*)$
  with  $(x^{\infty},y^{\infty},z^{\infty})$ in the previous arguments,
  starting from \eqref{opt-x} and \eqref{opt-y}. Thus, inequality \eqref{aim-ineq1}
  still holds with  $(x^{*},y^{*},z^{*})$ replaced by $(x^{\infty},y^{\infty},z^{\infty})$.
  Hence, from the definition of $w^k$, $\sum_{k=1}^{\infty}\max(\mu_k,\nu_k)<\infty$ and equation \eqref{yzk},
  we have $\lim_{k\to\infty,k\in\mathcal{K}}\|\mathcal{H}w^k\|_{\mathcal{W}_k}
  \le \lim_{k\to\infty,k\in\mathcal{K}}\|\mathcal{H}w^k\|_{\widetilde{\mathcal{W}}}=0$ where
  \[
    \widetilde{\mathcal{W}}={\rm Diag}\Big(\mathcal{I},\ \mathcal{I},\ \mathcal{I},\ \mathcal{I},\ (2-\tau)\mathcal{I}\Big).
  \]
  This means that for any $\varepsilon >0$, there exists a sufficiently large $k_{0}\in \mathcal{K}$ such that
  $\|\mathcal{H}w^{k_0}\|_{\mathcal{W}_{k_0}} < \frac{\varepsilon}{2K_1}$ and
  $\sum_{l=k_{0}}^{\infty}\varpi_l < \frac{\varepsilon}{2K_1}$.
  By Lemma \ref{lemma31-C1}, for any $k\geq k_0$ we have
  \begin{align*}
    \|\mathcal{H}w^{k+1}\|^2_{\overline{\mathcal{W}}}
  \le\|\mathcal{H}w^{k+1}\|^2_{\mathcal{W}_{k+1}}
  &\le\prod_{j=k_0}^{k}\vartheta_j\|\mathcal{H}w^{k_0}\|^2_{\mathcal{W}_{k_0}}
    +\sum_{l=k_0}^{k-1}\prod_{j=k_0+1}^{k}\vartheta_j\varpi_l+\varpi_k\nonumber\\
  &\le  K_1\|\mathcal{H}w^{k_0}\|^2_{\mathcal{W}_{k_0}}+  K_1\sum_{l=k_0}^{k}\varpi_l
  \leq \varepsilon.
  \end{align*}
  This, by the positive definiteness of $\overline{\mathcal{W}}$, shows that
  $\lim_{k\to\infty}\mathcal{H}w^{k+1}= 0$. Consequently,
  \[
    \lim_{k\to\infty}\|x^{k+1}\!-\!x^{\infty}\|_{\mathcal{P}_{\!f}+\Sigma_{f}} = 0,
    \lim_{k\to\infty}\|y^{k+1}\!-\!y^{\infty}\|_{\mathcal{T}_{g}} = 0
    \ {\rm and}\ \lim_{k\to\infty}\|z^{k+1}\!-\!z^{\infty}\| = 0.
  \]
  Combining $\lim_{k\to\infty}\|y^{k+1}\!-\!y^{\infty}\|_{\mathcal{T}_{g}} = 0$
  with $\mathcal{T}_{g}\succ 0$ yields $\lim_{k\to\infty}\|y^{k+1}\!-\!y^{\infty}\|= 0$.
  In addition, the second limit in \eqref{yzk} implies $\lim_{k\to\infty}h(x^k,y^k)=0$.
  Together with
  \[
   \|x^{k+1}\!-\!x^{\infty}\|^2_{\mathcal{A}\mathcal{A}^*}
  \le 2\|y^{k+1}\!-\!y^{\infty}\|^2_{\mathcal{B}\mathcal{B}^*}+2\|h(x^{k+1},y^{k+1})\|^2,
  \]
  we obtain $\lim_{k\to\infty}\|x^{k+1}\!-\!x^{\infty}\|^2_{\mathcal{A}\mathcal{A}^*}= 0$.
  Noting that $\lim_{k\to\infty}\|x^{k+1}\!-\!x^{\infty}\|_{\mathcal{P}_{\!f}+\Sigma_{f}}=0$,
  we have $\lim_{k\to\infty}\|x^{k+1}\!-\!x^{\infty}\|^2_{\mathcal{T}_f}= 0$.
  By the positive definiteness of $\mathcal{T}_{f}$, it follows that $\lim_{k\to\infty}\|x^{k}\!-\!x^{\infty}\|=0$.
  Thus, $\lim_{k\to\infty}x^{k} = x^{\infty},\lim_{k\to\infty}y^{k} = y^{\infty}$
  and $\lim_{k\to\infty}z^{k} = z^{\infty}$. That is, $(x^{\infty},y^{\infty},z^{\infty})$ is
  the unique limit point of $\{(x^k,y^k,z^k)\}$.
 \end{proof}

 \begin{remark}
  Theorem \ref{theorem-C1} shows that one can establish the convergence of $\{(x^{k},y^{k},z^{k})\}$
  generated by the IEIDP-ADMM with (C1) if $\mathcal{P}_{\!f}$ and $\mathcal{P}_{\!g}$ are chosen such that
 \begin{align*}
  \mathcal{P}_{\!f}+\frac{1}{2}\Sigma_{f}\succeq 0,\ \mathcal{P}_{g}+\frac{3}{8}\Sigma_{g}\succeq 0,\
  \Sigma_{f}+\mathcal{P}_{\!f}+\sigma\mathcal{A}\mathcal{A}^* \succ 0,\ \Sigma_{g}+\mathcal{P}_{g}+\sigma\mathcal{B}\mathcal{B}^* \succ 0.
  \end{align*}
  In fact, using the same arguments, one can get the convergence of $\{(x^{k},y^{k},z^{k})\}$
  generated by the IEIDP-ADMM with (C1) if $\mathcal{P}_f$ and $\mathcal{P}_g$ are chosen such that
  for some $a\in [\frac{1}{2},1)$,
 \begin{align*}
   \mathcal{P}_{\!f}\!+\!\frac{1}{2}\Sigma_{f}\succeq 0,\ \mathcal{P}_{g}\!+\!\frac{1}{2}\Sigma_{g}\succeq 0,\
  \Sigma_{f}\!+\!\mathcal{P}_{\!f}\!+\!\sigma\mathcal{A}\mathcal{A}^* \succ 0,\ a\Sigma_{g}\!+\!\mathcal{P}_{g}\!+\!\sigma\mathcal{B}\mathcal{B}^* \succ 0.
 \end{align*}
 \end{remark}

 \subsection{Convergence of the IEIDP-ADMM with (C2)}\label{subsec4.2}

  For each $k\ge 1$, we write $w^{k}:=(x_e^{k};y_e^{k};z_e^{k};\Delta y^{k};\Delta y^{k};\Delta z^{k})$,
  and let $\mathcal{H}\!:\mathbb{X}\times\mathbb{Y}\times\mathbb{Z}\times\mathbb{Y}\times\mathbb{Y}
  \times\mathbb{Z}\to \mathbb{X}\times\mathbb{Y}\times\mathbb{Z}\times\mathbb{Y}\times\mathbb{Y}\times\mathbb{Z}$
  be the block diagonal linear operator defined by
  \[
    \mathcal{H}:={\rm Diag}\Big((\mathcal{P}_{\!f}\!+\!\Sigma_{\!f})^{1/2},\,(\mathcal{T}_g)^{1/2},\,
    \frac{1}{\sqrt{\tau\sigma}}\mathcal{I},\,(\mathcal{P}_{\!g}\!+\!{\textstyle\frac{3}{4}}\Sigma_g)^{1/2},
    \,(\mathcal{T}_g)^{1/2},\,\frac{1}{\tau\sqrt{\sigma}}\mathcal{I}\Big)
  \]
  for the proximal operators $\mathcal{P}_{\!f}$ and $\mathcal{P}_{\!g}$ satisfying $\mathcal{P}_{\!f}+\Sigma_{\!f}\succeq 0$
  and $\mathcal{P}_{\!g}\!+\!\frac{3}{4}\Sigma_{\!g}\succeq 0$.
  To establish the convergence of the IEIDP-ADMM with (C2), we need the following lemma.
 \begin{lemma}\label{lemma31-C2}
  Let $\{(x^{k},y^{k},z^{k})\}_{k\ge 1}$ be the sequence given by the IEIDP-ADMM with
  (C2) and $\max(\mu_k,\nu_{k})\le \min(0.1,\frac{2-\tau}{4})$. Suppose that Assumption \ref{assump}
  holds and $\mathcal{P}_{\!f}$ and $\mathcal{P}_{\!g}$ also satisfy
  $\mathcal{P}_{\!f}+\frac{3}{8}\Sigma_{\!f} \succeq 0$ and
  $\mathcal{P}_{\!g}+\frac{3}{8}\Sigma_{\!g}\succeq 0$.
  Then, when $\tau\in (0,2)$, for all $k\ge 1$ we have
  \begin{align*}
  &\|\mathcal{H}w^{k+1}\|^2_{\mathcal{W}_{k+1}}\le
  \big[1\!+10\nu_k\!+40(\mu_{k+1}\!+\!\nu_{k+1})\big]\Big[1\!+\frac{12(\mu_k+2\mu_{k+1})}{2-\tau}\Big]\|\mathcal{H}w^{k}\|^2_{\mathcal{W}_{k}}\nonumber\\
  &\qquad\qquad -\!\min\!\big(0.1\min(\tau,1\!+\tau\!-\tau^2),\min(\tau,1\!+\tau\!-\tau^2)\big)
     \Big(\frac{10\sigma}{\tau}\|h(x^{k},y^{k})\|^2+\|\Delta y^{k+1}\|^2_{\mathcal{T}_g}\Big)
  \end{align*}
  where the operator $\mathcal{W}_{k}\!:\mathbb{X}\times\mathbb{Y}\times\mathbb{Z}
  \times\mathbb{Y}\times\mathbb{Y}\times\mathbb{Z}\to \mathbb{X}\times\mathbb{Y}\times\mathbb{Z}
  \times\mathbb{Y}\times\mathbb{Y}\times\mathbb{Z}$ is defined by
  \[
   \mathcal{W}_k:={\rm Diag}\big((1\!-\mu_k)\mathcal{I},\,(1\!-5\mu_{k}\!-\nu_k)\mathcal{I},\
                   \mathcal{I},\ \mathcal{I},\ \nu_k\mathcal{I},\ (2\!-\!\tau-\!3\mu_k)\mathcal{I}\big).
  \]
 \end{lemma}
 \begin{proof}
  Let $\mathcal{V}_{k}\!:\mathbb{X}\times\mathbb{Y}\times\mathbb{Z}\times\mathbb{Y}\times\mathbb{Y}
  \times\mathbb{Z}\to \mathbb{X}\times\mathbb{Y}\times\mathbb{Z}\times\mathbb{Y}\times\mathbb{Y}\times\mathbb{Z}$
  for $k\ge 1$ be defined by
  \[
   \mathcal{V}_k={\rm Diag}\Big(\mathcal{I},\ \mathcal{I},\ \mathcal{I},\ \mathcal{I},\ \nu_{k-1}\mathcal{I},\ \big(\frac{2\!-\!\tau}{1\!-\alpha_{k}}+6\mu_{k}\big)\mathcal{I}\Big)
   \ \ {\rm with}\ \alpha_{k}=\nu_{k-1}+4(\nu_k+\mu_k).
  \]
  With the notations $\mathcal{W}_k$ and $\mathcal{V}_k$, we first establish the following important inequality:
  \begin{align}\label{aim-ineq2}
  &\|\mathcal{H}w^{k+1}\|^2_{\mathcal{W}_{k+1}}\!-\!\|\mathcal{H}w^{k}\|^2_{\mathcal{V}_{k+1}}\nonumber\\
  &\le -\min(\tau,1+\tau-\!\tau^2)\left[\frac{\sigma}{\tau(1\!-\!\alpha_{k+1})}\|h(x^{k},y^{k})\|^2
       +\big(1\!-\!\alpha_{k+1}\big)\|\Delta y^{k+1}\|^2_{\mathcal{T}_{g}}\right].
  \end{align}
  Indeed, when $\tau\in (0,1]$, since $\alpha_{k+1}\le0.9$ by $\max(\mu_k,\nu_{k})\le 0.1$,
  it follows from \eqref{Cauchy} that
  \begin{align*}
   2(1\!-\!\tau)\sigma\langle h(x^{k},y^{k}), \mathcal{B}^*\Delta y^{k+1}\rangle
   &\le \frac{\sigma(1\!-\!\tau)}{1\!-\!\alpha_{k+1}}\| h(x^{k},y^{k})\|^2 + (1\!-\!\alpha_{k+1})\|\Delta y^{k+1}\|^2_{(1\!-\!\tau)\sigma\mathcal{B}\mathcal{B}^*}.
 \end{align*}
  Substituting the last inequality into Proposition \ref{prop21}(b) and using \eqref{zk}, we obtain that
  \begin{align}\label{temp-equaC21}
  &\big(1\!-\!\mu_{k+1}\big)\|x_e^{k+1}\|_{\mathcal{P}_f+\Sigma_{f}}^2\!-\!\|x_e^{k}\|_{\mathcal{P}_f+\Sigma_{f}}^2
      +\big(1\!-\!\nu_{k+1}-5\mu_{k+1}\big)\|y_e^{k+1}\|^2_{\mathcal{T}_g}-\|y_e^{k}\|^2_{\mathcal{T}_{g}}\nonumber\\
 & +(\tau\sigma)^{-1}\big(\|z_e^{k+1}\|^2\!-\|z_e^k\|^2\big)
   +\|\Delta y^{k+1}\|^2_{\mathcal{P}_{g}+\frac{3}{4}\Sigma_{g}}
    \!-\|\Delta y^k\|^2_{\mathcal{P}_{g}+\frac{3}{4}\Sigma_{g}}\nonumber\\
 &  +\nu_{k+1}\|\Delta y^{k+1}\|^2_{\mathcal{T}_{g}}-\nu_{k}\|\Delta y^{k}\|^2_{\mathcal{T}_{g}}
    +\big(2\!-\!\tau-3\mu_{k+1}\big)\sigma\|h(x^{k+1},y^{k+1})\|^2 \nonumber\\
 &\!-\big(\frac{2\!-\!\tau}{1\!-\!\alpha_{k+1}}+6\mu_{k+1}\big)\frac{1}{\tau^2\sigma}\|\Delta z^k\|^2
    \nonumber\\
 \le& -\frac{\sigma\|h(x^{k},y^{k})\|^2}{1\!-\!\alpha_{k+1}}
     -\!\big(1\!-\!\alpha_{k+1}\big)\|\Delta y^{k+1}\|^2_{\mathcal{T}_{g}-(1-\tau)\sigma\mathcal{B}\mathcal{B}^*}
     -\|\Delta x^{k+1}\|_{\mathcal{P}_{\!f}+\frac{1}{2}\Sigma_{\!f}-\mu_{k+1}(\mathcal{P}_{\!f}+\Sigma_{\!f})}^2\nonumber\\
 \le & -\frac{\sigma}{1\!-\!\alpha_{k+1}}\|h(x^{k},y^{k})\|^2
       -\tau\big(1\!-\!\alpha_{k+1}\big)\|\Delta y^{k+1}\|^2_{\mathcal{T}_{g}},
 \end{align}
 where the last inequality is using $\mathcal{P}_{\!f}+\frac{1}{2}\Sigma_{\!f}-\mu_{k+1}(\mathcal{P}_{\!f}+\Sigma_{\!f})\succeq 0$
 and
 \(
   \mathcal{T}_g-(1\!-\tau)\sigma\mathcal{B}\mathcal{B}^*\succeq \tau\mathcal{T}_g,
 \)
 implied by $\mathcal{P}_{\!f}+\frac{3}{8}\Sigma_{f} \succeq 0$ and $\mu_{k+1}\le 0.1$.
 When $\tau\in(1,2)$, by \eqref{Cauchy} and $\max(\mu_k,\nu_{k})\le 0.1$,
 \begin{align*}
   2(1\!-\!\tau)\sigma\langle h(x^{k},y^{k}), \mathcal{B}^*(y^{k+1}\!-\!y^{k})\rangle
  \le \frac{\sigma\| h(x^{k},y^{k})\|^2}{1\!-\!\alpha_{k+1}}\frac{\tau\!-\!1}{\tau}
      + (1\!-\!\alpha_{k+1})\|\Delta y^{k+1}\|^2_{\sigma(\tau\!-\!1)\tau\mathcal{B}\mathcal{B}^*}.
 \end{align*}
 Substituting the last inequality into Proposition \ref{prop21}(b) and using \eqref{zk}, we obtain that
 \begin{align}\label{temp-equaC22}
  &\big(1\!-\!\mu_{k+1}\big)\|x_e^{k+1}\|_{\mathcal{P}_f+\Sigma_{f}}^2\!-\!\|x_e^{k}\|_{\mathcal{P}_f+\Sigma_{f}}^2
      +\big(1\!-\!\nu_{k+1}-5\mu_{k+1}\big)\|y_e^{k+1}\|^2_{\mathcal{T}_g}-\|y_e^{k}\|^2_{\mathcal{T}_{g}}\nonumber\\
 & +(\tau\sigma)^{-1}\big(\|z_e^{k+1}\|^2\!-\|z_e^k\|^2\big)
   +\|\Delta y^{k+1}\|^2_{\mathcal{P}_{g}+\frac{3}{4}\Sigma_{g}}
    \!-\|\Delta y^k\|^2_{\mathcal{P}_{g}+\frac{3}{4}\Sigma_{g}}\nonumber\\
 &  +\nu_{k+1}\|\Delta y^{k+1}\|^2_{\mathcal{T}_{g}}-\nu_{k}\|\Delta y^{k}\|^2_{\mathcal{T}_{g}}
    +\big(2\!-\!\tau-3\mu_{k+1}\big)\sigma\|h(x^{k+1},y^{k+1})\|^2 \nonumber\\
 &\!-\big(\frac{2\!-\!\tau}{1\!-\!\alpha_{k+1}}+6\mu_{k+1}\big)\frac{1}{\tau^2\sigma}\|\Delta z^k\|^2
    \nonumber\\
 \le& -\!\frac{(1\!+\tau^{-1}\!-\tau)\sigma\|h(x^{k},y^{k})\|^2}{1\!-\!\alpha_{k+1}}
     -\!\big(1\!-\!\alpha_{k+1}\big)\|\Delta y^{k+1}\|^2_{\mathcal{T}_{g}-(\tau-1)\tau\sigma\mathcal{B}\mathcal{B}^*}\nonumber\\
  &\ -\|\Delta x^{k+1}\|_{\mathcal{P}_{\!f}+\frac{1}{2}\Sigma_{\!f}-\mu_{k+1}(\mathcal{P}_{\!f}+\Sigma_{\!f})}^2\nonumber\\
 \le& -\!\frac{(1\!+\tau^{-1}\!-\tau)\sigma\|h(x^{k},y^{k})\|^2}{1\!-\!\alpha_{k+1}}
      -(1+\tau-\tau^2)\!\big(1\!-\!\alpha_{k+1}\big)\|\Delta y^{k+1}\|^2_{\mathcal{T}_{g}},
 \end{align}
 where the last inequality is due to
 $\mathcal{P}_{\!f}+\frac{1}{2}\Sigma_{\!f}-\mu_{k+1}(\mathcal{P}_{\!f}+\Sigma_{\!f})\succeq 0$ by
 $\mu_{k+1}\le 0.1$, and
  \[
   \mathcal{T}_{g}-(\tau\!-\!1)\tau\sigma\mathcal{B}\mathcal{B}^*\succeq (1\!+\tau\!-\tau^2)\mathcal{T}_g.
  \]
 By the definitions of the vector $w^{k+1}$ and the operators $\mathcal{H}, \mathcal{W}_{k}$
 and $\mathcal{V}_k$, the left hand side of \eqref{temp-equaC21} and \eqref{temp-equaC22} is $\|\mathcal{H}w^{k+1}\|^2_{\mathcal{W}_{k+1}}\!-\!\|\mathcal{H}w^{k}\|^2_{\mathcal{V}_{k+1}}$.
 Along with \eqref{temp-equaC21} and \eqref{temp-equaC22}, we get \eqref{aim-ineq2}.

 \medskip

 Since $0\le \max(\mu_k,\nu_{k})\leq \min(0.1,\frac{2-\tau}{4})$ for all $k\ge 1$, it is not difficult to check that
 \[
   \frac{1}{1-\alpha_{k+1}}\le 1\!+10\nu_k\!+40(\mu_{k+1}\!+\!\nu_{k+1})
   \ {\rm and}\ \frac{1}{1\!-\!\frac{3\mu_{k}}{2\!-\!\tau}}\le 1\!+\frac{12\mu_k}{2-\tau},
 \]
 which in turn implies that
 \(
    \frac{2-\tau}{(1-\alpha_{k+1})(1\!-\!\frac{3\mu_{k}}{2\!-\!\tau})}
    \le \big[1\!+10\nu_k\!+40(\mu_{k+1}\!+\!\nu_{k+1})\big]\big(1\!+\frac{12\mu_k}{2-\tau}\big)
 \)
 and
 $\frac{6\mu_{k+1}}{2-\tau\!-\!3\mu_{k}}\le \frac{24\mu_{k+1}}{2-\tau}$.
 Together with the expression of $\mathcal{V}_{k+1}\mathcal{W}_{k}^{-1}$, we obtain that
  \begin{equation}\label{HGk-C2}
  \mathcal{V}_{k+1}\mathcal{W}_{k}^{-1} \preceq \big[1\!+10\nu_k\!+40(\mu_{k+1}\!+\!\nu_{k+1})\big]\big(1\!+\frac{12(\mu_k+2\mu_{k+1})}{2-\tau}\big)\mathcal{I}.
 \end{equation}
 Since $\nu_{k}\le 0.1$ for all $k\ge 1$, we have $0.1\le 1\!-\alpha_{k+1}\le 1$.
 Now combining \eqref{HGk-C2} with inequality \eqref{aim-ineq2} yields the desired result.
 Thus, we complete the proof.
 \end{proof}

 By Lemma \ref{lemma31-C2} one may obtain the following convergence result of the IEIDP-ADMM
 with the criterion (C2). Since the proof is similar to that of Theorem \ref{theorem-C1}, we omit it.
 \begin{theorem}\label{theorem-C2}
  Let $\{(x^{k},y^{k},z^{k})\}_{k\ge 1}$ be the sequence generated by the IEIDP-ADMM with
  the criterion (C2) and $\max(\mu_k,\nu_{k})\le \min(0.1,\frac{2-\tau}{4})$.
  Suppose that Assumption \ref{assump} holds and the operators $\mathcal{P}_{\!f}$ and $\mathcal{P}_{\!g}$
  also satisfy $\mathcal{P}_{\!f}+\frac{3}{8}\Sigma_{f}\succeq 0$ and $\mathcal{P}_{g}+\frac{3}{8}\Sigma_{g} \succeq 0$.
  Then, for (a) $\tau\in \big(0,\frac{1+\sqrt{5}}{2}\big)$ or (b) $\tau\in[\frac{1+\sqrt{5}}{2},2)$ but
  $\sum_{k=0}^{\infty}\big(\frac{10\sigma}{\tau}\|h(x^{k},y^{k})\|^2\!+\!\|\Delta y^{k+1}\|^2_{\mathcal{T}_g}\big)<\infty$,
  the sequence $\{(x^{k},y^{k})\}$ converges to an optimal solution of problem
  \eqref{prob} and the sequence $\{z^{k}\}$ converges to an optimal solution to the dual problem of \eqref{prob}.
 \end{theorem}

 \begin{remark}
  Theorem \ref{theorem-C2} shows that one can establish the convergence of $\{(x^{k},y^{k},z^{k})\}$
  generated by the IEIDP-ADMM with (C2) if $\mathcal{P}_{\!f}$ and $\mathcal{P}_{\!g}$ are chosen such that
 \begin{align*}
  \mathcal{P}_{\!f}+\frac{3}{8}\Sigma_{f}\succeq 0,\ \mathcal{P}_{g}+\frac{3}{8}\Sigma_{g}\succeq 0,\
  \Sigma_{f}+\mathcal{P}_{\!f}+\sigma\mathcal{A}\mathcal{A}^* \succ 0,\ \Sigma_{g}+\mathcal{P}_{g}+\sigma\mathcal{B}\mathcal{B}^* \succ 0.
  \end{align*}
  In fact, using the same arguments, one can get the convergence of $\{(x^{k},y^{k},z^{k})\}$
  generated by the IEIDP-ADMM with (C2) if $\mathcal{P}_f$ and $\mathcal{P}_g$ are chosen such that
  for some $a_1,a_2\in [\frac{1}{2},1)$,
 \begin{align*}
   \mathcal{P}_{\!f}\!+\!\frac{1}{2}\Sigma_{f}\succeq 0,\ \mathcal{P}_{g}\!+\!\frac{1}{2}\Sigma_{g}\succeq 0,\
  a_1\Sigma_{f}\!+\!\mathcal{P}_{\!f}\!+\!\sigma\mathcal{A}\mathcal{A}^* \succ 0,\ a_2\Sigma_{g}\!+\!\mathcal{P}_{g}\!+\!\sigma\mathcal{B}\mathcal{B}^* \succ 0.
 \end{align*}
 \end{remark}

  \subsection{Convergence of the IEIDP-ADMM with (C2')}\label{subsec4.3}

  Let $w^{k}$ for $k\ge 1$ be same as the one of the last subsection. Define the block diagonal linear operator  $\mathcal{H}\!:\mathbb{X}\times\mathbb{Y}\times\mathbb{Z}\times\mathbb{Y}\times\mathbb{Y}
  \times\mathbb{Z}\to \mathbb{X}\times\mathbb{Y}\times\mathbb{Z}\times\mathbb{Y}\times\mathbb{Y}\times\mathbb{Z}$ by
  \[
    \mathcal{H}:={\rm Diag}\big((\mathcal{P}_{\!f}\!+\!\Sigma_{\!f})^{1/2},\ (\mathcal{T}_g)^{1/2},\
    (\tau\sigma)^{-1}\mathcal{I},\ (\mathcal{P}_{g}\!+\!{\textstyle\frac{3}{4}}\Sigma_g)^{1/2},
    \ (\mathcal{T}_g)^{1/2},\ (\tau^2\sigma)^{-1}\mathcal{I}\big)
  \]
  with the proximal operators $\mathcal{P}_{\!f}$ and $\mathcal{P}_{g}$ satisfying $\mathcal{P}_{\!f}+\Sigma_{\!f}\succeq 0$
  and $\mathcal{P}_{g}\!+\!\frac{3}{4}\Sigma_g\succeq 0$.
 \begin{lemma}\label{lemma31-C3}
  Let $\{(x^{k},y^{k},z^{k})\}_{k\ge 1}$ be the sequence generated by the IEIDP-ADMM with
  (C2') and $\max(\mu_k^2,\nu_{k}^2)\le\min(\frac{1}{8\gamma},\frac{2-\tau-2.5\gamma^{-1}}{3\gamma})$
  for some constant $\gamma\ge 360$. Suppose that Assumption \ref{assump} holds
  and the operators $\mathcal{P}_{\!f}$ and $\mathcal{P}_{g}$ also satisfy $\mathcal{P}_{\!f}+\frac{3}{8}\Sigma_{f} \succeq 0$
  and $\mathcal{P}_{g}+\frac{3}{8}\Sigma_{g}\succeq 0$. Then, when $\tau\in(0,2)$,
  the following inequality holds for all $k\ge 1$
  \begin{align*}
  \|\mathcal{H}w^{k+1}\|^2_{\mathcal{W}_{k+1}}
  & \le\max\Big(1+\frac{3\gamma\mu_{k}^2}{2\!-\!\tau\!-\!2.5\gamma^{-1}},1+2\gamma(\nu_k^2+3\mu_k^2)\Big)\|\mathcal{H}w^{k}\|^2_{\mathcal{W}_{k}}\\
    &\!\left\{\begin{array}{ll}
        \!-\frac{\gamma-10}{\gamma}\min(\tau,2.6\!-\!1.6\tau)\|h(x^{k},y^{k})\|^2\!-\!c_1\|\Delta y^{k+1}\|_{\mathcal{T}_g}^2 & {\rm if}\ \tau\in(0,1.6]\\
        +\frac{\gamma-4}{\gamma}|2.6\!-\!1.6\tau|\|h(x^{k},y^{k})\|^2\!+\!c_2\|\Delta y^{k+1}\|_{\mathcal{T}_g}^2 & {\rm if}\ \tau\in(1.6,2)
        \end{array}\right.
 \end{align*}
 where $c_1=2\!-\!\tau\!-\!6.5\gamma^{-1}-\frac{\max((1\!-\!\tau),(\tau\!-\!1)/1.6)}{1\!-\!6\gamma^{-1}}$,
  $c_2=\big|2\!-\!\tau\!-\!6.5\gamma^{-1}-\frac{\max((1\!-\!\tau),(\tau\!-\!1)/1.6)}{1\!-\!10\gamma^{-1}}\big|$,
  and  $\mathcal{W}_{k}\!:\mathbb{X}\times\mathbb{Y}\times\mathbb{Z}\times\mathbb{Y}\times\mathbb{Y}
  \times\mathbb{Z}\to \mathbb{X}\times\mathbb{Y}\times\mathbb{Z}\times\mathbb{Y}\times\mathbb{Y}\times\mathbb{Z}$
  is the block diagonal operator
  \[
   \mathcal{W}_k:={\rm Diag}\big((1\!-\!\gamma\mu^2_{k})\mathcal{I},\,(1\!-\!\gamma\nu^2_{k}-3\gamma\mu_{k}^2)\mathcal{I},\,
                   \mathcal{I},\,\mathcal{I},\,\gamma^{-1}\mathcal{I},\,(2\!-\!\tau\!-\!2.5\gamma^{-1}\!-\!1.5\gamma\mu_{k}^2)\mathcal{I}\big).
  \]
 \end{lemma}
 \begin{proof}
  Let $\mathcal{V}\!:\mathbb{X}\times\mathbb{Y}\times\mathbb{Z}\times\mathbb{Y}\times\mathbb{Y}
  \times\mathbb{Z}\to \mathbb{X}\times\mathbb{Y}\times\mathbb{Z}\times\mathbb{Y}\times\mathbb{Y}\times\mathbb{Z}$
  be the block diagonal linear operator defined by
  \(
    \mathcal{V}:={\rm Diag}\big(\mathcal{I},\,\mathcal{I},\,
                   \mathcal{I},\,\mathcal{I},\,\mathcal{I},\,(2\!-\!\tau\!-\!2.5\gamma^{-1})\mathcal{I}\big).
  \)
  With the notations $\mathcal{W}_k$ and $\mathcal{V}$, we first establish the following inequality
  \begin{align}\label{aim-ineq3}
    \|\mathcal{H}w^{k+1}\|^2_{\mathcal{W}_{k+1}}\!-\!\|\mathcal{H}w^{k}\|^2_{\mathcal{V}}
    &\le -\!\min(\tau,2.6\!-\!1.6\tau)\big[1\!-\!6\gamma^{-1}\!-\!\gamma(\nu^2_{k}+\nu^2_{k+1})\big]\|\Delta y^{k+1}\|_{\mathcal{T}_g}\nonumber\\
     &\quad\ -\!c_k(\gamma,\tau)\| h(x^{k},y^{k})\|^2,
  \end{align}
  where
  \[
    c_k(\gamma,\tau):=2\!-\!\tau\!-\!\frac{6.5}{\gamma}-\frac{\max(1\!-\!\tau,(\tau\!-\!1)/1.6)}{1\!-\!6\gamma^{-1}\!-\!\gamma(\nu^2_{k}+\nu^2_{k+1})}.
  \]
  Indeed, when $\tau\in (0,1]$, from $1\!-6\gamma^{-1}\!-\gamma(\nu^2_{k}+\nu^2_{k+1})>0$
  and equation \eqref{Cauchy} it follows that
  \begin{align*}
   2\sigma(1\!-\!\tau)\big|\langle h(x^{k},y^{k}), \mathcal{B}^*\Delta y^{k+1}\rangle\big|
  &\le \frac{\sigma(1\!-\!\tau)}{1\!-\!6\gamma^{-1}\!-\!\gamma(\nu^2_{k}+\nu^2_{k+1})}\| h(x^{k},y^{k})\|^2 \nonumber\\
  &\quad +\![1\!-\!6\gamma^{-1}\!-\!\gamma(\nu^2_{k}+\nu^2_{k+1})]\|\Delta y^{k+1}\|^2_{(1-\tau)\sigma\mathcal{B}\mathcal{B}^*}.
  \end{align*}
   Substituting the last inequality into Proposition \ref{prop21}(c) then yields that
   \begin{align}\label{C3-case1}
    &\|\mathcal{H}w^{k+1}\|^2_{\mathcal{W}_{k+1}}\!-\!\|\mathcal{H}w^{k}\|^2_{\mathcal{V}}
     +c_k(\gamma,\tau)\sigma\|h(x^{k},y^{k})\|^2\nonumber\\
    &\le -\|\Delta x^{k+1}\|_{\mathcal{P}_{\!f}
      +\frac{1}{2}\Sigma_{\!f}-\frac{1}{\gamma}(\mathcal{P}_{\!f}+\Sigma_{\!f})}^2
     \!-\!\big[1\!-\!6\gamma^{-1}\!-\!\gamma(\nu^2_{k}+\nu^2_{k+1})\big]
     \|\Delta y^{k+1}\|_{\mathcal{T}_g-(1-\tau)\sigma\mathcal{B}\mathcal{B}^*}^2 \nonumber\\
    &\le -\tau\big(1\!-\!6\gamma^{-1}\!-\!\gamma\nu^2_{k}-\gamma\nu^2_{k+1}\big)\|\Delta y^{k+1}\|_{\mathcal{T}_g},
   \end{align}
  where the last inequality is using
  $\mathcal{P}_{\!f}+\frac{1}{2}\Sigma_{\!f}-\gamma^{-1}(\mathcal{P}_{\!f}+\Sigma_{\!f})\succeq 0$
  and $\mathcal{T}_g-(1-\tau)\sigma\mathcal{B}\mathcal{B}^*\succeq \tau\mathcal{T}_g$.
  When $\tau\in(1,2)$, from $1\!-6\gamma^{-1}\!-\gamma(\nu^2_{k}+\nu^2_{k+1})>0$
  and equation \eqref{Cauchy} it follows that
  \begin{align*}
   2\big|(1\!-\!\tau)\sigma\langle h(x^{k},y^{k}), \mathcal{B}^*(y^{k+1}\!-\!y^{k})\rangle\big|
   &\le \frac{\sigma(\tau\!-\!1)}{1.6(1\!-\!6\gamma^{-1}\!-\!\gamma\nu^2_{k}\!-\!\gamma\nu^2_{k+1})}\| h(x^{k},y^{k})\|^2\\
   &\quad + \big(1\!-\!6\gamma^{-1}\!-\!\gamma\nu^2_{k}\!-\!\gamma\nu^2_{k+1}\big)\|\Delta y^{k+1}\|^2_{1.6(\tau\!-\!1)\sigma\mathcal{B}\mathcal{B}^*}.\nonumber
  \end{align*}
  Substituting it into Proposition \ref{prop21}(c) and
  using the notations $\mathcal{W}_k$ and $\mathcal{V}$, we have
   \begin{align}\label{C3-case2}
    &\|\mathcal{H}w^{k+1}\|^2_{\mathcal{W}_{k+1}}\!-\!\|\mathcal{H}w^{k}\|^2_{\mathcal{V}}
     +c_k(\gamma,\tau)\sigma\|h(x^{k},y^{k})\|^2\nonumber\\
    &\le -\|\Delta x^{k+1}\|_{\mathcal{P}_{\!f}
      +\frac{1}{2}\Sigma_{\!f}-\frac{1}{\gamma}(\mathcal{P}_{\!f}+\Sigma_{\!f})}^2
     \!-\!\big[1\!-\!6\gamma^{-1}\!-\!\gamma(\nu^2_{k}+\nu^2_{k+1})\big]
     \|\Delta y^{k+1}\|_{\mathcal{T}_g-1.6(\tau-1)\sigma\mathcal{B}\mathcal{B}^*}^2 \nonumber\\
    &\le -(2.6-1.6\tau)\big(1\!-\!6\gamma^{-1}\!-\!\gamma\nu^2_{k}\!-\!\gamma\nu^2_{k+1}\big)\|\Delta y^{k+1}\|_{\mathcal{T}_g}^2,
   \end{align}
   where the last inequality is using
  $\mathcal{P}_{\!f}+\frac{1}{2}\Sigma_{\!f}-\gamma^{-1}(\mathcal{P}_{\!f}+\Sigma_{\!f})\succeq 0$
  and $\mathcal{T}_g-1.6(\tau\!-\!1)\sigma\mathcal{B}\mathcal{B}^*\succeq (1-\!1.6(\tau\!-1))\mathcal{T}_g$.
   From \eqref{C3-case1} and \eqref{C3-case2}, we immediately obtain inequality \eqref{aim-ineq3}.

  \medskip

  Now by the given condition $\max(\mu_k^2,\nu_{k}^2)\le\min(\frac{1}{8},\frac{2-\tau-2.5\gamma^{-1}}{3})\frac{1}{\gamma}$,
  we can check that
  \[
    \frac{2\!-\!\tau\!-\!2.5\gamma^{-1}}{2\!-\!\tau\!-\!2.5\gamma^{-1}\!-\!1.5\gamma\mu_{k}^2}
    \le 1+\frac{3\gamma\mu_{k}^2}{2\!-\!\tau\!-\!2.5\gamma^{-1}}\ {\rm and}\
     \frac{1}{1-\gamma\nu_k^2-3\gamma\mu_k^2}\le 1+2\gamma(\nu_k^2+3\mu_k^2).
  \]
  Together with the expressions of $\mathcal{V}$ and $\mathcal{W}_k$, it is not difficult to verify that
  \[
    \mathcal{V}\mathcal{W}_{k}^{-1} \preceq \max\Big(1+\frac{3\gamma\mu_{k}^2}{2\!-\!\tau\!-\!2.5\gamma^{-1}},1+2\gamma(\nu_k^2+3\mu_k^2)\Big)\mathcal{I}.
  \]
  Combining this relation with \eqref{aim-ineq3} and the condition
  $\max(\mu_k^2,\nu_{k}^2)\le\min(\frac{1}{8},\frac{2-\tau-2.5\gamma^{-1}}{3})\frac{1}{\gamma}$,
  we obtain the desired result. The proof is completed.
  \end{proof}

  \medskip

 By Lemma \ref{lemma31-C3} we can establish the following convergence result of the IEIDP-ADMM with
 the criterion (C2'). Since the proof is similar to that of Theorem \ref{theorem-C1}, we omit it.
 \begin{theorem}\label{theorem-C3}
  Let $\{(x^{k},y^{k},z^{k})\}_{k\ge 1}$ be the sequence generated by the IEIDP-ADMM with
  the criterion (C2') and $\max(\mu_k^2,\nu_{k}^2)\le\min(\frac{1}{8\gamma},\frac{2-\tau-2.5\gamma^{-1}}{3\gamma})$
  for some constant $\gamma\ge 360$. Suppose that Assumption \ref{assump} holds
  and $\mathcal{P}_{\!f}$ and $\mathcal{P}_{\!g}$ also satisfy $\mathcal{P}_{\!f}+\frac{3}{8}\Sigma_{f}\succeq 0$
  and $\mathcal{P}_{g}+\frac{3}{8}\Sigma_{g}\!\succeq 0$.
  Then, for (a) $\tau\in \big(0,1.6\big]$ or (b) $\tau\in(1.6,2)$ but
  $\sum_{k=0}^{\infty}\big(|2.6\!-1.6\tau|\|h(x^{k},y^{k})\|^2\!+\!c_2\|\Delta y^{k+1}\|_{\mathcal{T}_g}^2\big)<\infty$,
  the sequence $\{(x^{k},y^{k})\}$ converges to an optimal solution of \eqref{prob}
  and the sequence $\{z^{k}\}$ converges to an optimal solution to the dual problem of \eqref{prob}.
 \end{theorem}

  To close this section, we want to point out that the convergence of the inexact positive
  definite proximal ADMM \cite{NWY11} with (C1) and a special (C2') is only established
  for $\tau=1$, while the convergence results of Theorem \ref{theorem-C1} and Theorem \ref{theorem-C3}
  extend it to the inexact indefinite proximal ADMM with $\tau\in(0,\frac{\sqrt{5}+1}{2})$
  and $\tau\in(0,1.6]$, respectively.

 \section{Applications to doubly nonnegative SDPs}\label{sec5}

  Let $\mathcal{S}_{+}^n$ be the cone of $n\times n$ positive semidefinite matrices in
  the vector space $\mathbb{S}^n$ of $n\times\! n$ real symmetric matrices, endowed with
  the Frobenius inner product $\langle\cdot,\cdot\rangle$ and its induced norm $\|\cdot\|$.
  The doubly nonnegative SDP problem is described as follows:
  \begin{align}\label{PDNNSDP}
   \max\Big\{\!-\!\big\langle C,X\big\rangle\ |\ \mathcal{A}_EX=b_E,\ \mathcal{A}_{I}X\ge b_I,\ X\in\mathcal{S}_{+}^n,\ X\!-\!M\in\mathcal{K}\Big\}
  \end{align}
  where $\mathcal{A}_E\!:\mathbb{S}^n\to\mathbb{R}^{m_{E}}$ and $\mathcal{A}_I\!:\mathbb{S}^n\to\mathbb{R}^{m_{I}}$
  are the linear operators, $b_E\!\in\!\mathbb{R}^{m_{E}}$ and $b_I\!\in\!\mathbb{R}^{m_{I}}$ are
  the given vectors, and $X\!-M\in\mathcal{K}$ means that every entry of $X\!-M$ is nonnegative.
  We always assume that $\mathcal{A}_E$ is surjective. The dual of \eqref{PDNNSDP} has the form
  \begin{align}\label{DDNNSDP}
   &\min \big(\delta_{\mathbb{R}_{+}^{m_I}}(y_I)\!-\!\langle b_I,y_I\rangle\big)
     +\big(\delta_{\mathcal{K}^*}(Z)\!-\!\langle M,Z\rangle\big)-\langle b_E,y_E\rangle+\delta_{\mathcal{S}_{+}^n}(S)\nonumber\\
   &\ {\rm s.t.}\ \ \mathcal{A}_{I}^*y_I+Z+\mathcal{A}_E^*y_E+S=C
  \end{align}
  where $\mathcal{K}^*$ is the positive dual cone of $\mathcal{K}$. For the four-block
  separable convex minimization problem \eqref{DDNNSDP}, one may use the multi-block ADMM
  with Gaussian back substitution \cite{HTY12,HY13}
  or the proximal ADMM \cite{STY14} to solve. In this section, we apply the IEIDP-ADMMs for
  \eqref{DDNNSDP} by viewing $(y_I,Z)$ as a block and $(y_E,S)$ as a block (respectively,
  viewing $(Z,y_E)$ as a block and $S$ as a block when $m_I=0$). Notice that, by introducing
  a slack variable, problem \eqref{PDNNSDP} can be equivalently written as
  \begin{align}\label{Equi-PDNNSDP}
   \max\Big\{\!-\!\big\langle C,X\big\rangle\ |\ \mathcal{A}_EX=b_E,\ \mathcal{A}_{I}X-x=b_I,\ X\in\mathcal{S}_{+}^n,\ X\!-\!M\in\mathcal{K},\ x\ge 0\Big\},
  \end{align}
  and an elementary calculation yields the dual problem of \eqref{Equi-PDNNSDP} as follows
  \begin{align}\label{Equi-DDNNSDP}
   &\min -\langle b_I,y_I\rangle+\delta_{\mathbb{R}_{+}^{m_I}}(z)+\big(\delta_{\mathcal{K}^*}(Z)\!-\!\langle M,Z\rangle\big)
     -\langle b_E,y_E\rangle+\delta_{\mathcal{S}_{+}^n}(S)\nonumber\\
   &\ {\rm s.t.}\ \ \mathcal{A}_{I}^*y_I+Z+\mathcal{A}_E^*y_E+S=C,\ \ y_I-z = 0.
  \end{align}
  Problem \eqref{Equi-DDNNSDP} is still a four-block separable convex minimization
  since $(z,Z)$ can be solved simultaneously. Hence, in this section we also apply
  the IEIDP-ADMMs for solving \eqref{Equi-DDNNSDP} by viewing $(y_I,z,Z)$ as a block
  and $(y_E,S)$ as a block.

  \medskip

  Throughout this section, instead of using the constraint qualification (CQ) in Assumption \ref{assump},
  we use the following more familiar Slater's CQ for problem \eqref{DDNNSDP}:
  \begin{assumption}\label{assumpB}
  (a) For problem \eqref{PDNNSDP}, there exists a point $\widehat{X}\in\mathbb{S}^n$ such that
      \[
        \mathcal{A}_E\widehat{X}=b_E,\ \mathcal{A}_I\widehat{X}\ge b_I,\ \widehat{X}\in{\rm int}\,(\mathcal{S}_{+}^n),\ \widehat{X}\in\mathcal{K}.
      \]

  \noindent
  (b) For problem \eqref{DDNNSDP}, there exists a point $(\widehat{y}_I,\widehat{Z},\widehat{y}_E,\widehat{S})
  \in\!\mathbb{R}^{m_I}\times\mathbb{S}^n\times\mathbb{R}^{m_E}\times\mathbb{S}^n$ such that
       \[
        \mathcal{A}_{I}^*\widehat{y}_I+\widehat{Z}+\mathcal{A}_E^*\widehat{y}_E+\widehat{S}=c,\ \widehat{Z}\in\mathcal{K}^*,\
        \widehat{y}_I\in\mathbb{R}_{+}^{m_I},\ \widehat{S}\in{\rm int}\,(\mathcal{S}_{+}^n).
      \]
 \end{assumption}
  By \cite[Corollary 5.3.6]{BL06}, under Assumption \ref{assumpB}, the strong duality for \eqref{PDNNSDP}
  and \eqref{DDNNSDP} holds, and the following Karush-Kuhn-Tucker (KKT) condition has nonempty solutions:
  \begin{equation}\label{KKT-cond}
    \left\{\begin{array}{l}
     \mathcal{A}_{E}X-b_E=0,\\
     \mathcal{A}_{I}^*y_I+Z+\mathcal{A}_{E}^*y_E+S-C=0,\\
     \langle X,S\rangle=0,\ X\in\mathcal{S}_{+}^n,\ S\in\mathcal{S}_{+}^n,\\
     \langle X,Z\rangle =0,\ X\in\mathcal{K},\ Z\in\mathcal{K}^*,\\
     \langle y_I,\mathcal{A}_{I}X\!-\!b_I\rangle=0,\ \mathcal{A}_{I}X\!-\!b_I\ge 0,\ y_I\in\mathbb{R}_{+}^{m_I}.
     \end{array}\right.
  \end{equation}

 \subsection{Numerical results for the DNNSDPs without $\mathcal{A}_{I}X\ge b_I$}\label{subsec5.1}

 In this case since the linear operator $[\mathcal{I}\ \ \mathcal{A}_{E}^*]^*[\mathcal{I}\ \ \mathcal{A}_{E}^*]$
 is not positive definite, we impose a semi-proximal term
 $\frac{1}{2}(Z-\!Z^k,y_E-\!y_E^k){\rm Diag}(\varepsilon\mathcal{I},0)(Z-\!Z^k,y_E-\!y_E^k)^T$ to guarantee that
 \begin{equation}\label{Tf-three}
  \mathcal{T}_{\!f}\succeq\sigma[\mathcal{I}\ \ \mathcal{A}_{E}^*]^*[\mathcal{I}\ \ \mathcal{A}_{E}^*]+{\rm Diag}(\varepsilon\mathcal{I},0)
           =\sigma\!\left[\begin{matrix}
                        \frac{\sigma+\varepsilon}{\sigma}\mathcal{I} &\mathcal{A}_{E}^*\\
                        \mathcal{A}_{E} & \mathcal{A}_{E}\mathcal{A}_{E}^*
                  \end{matrix}\right]\succ 0,
 \end{equation}
 and propose the following partial IEIDP-ADMM for problem (\ref{DDNNSDP}) with three blocks,
 where for a given $\sigma>0$, the augmented Lagrangian function of \eqref{DDNNSDP} is defined as
  \begin{align}
     L_{\sigma}(y_I,Z,y_E,S,X)&:=(\delta_{\mathbb{R}_{+}^{m_I}}(y_I)-\langle b_I,y_I\rangle)+(\delta_{\mathcal{K}^*}(Z)-\langle M,Z\rangle)
     -\langle b_E,y_E\rangle+\delta_{\mathcal{S}_{+}^n}(S)\nonumber\\
     &\quad\ +\!\langle X,\mathcal{A}_{I}^*y_I\!+\!Z\!+\!\mathcal{A}_{E}^*y_E\!+\!S\!-C\rangle +\frac{\sigma}{2}\big\|\mathcal{A}_{I}^*y_I\!+\!Z\!+\!\mathcal{A}_{E}^*y_E\!+\!S\!-C\big\|^2\nonumber\\
    &\qquad \forall(y_I,Z,y_E,S,X)\in\mathbb{R}^{m_I}\times\mathbb{S}^n\times\mathbb{R}^{m_E}\times\mathbb{S}^n\times\mathbb{S}^n.\nonumber
  \end{align}

 \bigskip

 \setlength{\fboxrule}{0.8pt}
 \noindent
 \fbox{
 \parbox{0.96\textwidth}
 {
 \begin{algorithm} \label{Alg1}({\bf A partial IEIDP-ADMM for (\ref{DDNNSDP}) with three blocks})
 \begin{description}
   \item[(S.0)] Let $\mathcal{T}\!=\varrho\mathcal{I}-\mathcal{A}_{I}\mathcal{A}_{I}^*$
                for $\varrho>\lambda_{\rm max}(\mathcal{A}_{I}\mathcal{A}_{I}^*)$.
                Let $\sigma,\tau>0$ be given. Choose a  \hspace*{0.05cm} small constant
                $\varepsilon>0$ and a point $(Z^0,y_{E}^0,S^0,X^0)=(0,0,0,0)$.
                Set $k:=0$.

   \item[(S.1)] Compute the following problems by one of the criteria (C1) and (C2):
                \begin{align}\label{3subprob}
                 &(Z^{k+1},y_E^{k+1})\approx\mathop{\arg\min}_{Z,y_E}\phi_k(Z,y_E)\!:=L_{\sigma}(0,Z,y_E,S^{k},X^{k})+\frac{1}{2}\|Z\!-\!Z^k\|_{\varepsilon\mathcal{I}}^2;\\
                 &S^{k+1}=\mathop{\arg\min}_{S}L_{\sigma}(0,Z^{k+1},y_E^{k+1},S,X^{k})= \Pi_{\mathcal{S}^n_{+}}(C\!-\!\mathcal{A}_{E}^*y_E^{k+1}\!-\!Z^{k+1}\!-\!\sigma^{-1}X^{k}).\nonumber
                \end{align}

   \item[(S.3)]  Update the Lagrange multiplier $X^{k+1}$ via the following formula
                \[
                  X^{k+1} = X^{k}+\tau\sigma( Z^{k+1}+\mathcal{A}_{E}^*y_E^{k+1}+S^{k+1} -C);
                \]
  \item[(S.4)]  Let $k\leftarrow k+1$, and go to Step (S.1).

  \end{description}
  \end{algorithm}}
  }

 \bigskip

 For the approximate optimal solution $(Z^{k+1},y_E^{k+1})$ of subproblem \eqref{3subprob},
 one may get it by solving the problem $\min_{Z,y_E}\phi_k(Z,y_E)$ in an alternating way.
 Let $k_0=k$. The iterates $(Z^{k_j},y_{_E}^{k_j})$ yielded by solving the problem $\min_{Z,y_E}\phi_k(Z,y_E)$
 alternately satisfy
 \[
   Z^{k_j}=\mathop{\arg\min}_{Z\in\mathbb{S}^n}\phi_k(Z,y_{\!_E}^{k_{j-1}})\ \ {\rm and}\ \
   y_{\!_E}^{k_j}=\mathop{\arg\min}_{y_E\in\mathbb{R}^{m_{\!E}}}\phi_k(Z^{k_j},y_{\!_E})\quad{\rm for}\ j=1,2,\ldots.
 \]
 From the expression of the function $\phi_k(\cdot,\cdot)$, it is immediate to obtain that
 \[
  \left\{\begin{array}{l}
   0\in\mathcal{N}_{\mathcal{K}^*}(Z^{k_j})-M+X^k +\sigma(Z^{k_j}\!+\!\mathcal{A}_{E}^*y_{\!_E}^{k_{j-1}}\!+\!S^{k}\!-\!C)+\varepsilon(Z^{k_j}\!-\!Z^{k}),\\
    0= \mathcal{A}_{E}X^{k}-b_E +\sigma\mathcal{A}_{E}(\mathcal{A}_{E}^*y_{\!_E}^{k_j}\!+\!Z^{k_j}\!+\!S^{k}\!-\!C).
   \end{array}\right.
 \]
  Comparing this system with the optimality condition of $\min_{Z,y_E} \phi_k(Z,y_E)$,
  with $\xi^{k_j}=\sigma\mathcal{A}_{E}^*(y_{\!_E}^{k_j}\!-\!y_{\!_E}^{k_{j-1}})$ we have
 $(\xi^{k_j},0)\in\partial\phi_k(Z^{k_j},y_{\!_E}^{k_j})$. This means that $(Z^{k_j},y_{\!_E}^{k_j})$
 satisfies the criterion (C1) with $\nu_k\equiv 0$ when $\|\xi^{k_j}\|\le \mu_{k+1}$
 and $\sum_{k=0}^{\infty}\mu_{k+1}<\infty$. In addition, let
 \[
   \delta:=(\sqrt{\sigma+\varepsilon}-\!\sqrt{\sigma})\min\Big(\sqrt{\sigma\!+\!\varepsilon},
  \frac{\sigma}{\sqrt{\sigma\!+\!\varepsilon}}\lambda_{\rm min}(\mathcal{A}_{E}\mathcal{A}_{E}^*)\Big).
 \]
 By using equation \eqref{Tf-three} and \cite[Theorem 7.7.6]{HJ86}, it is not difficult to verify that
 \[
  \mathcal{T}_{\!f}\succeq
  \sigma\left[\begin{matrix}
                        \frac{\sqrt{\sigma+\varepsilon}-\sqrt{\sigma}}{\sigma}\sqrt{\sigma\!+\!\varepsilon}\mathcal{I} &0\\
                        0& \frac{\sqrt{\sigma\!+\!\varepsilon}-\sqrt{\sigma}}{\sqrt{\sigma+\varepsilon}}\mathcal{A}_{E}\mathcal{A}_{E}^*
                  \end{matrix}\right]\succeq\delta \mathcal{I}.
  \]
  This means that $(Z^{k_j},y_{\!_E}^{k_j})$ satisfies the criterion (C2) with $\mathcal{F}=\delta^{-1}\mathcal{I}$
  and $\nu_k\equiv 0$ once
  \[
   \|\xi^{k_j}\|
   \le \sqrt{\delta\sigma}\mu_{k+1}\sqrt{\|Z^{k_j}\!-\!Z^k+\mathcal{A}_{E}^*(y_{\!_E}^{k_j}\!-\!y_{\!_E}^k)\|^2
            +\frac{\varepsilon}{\sigma}\|Z^{k_j}\!-\!Z^k\|^2}\ {\rm and}\ {\textstyle\sum_{k=0}^{\infty}}\mu_{k+1}<\infty,
  \]
  since the right hand side of the first inequality is less than
  $\sqrt{\delta}\mu_{k+1}\|(Z^{k_j}\!-\!Z^k;y_{\!_E}^{k_j}\!-\!y_{\!_E}^k)\|_{\mathcal{T}_{\!f}}$.
  In the sequel, we call Algorithm \ref{Alg1} with the subproblems in \eqref{3subprob} solved alternately
  by the criteria (C1) and (C2) IEIDP-ADMM1 and IEIDP-ADMM2, respectively.

  \medskip

  We apply IEIDP-ADMM1 and IEIDP-ADMM2 for the doubly nonnegative SDPs without
  inequality constraint $\mathcal{A}_{I}X\ge b_I$, and compare their performance with
  that of the $3$-block ADMM of step-size $\tau=1.618$ (for short, ADMM3d).
  Among others, the doubly nonnegative SDP test examples can be found in \cite{STY14}.
  We have implemented IEIDP-ADMM1, IEIDP-ADMM2 and ADMM3d in MATLAB,
  where $\varepsilon=10^{-5}$ and $\mu_{k}=\min(0.1,\frac{1}{k^{1.001}})$ for $k\ge 1$
  are used for IEIDP-ADMM1 and IEIDP-ADMM2. Notice that when
  $\|(Z^{k_j}\!-\!Z^k;y_{\!_E}^{k_j}\!-\!y_{\!_E}^k)\|_{\mathcal{T}_{\!f}}<1$, the criterion (C2) is
  more restrictive than (C1). Moreover, the criterion (C2) will require much
  more inner iterations as the primal and dual infeasibility becomes smaller since
  $\|(Z^{k_j}\!-\!Z^k;y_{\!_E}^{k_j}\!-\!y_{\!_E}^k)\|_{\mathcal{T}_{\!f}}$ is close to $0$.
  So, in the implementation of IEIDP-ADMM2, we modify the criterion (C2) into
  \begin{equation}
    \|\xi^{k_j}\|\le \max\big(\sqrt{\delta}\mu_{k+1}\|(Z^{k_j}\!-\!Z^k;y_{\!_E}^{k_j}\!-\!y_{\!_E}^k)\|_{\mathcal{T}_{\!f}},
    0.1\max(\eta_{P},\eta_{D})\big),
  \end{equation}
  where $\eta_{P}$ and $\eta_{D}$ are defined below. In addition, the implementation of ADMM3d
  here is different from that of \cite{WGY10} since the former uses
  the solution order $Z\rightarrow y_E\rightarrow S$, while the latter uses the order
  $y_E\rightarrow Z\rightarrow S$. The computational results for all DNNSDPs
  are obtained on a Windows system with Intel(R) Core(TM) i3-2120 CPU@3.30GHz.

  \medskip

  We measure the accuracy of an approximate optimal solution $(Z,y_E,S,X)$ for \eqref{PDNNSDP}
  and \eqref{DDNNSDP} by using the relative residual
  \(
    \eta=\max\big\{\eta_{P},\eta_{D},\eta_{\mathcal{S}},\eta_{\mathcal{K}},\eta_{\mathcal{S}^*},\eta_{\mathcal{K}^*},\eta_{C_1},\eta_{C_2}\big\}
  \)
  where
  \begin{align*}
   \eta_{P}\!=\!\frac{\|\mathcal{A}_{E}X\!-\!b_E\|}{1+\|b_E\|},\ \eta_D\!=\!\frac{\|\mathcal{A}_{E}^*y_E\!+\!S\!+\!Z\!-\!C\|}{1+\|C\|},\
   \eta_{\mathcal{S}}\!=\!\frac{\|\Pi_{\mathcal{S}_{+}^n}(-X)\|}{1+\|X\|},\ \eta_{\mathcal{K}}=\frac{\|\Pi_{\mathcal{K}^*}(-X)\|}{1+\|X\|},\\
   \eta_{\mathcal{S}^*}\!=\!\frac{\|\Pi_{\mathcal{S}_{+}^n}(-S)\|}{1+\|S\|},\
   \eta_{\mathcal{K}^*}\!=\!\frac{\|\Pi_{\mathcal{K}^*}(-Z)\|}{1+\|Z\|},\
   \eta_{C_1}\!=\!\frac{\langle X,S\rangle}{1+\|X\|+\|S\|},\
   \eta_{C_2}\!=\!\frac{\langle X,Z\rangle}{1+\|X\|+\|Z\|}.
  \end{align*}
  We terminated the three solvers IEIDP-ADMM1, IEIDP-ADMM2 and ADMM3d whenever $\eta<10^{-6}$
  or the number of iteration is over $k_{\rm max}=20000$.

  \medskip

  In the implementation of the three solvers, the penalty parameter $\sigma$ is dynamically adjusted
  according to the progress of the algorithms, and the idea to adjust $\sigma$ is to balance
  the progress of primal feasibilities $(\eta_P,\eta_{\mathcal{S}},\eta_{\mathcal{K}})$
  and dual feasibilities $(\eta_D,\eta_{\mathcal{S}^*},\eta_{\mathcal{K}^*})$. The exact details
  on the adjustment strategies are not given here. In addition,
  all the solvers also adopt some kind of restart strategies to ameliorate slow convergence.
  During the testing, we use the same adjustment strategy of $\sigma$ and restart strategy
  for all the solvers.

   \medskip

  Figure \ref{3block-fig} shows the performance profiles of IEIDP-ADMM1, IEIDP-ADMM2
  and ADMM3d in terms of number of iterations and computing time, respectively,
  for the total $605$ (including ${\rm BIQ}$(165), ${\rm RCP}$(120), $\theta_{+}$(113),
  ${\rm FAP}$(13) and ${\rm QAP}$(95)) tested problems.
  We recall that a point $(x,y)$ is in the performance profiles curve of a method
  if and only if it can solve $(100y)\%$ of all tested problems no slower than $x$ times of
  any other methods. We see that IEIDP-ADMM1 and ADMM3d need the comparable iterations
  and computing time. Among others, IEIDP-ADMM2 requires the least number of iterations for
  $60\%$ test problems, but it needs the most computing time which is about $1.5$ times that of
  IEIDP-ADMM1 and ADMM3d for about $80\%$ test problems.

  \begin{figure}[htbp]
 \begin{minipage}{0.5\linewidth}
 \centering
 \includegraphics[width=3.1in]{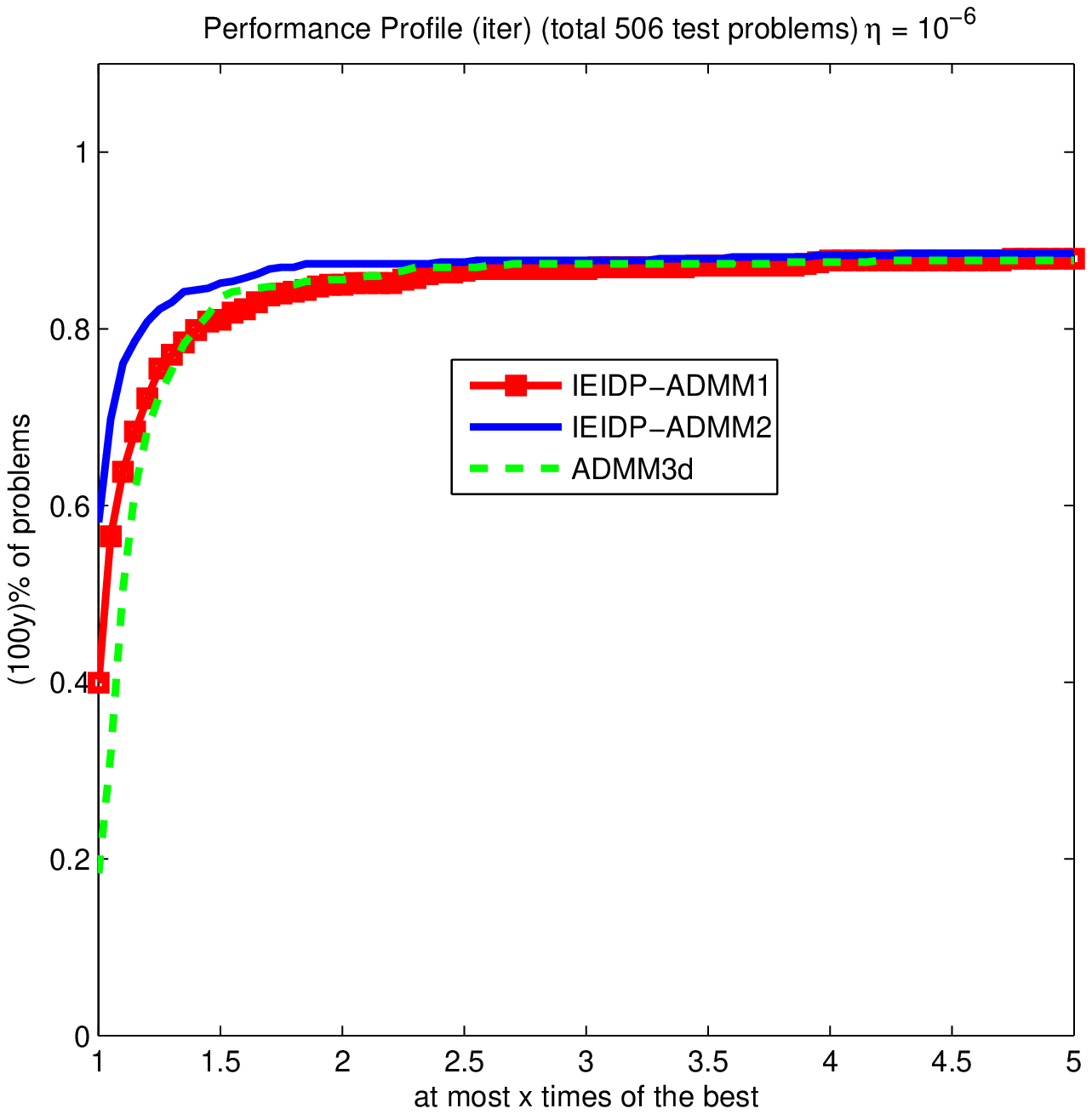}
 \end{minipage}%
 \begin{minipage}{0.5\linewidth}
 \centering
 \includegraphics[width=3.1in]{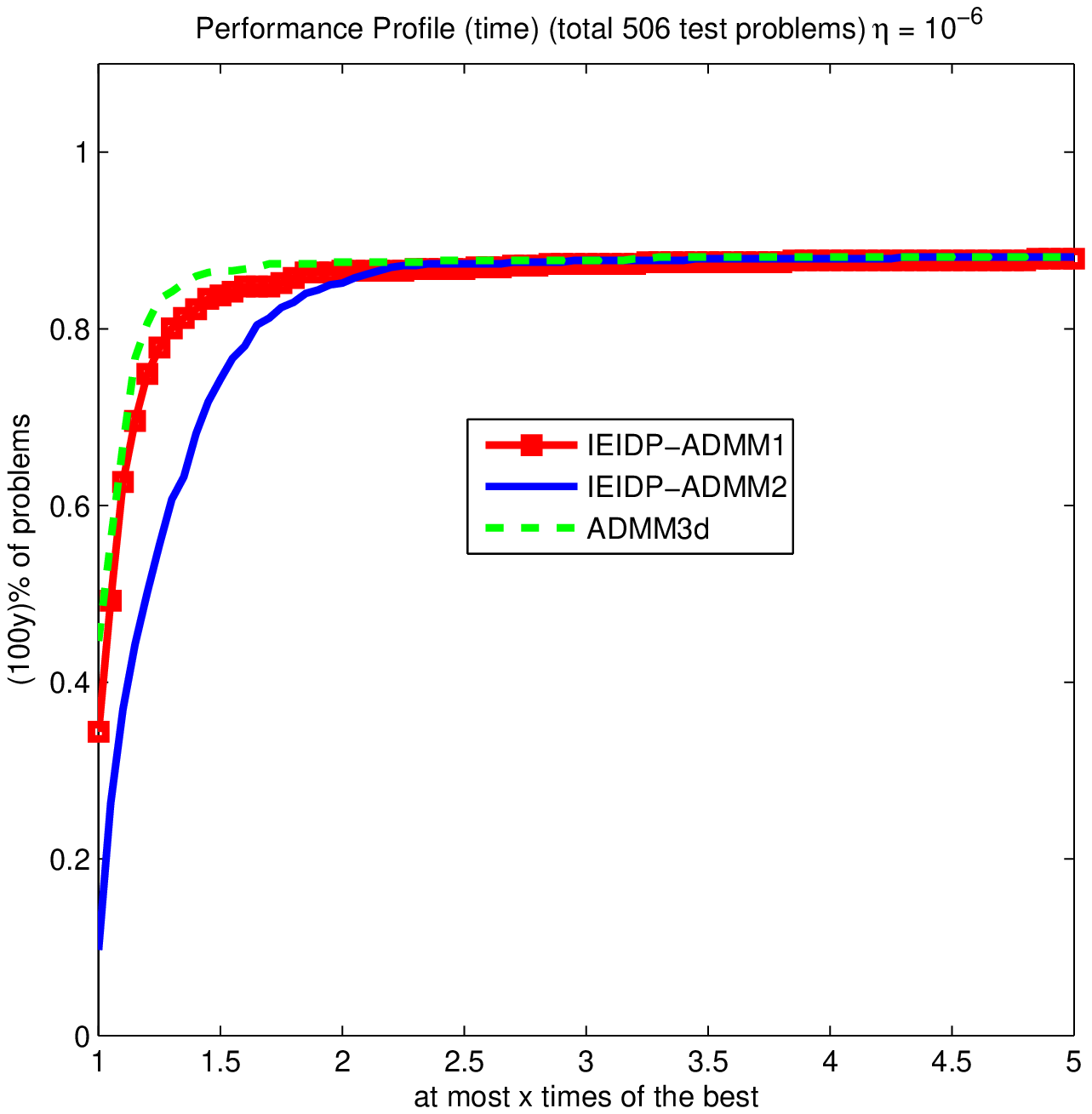}
 \end{minipage}
  \caption{\small Performance profiles of the number of iterations and computing time of solvers}
 \label{3block-fig}
 \end{figure}

 \subsection{Numerical results for the DNNSDPs with $\mathcal{A}_{I}X\ge b_I$}\label{subsec5.2}

  For this case, we may apply the proposed IEIDP-ADMMs for solving \eqref{DDNNSDP} or \eqref{Equi-DDNNSDP}.
  Firstly, we report the numerical results of the IEIDP-ADMMs for solving problem \eqref{DDNNSDP}.

 \subsubsection{Numerical results of the IEIDP-ADMMs for problem \eqref{DDNNSDP}}\label{subsubsec5.2.1}

 Since $[\mathcal{I}\ \ \mathcal{A}_{I}^*]^*[\mathcal{I}\ \ \mathcal{A}_{I}^*]$ and
 $[\mathcal{A}_{E}^*\ \ \mathcal{I}]^*[\mathcal{A}_{E}^*\ \ \mathcal{I}]$ are not positive definite
 and $\mathcal{A}_I$ is not surjective, we introduce the semi-proximal terms
 $\frac{1}{2}(y_I-\!y_I^k,Z-\!Z^k){\rm Diag}(\sigma\mathcal{T},0)(y_I-y_I^k,Z-Z^k)^T$ and
 $\frac{1}{2}(y_E-\!y_E^k,S-S^k){\rm Diag}(0,\varepsilon\mathcal{I})(y_E-y_E^k,S-\!S^k)^T$ to guarantee that
 \begin{equation}\label{Tf-four}
   \mathcal{T}_{\!f}\succeq \sigma[\mathcal{A}_{I}^*\ \ \mathcal{I}]^*[\mathcal{A}_{I}^*\ \ \mathcal{I}]
   +{\rm Diag}(\sigma\mathcal{T},0)
   =\sigma\!\left[\begin{matrix}
                        \varrho\mathcal{I} &\mathcal{A}_{I}\\
                        \mathcal{A}_{I}^* & \mathcal{I}
                  \end{matrix}\right]\succ 0
 \end{equation}
  and
  \begin{equation}\label{Tg-four}
  \mathcal{T}_{\!g}\succeq \sigma[\mathcal{A}_{E}^*\ \ \mathcal{I}]^*[\mathcal{A}_{E}^*\ \ \mathcal{I}]
  +{\rm Diag}(0,\varepsilon\mathcal{I})
  =\sigma\!\left[\begin{matrix}
                        \mathcal{A}_{E}\mathcal{A}_{E}^* &\mathcal{A}_{E}\\
                        \mathcal{A}_{E}^* & \frac{\sigma+\varepsilon}{\sigma}\mathcal{I}
                  \end{matrix}\right]\succ 0,
  \end{equation}
  and propose the following partial inexact indefinite proximal ADMMs for solving (\ref{DDNNSDP}).

 \bigskip

 \setlength{\fboxrule}{0.8pt}
 \noindent
 \fbox{
 \parbox{0.96\textwidth}
 {
 \begin{algorithm} \label{Alg2}({\bf An inexact indefinite-proximal ADMM for (\ref{DDNNSDP})})
 \begin{description}
   \item[(S.0)] Let $\mathcal{T}\!=\varrho\mathcal{I}-\mathcal{A}_{I}\mathcal{A}_{I}^*$
                for $\varrho>\lambda_{\rm max}(\mathcal{A}_{I}\mathcal{A}_{I}^*)$.
                Let $\sigma,\tau>0$ be given. Choose a small \hspace*{0.05cm} constant
                $\varepsilon>0$ and a point $(y_{I}^0,Z^0,y_{E}^0,S^0,X^0)=(0,0,0,0,0)$.
                Set $k:=0$.

   \item[(S.1)] Compute the following problems by one of the criteria (C1)-(C2):
               \begin{align}\label{4subprob}
                 &(y_I^{k+1},Z^{k+1})\approx\mathop{\arg\min}_{y_I,Z}\phi_k(y_I,Z)\!:=L_{\sigma}(y_I,Z,y_E^k,S^{k},X^{k})+\!\frac{1}{2}\|y_I\!-\!y_I^k\|_{\sigma\mathcal{T}}^2\\
                 &(y_E^{k+1},S^{k+1})\approx\mathop{\arg\min}_{y_E,S}\psi_k(y_E,S)\!:=\!L_{\sigma}(y_I^{k+1},Z^{k+1},y_E,S,X^{k})+\!\frac{1}{2}\|S\!-\!S^k\|_{\varepsilon\mathcal{I}}^2.\nonumber
                \end{align}

   \item[(S.3)]  Update the Lagrange multiplier $X^{k+1}$ via the formula
                \[
                  X^{k+1} = X^{k}+\tau\sigma(\mathcal{A}_{I}^*y_I^{k+1}+ Z^{k+1}+\mathcal{A}_{E}^*y_E^{k+1}+S^{k+1} -C).
                \]

  \item[(S.4)]  Let $k\leftarrow k+1$, and go to Step (S.1).

  \end{description}
  \end{algorithm}}
  }

  \bigskip

 One may obtain the approximate optimal solutions $(y_I^{k+1},Z^{k+1})$ and
 $(y_E^{k+1},S^{k+1})$ by computing $\min_{y_I,Z} \phi_k(y_I,Z)$ and $\min_{y_E,S} \psi_k(y_E,S)$ in an alternating way.
 Let $k_0\!=k$. The iterates $(y_{\!_I}^{k_j},Z^{k_j})$ for $j\ge 1$ yielded by minimizing $\phi_k(y_I,Z)$
 alternately satisfy
 \[
  \left\{\begin{array}{l}
    0\in\mathcal{N}_{\mathbb{R}_{+}^{m_I}}(y_{\!_I}^{k_j})\!-b_I+\mathcal{A}_{I}X^k+\sigma\mathcal{A}_{I}(\mathcal{A}_{I}^*y_{\!_I}^{k_j}+Z^{k_{j-1}}+\mathcal{A}_{E}^*y_E^k +\!S^{k}\!-\!C)
        +\sigma\mathcal{T}(y_{\!_I}^{k_j}\!-\!y_{\!_I}^k),\\
    0\in\mathcal{N}_{\mathcal{K}^*}(Z^{k_j})-M+X^k +\sigma(\mathcal{A}_{I}^*y_{\!_I}^{k_{j}}\!+\!Z^{k_j}\!+\!\mathcal{A}_{E}^*y_E^k\!+\!S^{k}\!-\!C).
   \end{array}\right.
 \]
 Let $\xi^{k_j}=\sigma\mathcal{A}_{I}(Z^{k_j}\!-\!Z^{k_{j-1}})$. Comparing the last system with
 the optimality condition of $\min_{y_I,Z} \phi_k(y_I,Z)$, we have $(\xi^{k_j},0)\in\partial\phi_k(y_{\!_I}^{k_j},Z^{k_j})$.
 This means that $(y_{\!_I}^{k_j},Z^{k_j})$ satisfies the criterion (C1) when $\|\xi^{k_j}\|\le \mu_{k+1}$
 and $\sum_{k=0}^{\infty}\mu_{k+1}<\infty$. Notice that
 \begin{align*}
   \mathcal{T}_{\!f}\succeq\left[\begin{matrix}
   \varrho\mathcal{I} &\mathcal{A}_{I}\\
   \mathcal{A}_{I}^* & \mathcal{I}
   \end{matrix}\right]
   &=\!\left[\begin{matrix}
   \mathcal{I} & 0\\
   -\varrho^{-1}\mathcal{A}_{I}^* & \mathcal{I}
   \end{matrix}\right]^{-1}
   \left[\begin{matrix}
   \varrho\mathcal{I} & 0\\
      0 & \mathcal{I}\!-\!\varrho^{-1}\mathcal{A}_I^*\mathcal{A}_I
   \end{matrix}\right]
   \left[\begin{matrix}
   \mathcal{I} & -\varrho^{-1}\mathcal{A}_I\\
   0 & \mathcal{I}
   \end{matrix}\right]^{-1}\\
   &\succeq\!\min(\varrho,\lambda_{\rm min}(\mathcal{I}\!-\!\varrho^{-1}\mathcal{A}_I^*\mathcal{A}_I))
   \left[\begin{matrix}
   \mathcal{I} & 0\\
   -\varrho^{-1}\mathcal{A}_{I}^* & \mathcal{I}
   \end{matrix}\right]^{-1}
   \left[\begin{matrix}
   \mathcal{I} & -\varrho^{-1}\mathcal{A}_I\\
   0 & \mathcal{I}
   \end{matrix}\right]^{-1}\\
 &\succeq\vartheta\lambda_{\rm min}\left(
   \left[\begin{matrix}
   \mathcal{I} & 0\\
   -\varrho^{-1}\mathcal{A}_{I}^* & \mathcal{I}
   \end{matrix}\right]^{-1}
    \left[\begin{matrix}
   \mathcal{I} & -\varrho^{-1}\mathcal{A}_I\\
   0 & \mathcal{I}
   \end{matrix}\right]^{-1}\right)\mathcal{I}\\
 &=\vartheta\lambda_{\rm max}\left(
      \left[\begin{matrix}
   \mathcal{I} & -\varrho^{-1}\mathcal{A}_I\\
   0 & \mathcal{I}
   \end{matrix}\right]
    \left[\begin{matrix}
   \mathcal{I} & 0\\
   -\varrho^{-1}\mathcal{A}_{I}^* & \mathcal{I}
   \end{matrix}\right]\right)^{-1}\mathcal{I}\succeq\vartheta\mathcal{I}
 \end{align*}
 where $\vartheta:=\min\big(\varrho,1\!-\!\varrho^{-1}\lambda_{\rm max}(\mathcal{A}_I^*\mathcal{A}_I)\big)$.
 So, $(y_{\!_I}^{k_j},Z^{k_j})$ satisfies (C2) with
 $\mathcal{F}=\frac{1}{\sigma\vartheta}\mathcal{I}$ when
  \[
   \|\xi^{k_j}\|\le\mu_{k+1}\sqrt{\vartheta}
     \sqrt{\|y_{\!_I}^{k_j}\!-\!y_I^k+\mathcal{A}_{I}^*(Z^{k_j}\!-\!Z^k)\|^2
           +\rho\|y_{\!_I}^{k_j}\!-\!y_I^k\|^2-\|\mathcal{A}_{I}^*(Z^{k_j}\!-\!Z^k)\|^2}.
  \]
  The iterates $(y_{\!_E}^{k_j},S^{k_j})$ for $j\!\ge 1$ yielded by minimizing $\psi_k(y_E,S)$ alternately satisfy
 \[
  \left\{\begin{array}{l}
   0=-b_E +\mathcal{A}_{E}X^k+\sigma\mathcal{A}_{E}(\mathcal{A}_{I}^*y_{I}^{k+1}\!+\!Z^{k+1}\!+\!\mathcal{A}_{E}^*y_{\!_E}^{k_j}\!+\!S^{k_{j-1}}\!-\!C),\\
    0\in\mathcal{N}_{\mathcal{S}_{+}^n}(S^{k_j})+X^k+\sigma(\mathcal{A}_{I}^*y_I^{k+1}+Z^{k+1}+\mathcal{A}_{E}^*y_{\!_E}^{k_j}+S^{k_j}\!-\!C) + \varepsilon(S^{k_j}-S^{k}).
   \end{array}\right.
 \]
  Let $\eta^{k_j}=\sigma\mathcal{A}_{E}(S^{k_j}\!-\!S^{k_{j-1}})$. Comparing the last system
  with the optimality condition of problem $\min_{y_E,S} \psi_k(y_E,S)$, we have
  $(\eta^{k_j},0)\in\partial\psi_k(y_{\!_E}^{k_j},S^{k_j})$. This means that $(y_{\!_E}^{k_j},S^{k_j})$ satisfies
  (C1) when $\|\eta^{k_j}\|\le \nu_{k+1}$ with $\sum_{k=0}^{\infty}\nu_{k+1}<\infty$.
  In addition, let
  \[
   \delta:=(\sqrt{\sigma+\varepsilon}-\!\sqrt{\sigma})\min\Big(
  \frac{\sigma}{\sqrt{\sigma\!+\!\varepsilon}}\lambda_{\rm min}(\mathcal{A}_{E}\mathcal{A}_{E}^*),\sqrt{\sigma\!+\!\varepsilon}\Big).
  \]
 By using equation \eqref{Tg-four} and \cite[Theorem 7.7.6]{HJ86}, it is not difficult to verify that
 \[
  \mathcal{T}_{\!g}\succeq
  \sigma\left[\begin{matrix}
                       \frac{\sqrt{\sigma\!+\!\varepsilon}-\sqrt{\sigma}}{\sqrt{\sigma+\varepsilon}}\mathcal{A}_{E}\mathcal{A}_{E}^* &0\\
                        0&  \frac{\sqrt{\sigma+\varepsilon}-\sqrt{\sigma}}{\sigma}\sqrt{\sigma\!+\!\varepsilon}\mathcal{I}
                  \end{matrix}\right]\succeq\delta \mathcal{I}.
  \]
  This means that $(y_{\!_E}^{k_j},S^{k_j})$ satisfies the criterion (C2) with $\mathcal{G}=\delta^{-1}\mathcal{I}$
  once
  \[
   \|\eta^{k_j}\|
   \le \sqrt{\delta\sigma}\nu_{k+1}\sqrt{\|\mathcal{A}_E^*(y_{\!_E}^{k_j}\!-\!y_E^k)+(S^{k_j}-S^k)\|^2
           +\frac{\varepsilon}{\sigma}\|S^{k_j}\!-\!S^k\|^2}.
  \]
  We call Algorithm \ref{Alg2} with the two subproblems in (S.1) solved alternately
  by the criteria (C1) and (C2) IEIDP-ADMM1 and IEIDP-ADMM2, respectively.

  \medskip

  We apply the IEIDP-ADMM1 and IEIDP-ADMM2 for solving the extended BIQ problems described
  in Section 4.2 of \cite{STY14}, and compare its performance with the four-block
  proximal ADMM of step-size $\tau=1.618$ (although without convergent guarantee)
  by adding a proximal term $\frac{\sigma}{2}\|y_I-y_I^k\|_{\mathcal{T}}^2$ for the $y_I$ part,
  where $\mathcal{T}=\|\mathcal{A}_I\mathcal{A}_I^*\|\mathcal{I}\!-\!\mathcal{A}_I\mathcal{A}_I^*$.
  We call this method PADMM4d. The computational results for all the extended BIQ problems are obtained on
  the same desktop computer as before.

  \medskip

  We measure the accuracy of an approximate optimal solution $(X,y_I,Z,y_E,S)$ for \eqref{PDNNSDP}
  and \eqref{DDNNSDP} by the relative residual
  \(
   \eta=\max\big\{\eta_{P},\eta_{D},\eta_{\mathcal{S}},\eta_{\mathcal{K}},\eta_{\mathcal{S}^*},\eta_{\mathcal{K}^*},\eta_{C_1},\eta_{C_2},\eta_I,\eta_{I^*}\big\},
  \)
  where $\eta_{P},\eta_{\mathcal{S}},\eta_{\mathcal{K}},\eta_{\mathcal{S}^*},\eta_{\mathcal{K}^*},\eta_{C_1},\eta_{C_2}$
  are defined as before, and $\eta_D,\eta_I,\eta_{I^*}$ are given by
  \begin{align*}
   \eta_D\!=\!\frac{\|\mathcal{A}_I^*y_I\!+Z+\mathcal{A}_E^*y_E+\!S-\!C\|}{1+\|C\|},\
   \eta_{I}\!=\!\frac{\|\max(0,b_I-\mathcal{A}_IX)\|}{1+\|b_I\|},\ \eta_{I^*}\!=\!\frac{\|\max(0,-y_I)\|}{1+\|y_I\|}.
  \end{align*}
  The three solvers IEIDP-ADMM1 and IEIDP-ADMM2 and PADMM4d were stopped
  whenever $\eta<10^{-6}$ or the number of iteration is over $k_{\rm max}=40000$.

   \begin{figure}[htbp]
   \begin{minipage}{0.5\linewidth}
   \centering
   \includegraphics[width=3.1in]{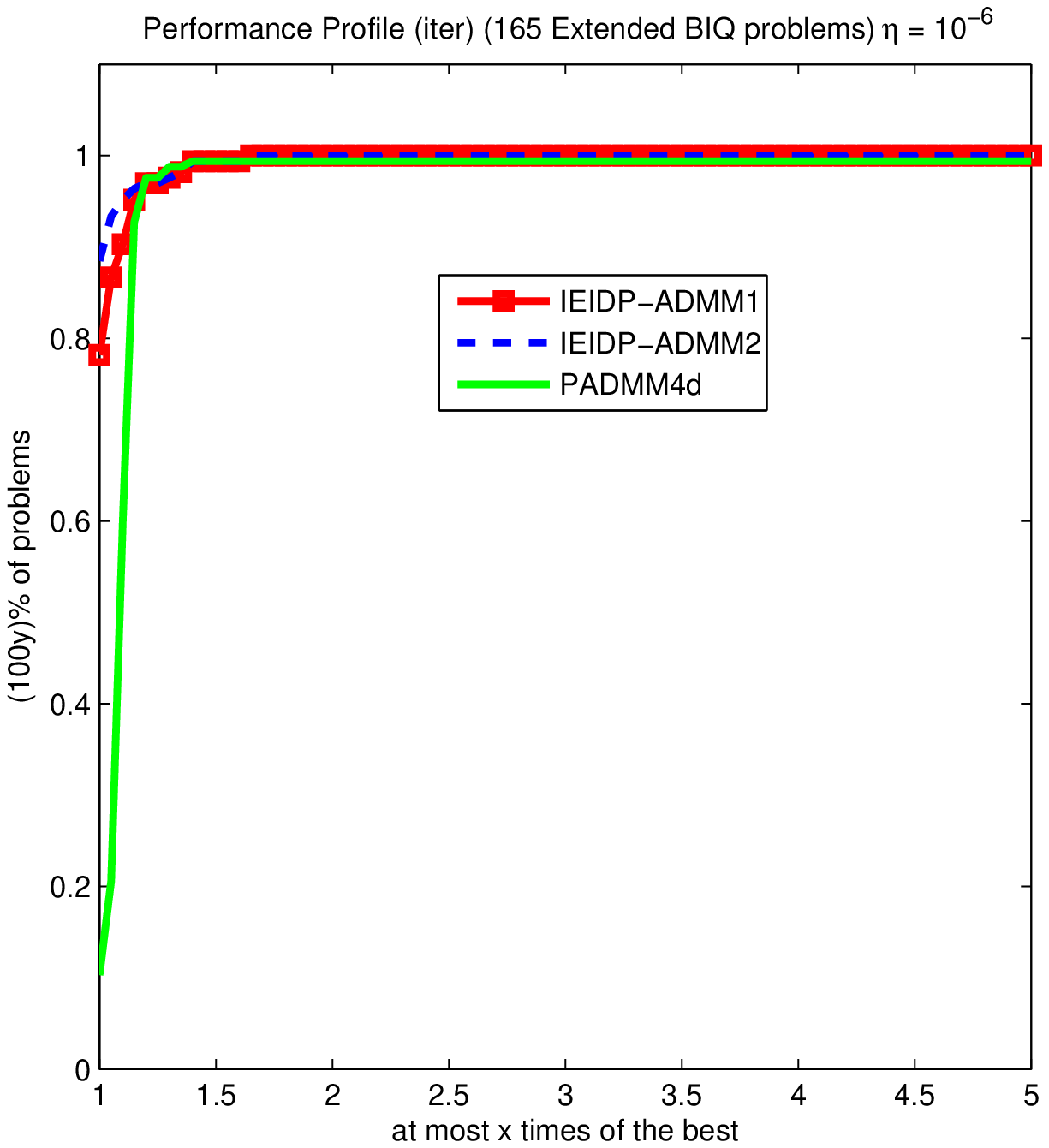}
   \end{minipage}%
   \begin{minipage}{0.5\linewidth}
   \centering
   \includegraphics[width=3.1in]{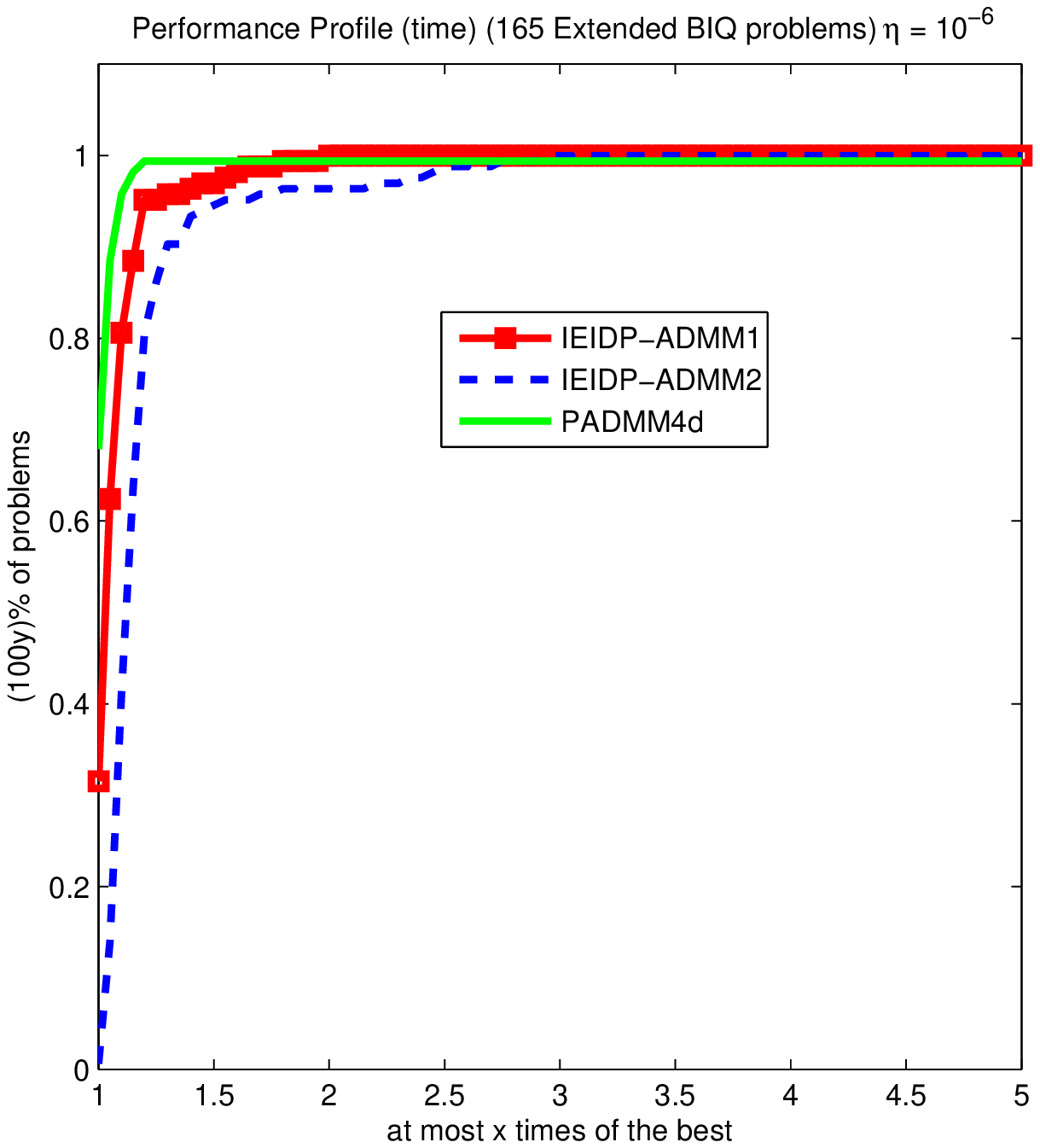}
   \end{minipage}
   \caption{\small Performance profiles of the number of iterations and computing time for EBIQ}
   \label{EBIQ-fig}
  \end{figure}

  \medskip

  Figure \ref{EBIQ-fig} plots the performance profiles of IEIDP-ADMM1, IEIDP-ADMM2 and PADMM4d
  in terms of the number of iterations and computing time, respectively, for the total $165$
  extended BIQ tested problems. It can be seen from this figure that IEIDP-ADMM1, IEIDP-ADMM2
  and PADMM4d are comparable in terms of the iterations and computing time, IEIDP-ADMM1 and
  IEIDP-ADMM2 need the least number of iterations for at least $80\%$ tested problems,
  which is about $90\%$ that of PADMM4d, and PADMM4d requires the least computing time for about
  $70\%$ tested problems, which is about $90\%$ that of IEIDP-ADMM2.

 \subsubsection{Numerical results of the IEIDP-ADMMs for problem \eqref{Equi-DDNNSDP}}\label{subsubsec5.2.2}

  Since
 $
  [\mathcal{A}_{I}^*\ \ 0\ \ \mathcal{I}]^*[\mathcal{A}_{I}^*\ \ 0\ \ \mathcal{I}]
  +[\mathcal{I}\ \ -\mathcal{I}\ \ 0]^*[\mathcal{I}\ \ -\mathcal{I}\ \ 0]
 $
 and
 $[\mathcal{A}_{E}^*\ \ \mathcal{I}]^*[\mathcal{A}_{E}^*\ \ \mathcal{I}]$ are not positive definite,
 we introduce the semi-proximal terms
 $\frac{1}{2}(y_I-\!y_I^k,z-z^k,Z-\!Z^k){\rm Diag}(\sigma\varepsilon\mathcal{I},0,0)(y_I-y_I^k,z-z^k,Z-Z^k)^T$
 and $\frac{1}{2}(y_E-\!y_E^k,S-S^k){\rm Diag}(0,\varepsilon\mathcal{I})(y_E-y_E^k,S-\!S^k)^T$ to ensure that
 \begin{equation}\label{Tf-four}
   \mathcal{T}_{\!f}
   \succeq \sigma\Big([\mathcal{A}_{I}^*\ \ 0\ \ \mathcal{I}]^*[\mathcal{A}_{I}^*\ \ 0\ \ \mathcal{I}]
  +[\mathcal{I}\ \ -\!\mathcal{I}\ \ 0]^*[\mathcal{I}\ \ -\!\mathcal{I}\ \ 0]\Big)
   +{\rm Diag}(\sigma\varepsilon\mathcal{I},0,0)\succ 0
 \end{equation}
  and
  \begin{equation}\label{Tg-four}
  \mathcal{T}_{\!g}\succeq \sigma[\mathcal{A}_{E}^*\ \ \mathcal{I}]^*[\mathcal{A}_{E}^*\ \ \mathcal{I}]
  +{\rm Diag}(0,\varepsilon\mathcal{I})
  =\sigma\!\left[\begin{matrix}
                        \mathcal{A}_{E}\mathcal{A}_{E}^* &\mathcal{A}_{E}\\
                        \mathcal{A}_{E}^* & \frac{\sigma+\varepsilon}{\sigma}\mathcal{I}
                  \end{matrix}\right]\succ 0,
  \end{equation}
  and propose the following inexact indefinite proximal ADMMs for solving (\ref{Equi-DDNNSDP}),
  where for a given $\sigma>0$, the augmented Lagrangian function of problem \eqref{Equi-DDNNSDP} is defined as:
  \begin{align}
     L_{\sigma}(y_I,z,Z,y_E,S,X,x)&:=-\langle b_I,y_I\rangle+\delta_{\mathbb{R}_{+}^{m_I}}(z)
     +(\delta_{\mathcal{K}^*}(Z)\!-\!\langle M,Z\rangle)-\langle b_E,y_E\rangle\nonumber\\
     &\quad\ +\delta_{\mathcal{S}_{+}^n}(S)+\!\langle X,\mathcal{A}_{I}^*y_I\!+\!Z\!+\!\mathcal{A}_{E}^*y_E\!+\!S\!-C\rangle
             +\langle x, y_I-z\rangle\nonumber\\
     &\quad\ +\frac{\sigma}{2}\big\|\mathcal{A}_{I}^*y_I\!+\!Z\!+\!\mathcal{A}_{E}^*y_E\!+\!S\!-C\big\|^2
             +\frac{\sigma}{2}\|y_I-z\|^2\\
    &\forall(y_I,z,Z,y_E,S,X,x)\in\mathbb{R}^{m_I}\times\mathbb{R}^{m_I}\times\mathbb{S}^n\times\mathbb{R}^{m_E}\times\mathbb{S}^n\times\mathbb{S}^n\times\mathbb{R}^{m_I}.\nonumber
   \end{align}.

   \bigskip

 \setlength{\fboxrule}{0.8pt}
 \noindent
 \fbox{
 \parbox{0.96\textwidth}
 {
 \begin{algorithm} \label{Alg3}({\bf An inexact indefinite proximal ADMM for (\ref{Equi-DDNNSDP})})
 \begin{description}
   \item[(S.0)] Let $\sigma,\tau>0$ be given. Choose a sufficiently small constant $\varepsilon>0$ and
                an initial \hspace*{0.05cm} point $(y_I^0,z^0,Z^0,y_E^0,S^0,X^0,x^0)=(0,0,0,0,0,0,0)$.
                Set $k:=0$.

   \item[(S.1)] Compute the following problems by one of the criteria (C1)-(C2):
                \begin{align*}\label{4subprob}
                 &(y_I^{k+1},(z^{k+1},Z^{k+1}))\approx\mathop{\arg\min}_{y_I,z,Z}L_{\sigma}(y_I,z,Z,y_E^k,S^{k},X^{k},x^k)
                                 +\frac{\sigma\varepsilon}{2}\|y_I\!-\!y_I^k\|^2,\\
                 &(y_E^{k+1},S^{k+1})\approx\mathop{\arg\min}_{y_E,S}
                 \!L_{\sigma}(y_I^{k+1},z^{k+1},Z^{k+1},y_E,S,X^{k},x^k)+\!\frac{\varepsilon}{2}\|S\!-\!S^k\|_{}^2.\
                \end{align*}

   \item[(S.3)]  Update the Lagrange multipliers $(X^{k+1},\zeta^{k+1})$ via the following formula
                \begin{align}
                  X^{k+1} &= X^{k}+\tau\sigma( Z^{k+1}+\mathcal{A}_{I}^*v^{k+1}+\mathcal{A}_{E}^*y^{k+1}+S^{k+1} -C),\nonumber\\
                     x^{k+1}&=x^k + \tau\sigma(y_I^{k+1}-z^{k+1}).
                \end{align}

  \item[(S.4)]  Let $k\leftarrow k+1$, and go to Step (S.1).

  \end{description}
  \end{algorithm}}
  }

 \bigskip

 For the approximate optimal solution $(y_I^{k+1},z^{k+1},Z^{k+1})$ in (S.1),
 one may get it by solving the problem $\min_{y_I,z,Z}\phi_k(y_I,z,Z)$ in an alternating way,
 where
 \[
   \phi_k(y_I,z,Z):=L_{\sigma}(y_I,z,Z,y_E^k,S^{k},X^{k},x^k)
                                 +\frac{\sigma\varepsilon}{2}\|y_I\!-\!y_I^k\|^2.
 \]
 The iterates $(y_{\!_I}^{k_j},z^{k_j},Z^{k_j})$ given by solving
 $\min_{y_I,z,Z}\phi_k(y_I,z,Z)$ alternately satisfy
 \begin{equation*}
   y_{\!_I}^{k_j}=\mathop{\arg\min}_{y\in\mathbb{R}^{m_I}}\phi_k(y_I,z^{k_{j-1}},Z^{k_{j-1}}),\
  (z^{k_j},Z^{k_j})=\!\mathop{\arg\min}_{(z,Z)\in\mathbb{R}^{m_I}\times\mathbb{S}^n}\!\phi_k(y_{\!_I}^{k_j},z,Z)\ \ {\rm for}\ j=1,2,\ldots
 \end{equation*}
 with $k_0=k$. We apply the conjugate gradient method to the first minimization, i.e.,
 \[
   \big(\mathcal{A}_I\mathcal{A}_I^*+(1+\varepsilon)\mathcal{I}\big)^{-1}y_I^{k_j}
   =\!\left[z^{k_{j-1}}\!+\!
     \mathcal{A}_I\Big(C\!-\!Z^{k_{j-1}}\!-\!S^k\!-\!\mathcal{A}_E^*y_E^k-\frac{X^k}{\sigma}\Big)
     -\frac{x^k\!-\!b_I}{\sigma}+\varepsilon y_I^k\right]+R^{k_j},
 \]
 where $R^{k_j}$ denotes the error yielded by the conjugate gradient method. Let
 \[
   \xi^{k_j}=\sigma(Z^{k_j}-Z^{k_{j-1}})+\sigma(z^{k_j}-z^{k_{j-1}})-\sigma R^{k_j}.
 \]
 Then, together with the definition of $(z^{k_j},Z^{k_j})$, we have
 $(\xi^{k_j},0,0)\in\partial\phi_k(y_{\!_I}^{k_j},z^{k_j},Z^{k_j})$. This means that $(y_I^{k_j},z^{k_j},Z^{k_j})$
 satisfies (C1) when $\|\xi^{k_j}\|\le \mu_{k+1}$ and $\sum_{k=0}^{\infty}\mu_{k+1}<\infty$.
 For the approximate optimal solution $(y_E^{k+1},S^{k+1})$ in (S.1),
 one may obtain it by solving the corresponding minimization alternately. Also,
 from Subsection \ref{subsubsec5.2.1} it follows that $(y_{\!_E}^{k_j},S^{k_j})$
 satisfies the criterion (C2) with $\mathcal{G}=\delta^{-1}\mathcal{I}$ if
 $\eta^{k_j}=\sigma\mathcal{A}_{E}(S^{k_j}\!-\!S^{k_{j-1}})$ satisfies
  \[
   \|\eta^{k_j}\|
   \le \sqrt{\delta\sigma}\nu_{k+1}\sqrt{\|\mathcal{A}_E^*(y_{\!_E}^{k_j}\!-\!y_E^k)+(S^{k_j}-S^k)\|^2
           +\frac{\varepsilon}{\sigma}\|S^{k_j}\!-\!S^k\|^2}.
  \]
  We call Algorithm \ref{Alg3} with the subproblems solved alternately by (C1) IEIDP-ADMM1.

  \medskip

  We apply the IEIDP-ADMM1 for solving the extended BIQ problems described in Section 4.2 of \cite{STY14},
  and compare its performance with the previous PADMM4d and the four-block ADMM of step-size $\tau=1.618$
  (although without convergent guarantee). We call the latter ADMM4d. The computational results for
  all the extended BIQ problems are obtained on the same desktop computer as before.
  We measure the accuracy of an approximate optimal solution $(X,y_I,z,Z,y_E,S)$ for \eqref{PDNNSDP}
  and \eqref{Equi-DDNNSDP} by the relative residual
  \(
   \eta=\max\big\{\eta_{P},\eta_{D},\eta_{\mathcal{S}},\eta_{\mathcal{K}},\eta_{\mathcal{S}^*},\eta_{\mathcal{K}^*},\eta_{C_1},\eta_{C_2},\eta_I,\eta_{I^*}\big\},
  \)
  where $\eta_{P},\eta_{\mathcal{S}},\eta_{\mathcal{K}},\eta_{\mathcal{S}^*},\eta_{\mathcal{K}^*},\eta_{C_1},\eta_{C_2}$
  are defined as before. The solvers IEIDP-ADMM1 and PADMM4d and ADMM4d were terminated whenever $\eta<10^{-6}$ or
  the number of iteration is over $k_{\rm max}=40000$.

   \begin{figure}[htbp]
   \begin{minipage}{0.5\linewidth}
   \centering
   \includegraphics[width=3.1in]{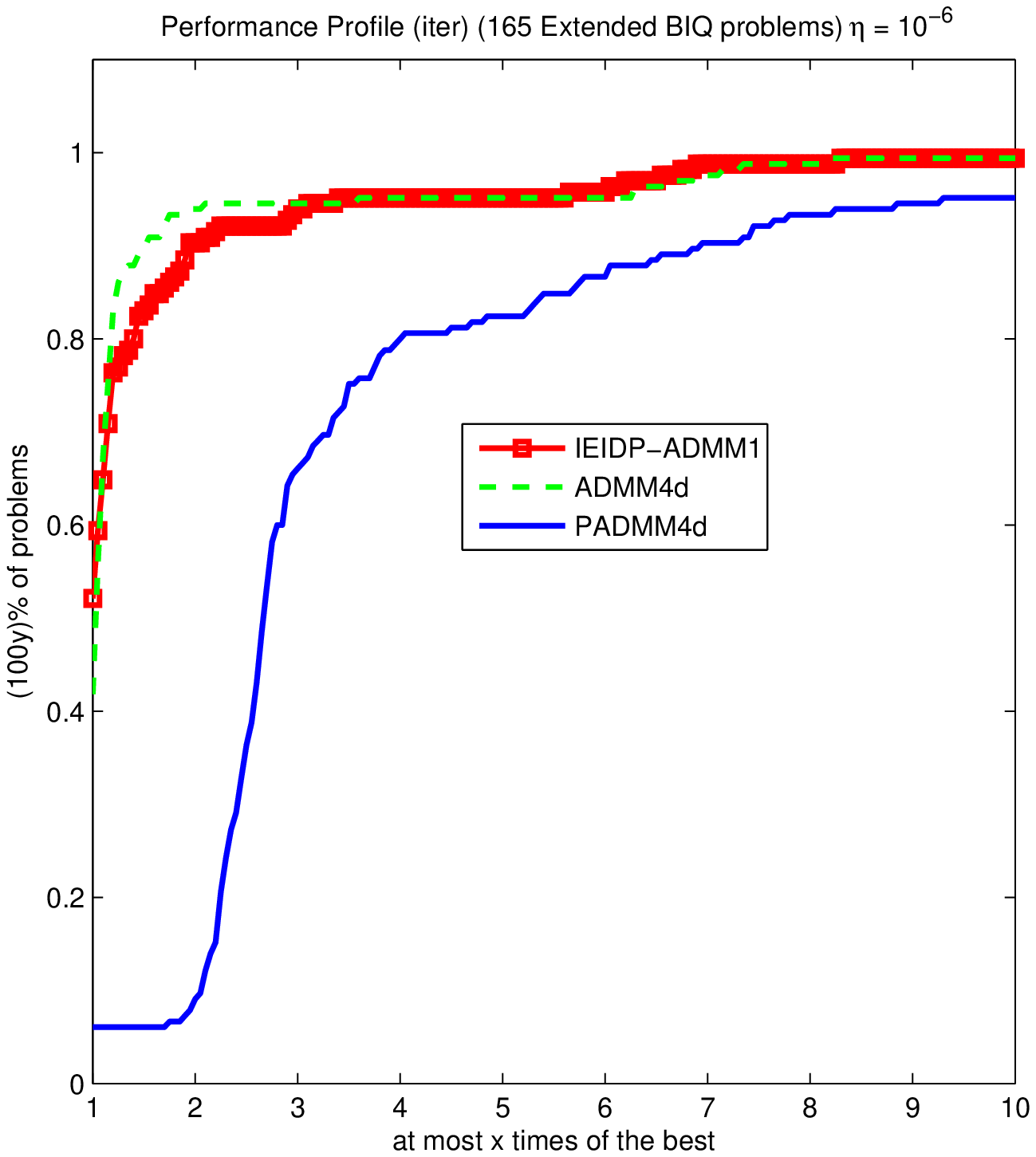}
   \end{minipage}%
   \begin{minipage}{0.5\linewidth}
   \centering
   \includegraphics[width=3.1in]{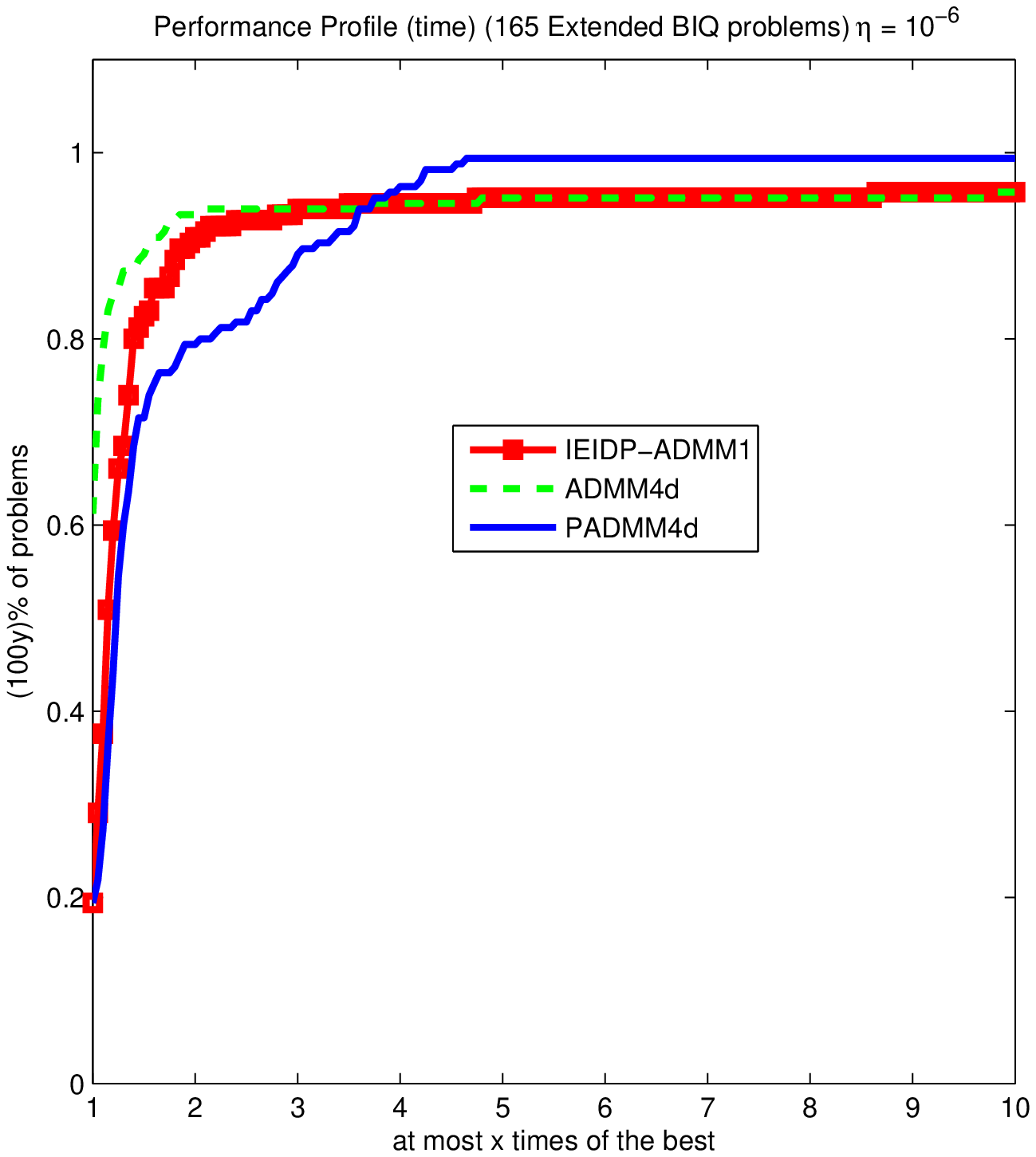}
   \end{minipage}
   \caption{\small Performance profiles of the number of iterations and computing time for EBIQ}
   \label{4EBIQ-fig}
  \end{figure}

   \medskip

  Figure \ref{4EBIQ-fig} plots the performance profiles of IEIDP-ADMM1, PADMM4d and ADMM4d
  in terms of the number of iterations and computing time, respectively, for the total $165$ extended BIQ tested problems.
  We see that, when applying the IEIDP-ADMM for solving the dual problem \eqref{Equi-DDNNSDP},
  the number of iterations and the computing time of the IEIDP-ADMM1 are still comparable with
  those of ADMM4d, but PADMM4d requires more $4$ times iterations than IEIDP-ADMM1
  and ADMM4d as do for at least $80\%$ test problems. This means that a small proximal term
  as possible is the key to the performance of proximal-type ADMMs. The computing time of PADMM4d
  is a little less than that of IEIDP-ADMM1 and ADMM4d since the latter solves an $m_I\times m_I$
  linear system with the conjugate gradient method, where $m_I$ may attain $374250$.

  \section{Conclusion}\label{sec6}

  We developed an inexact indefinite proximal ADMM of step-size $\tau\in\!(0,\frac{\sqrt{5}+1}{2})$
  with two easily implementable inexactness criteria for the two-block separable convex minimization problems
  with linear constraints, for which it is either impossible or too expensive to obtain the exact solutions of
  the subproblems involved in the proximal ADMM. Numerical results for the DNNSDPs with many linear equality
  and/or inequality constraints show that the inexact indefinite proximal ADMMs are effective
  for this class of difficult three or four block separable separable convex optimization problems
  with linear constraints. Among others, the inexact indefinite proximal ADMM with the absolute
  error criterion (C1) is comparable with the directly extended ADMM of step-size $\tau=1.618$,
  whether in terms of the number of iterations or computing time, and is superior to the one with
  the relative error criterion (C2) by weighing the number of iterations and the computing time
  since the latter is very restrictive and requires too many iterations for the solution of subproblems.
  In our future research work, we will explore other easily implementable inexact criteria like relaxing
  $\mu_{k+1}$ and $\nu_{k+1}$ in (C2) to be a constant, and study the nonergodic convergence
  \cite{DY2014a,DY2014b} for the inexact indefinite proximal ADMMs.

 \bigskip
 \noindent
 {\large\bf Acknowledgements.} The authors would like to thank Professor Kim-Chuan Toh from
 National University of Singapore for providing us the code, and Professor Defeng Sun
 from National University of Singapore for the suggestion on numerical comparison of \eqref{DDNNSDP} and \eqref{Equi-DDNNSDP}.


\begin{thebibliography}{1}

  \bibitem{BGA08}
  {\sc O.\ Banerjee, L. E.\ Ghaoui and A.\ \'{d}Aspremont},
  {\em Sparse maximum likelihood estimation for multivariate Gaussian or binary data},
  Journal of Machine Learning Research, vol. 9, pp. 485-516, 2008.


  \bibitem{BL06}
  {\sc J. M.\ Borwein and A. S.\ Lewis},
  {\em Convex Analysis and Nonlinear Optimization: Theory and Examples}, Springer, 2006.



  \bibitem{CHYY14}
  {\sc C. H.\ Chen, B. S.\ He, Y. Y.\ Ye and X. M.\ Yuan},
  {\em The direct extension of admm for multi-block convex minimization problems is not necessarily convergent},
  Mathematical Programming, Series A, DOI 10.1007/s10107-014-0826-5, 2014.



  \bibitem{CHAN07}
  {\sc T.\ Chan, N.\ Ng, A.\ Yau and A.\ Yip},
  {\em Superresolution image reconstruction using fast inpainting algorithms},
  Applied and Computational Harmonic Analysis, vol. 23, pp. 3-24, 2007.


   \bibitem{CWY2014}
   {\sc Z. M.\ Chen, L. Wan and Q. Z.\ Yang},
   {\em An Inexact alternating direction method for structured variational inequalities},
    Journal of Optimization Theory and Applications, vol. 163, pp. 439-459, 2014.



  \bibitem{DY2014a}
   {\sc D.\ Davis and W.\ Yin},
   {\em  Convergence rate analysis of several splitting schemes},
   arXiv preprint arXiv:1406.4834, 2014.

   \bibitem{DY2014b}
   {\sc D.\ Davis and W.\ Yin},
   {\em  convergence rates of relaxed Peaceman-Rachford and ADMM under regularity assumptions},
   arXiv preprint arXiv:1407.5210, 2014.



   \bibitem {EB92}
  {\sc J.\ Eckstein and D. P.\ Bertsekas},
  {\em On the Douglas-Rachford splitting method and the proximal point algorithm for maximal monotone operators},
   Mathematical Programming, vol. 55, pp. 293-318, 1992.


  \bibitem{FPST13}
  {\sc M.\ Fazel, T. K.\ Pong, D. F.\ Sun and P.\ Tseng},
  {\em Hankel matrix rank minimization with applications to system identification and realization},
  SIAM Journal on Matrix Analysis, vol. 34, pp. 946-977, 2013.


  \bibitem{FHWY2014}
  {\sc X. L. Fu, B. S. He, X. F. Wang and X. M. Yuan},
  {\em Block-wise alternating direction method of multipliers with Gaussian back substitution for multiple-block convex programming},
   Manuscript, 2014.


  \bibitem {GM75}
 {\sc R.\ Glowinski and A.\ Marrocco},
 {\em Sur l¡¯ approximation par \'{e}l\'{e}ments finis d'ordre un, etla r\'{e}solution, par p\'{e}nalisation-dualit\'{e},
  d'une classe de probl\`{e}mes de dirichlet non lin\'{e}ares},
  Revue Francaise d' Automatique, Informatique et Recherche Op\'{e}rationelle, vol. 9, pp. 41-76, 1975.

  \bibitem {GM76}
  {\sc D.\ Gabay and B.\ Mercier},
  {\em A dual algorithm for the solution of nonlinear variational problems via finite element approximation},
  Computers and Mathematics with Applications, vol. 2, pp. 17-40, 1976.


  \bibitem{GHY2014}
   {\sc G. Y.\ Gu, B. S.\ He and J. F.\ Yang},
   {\em Inexact alternating direction based contraction methods for separable linearly constrained convex optimization},
    Journal of Optimization Theory and Applications, vol. 163, pp. 105-129, 2014.



  \bibitem{Henstenes76}
  {\sc M. R.\ Hestenes},
 {\em Multiplier and gradient methods},
  Journal of Optimization Theory and Applications, vol. 4, pp. 303-320, 1969.


  \bibitem{HLHY02}
  {\sc B. S.\ He, L. Z.\ Liao, D. R.\ Han and H.\ Yang},
  {\em A new inexact alternating directions method for monotone variational inequalities},
  Mathematical Programming, vol. 92, pp. 103-118, 2002.

 \bibitem {HTY12}
 {\sc B. S.\ He, M.\ Tao and X. M.\ Yuan},
 {\em Alternating direction method with Gaussian back substitution for separable convex programming},
 SIAM Journal on Optimization, vol. 22, pp. 313-340, 2012.


 \bibitem {HY13}
 {\sc B. S.\ He and X. M.\ Yuan},
 {\em Linearized alternating direction method of multipliers with Gaussian back substitution for separable
 convex programming}, Numerical Algebra Control Optimization, vol. 3, pp. 247-260, 2013.

  \bibitem{HXY14}
  {\sc B. S.\ He, M. H.\ Xu and X. M.\ Yuan},
  {\em Block-wise ADMM with a relaxation factor for multiple-block convex programming},
   Manuscript, 2014.

   \bibitem {HTY14}
 {\sc B. S.\ He, M.\ Tao and X. M.\ Yuan},
  {\em A splitting method for separable convex programming},
  IMA Journal of Numerical Analysis, vol. 22, pp. 1-33, 2014.



  \bibitem{HJ86}
  {\sc R. A.\ Horn and C. R.\ Johnson},
  {\em Matrix Analysis}, Cambridge University Presss, Cambridge, 1991.


  \bibitem {LST2014}
  {\sc M.\ Li, D. F.\ Sun and K.-C.\ Toh},
  {\em A majorized ADMM with indefinite proximal terms for linearly constrained
   convex composite optimization}, arXiv preprint arXiv:1412.1911, 2014.


  \bibitem{NWY10}
  {\sc M. K.\ Ng, P.\ Weiss and X. M.\ Yuan},
  {\em Solving constrained total-variation image restoration and reconstruction problems via alternating direction methods},
  SIAM Journal on Scientific Computing, vol. 32, 2710-2736, 2010.

   \bibitem{NWY11}
  {\sc M. K.\ Ng, F.\ Wang and X. M.\ Yuan},
  {\em Fast minimization methods for solving constrained total-variation superresolution image reconstruction},
  Multidimensional Systems and Signals Processing, vol. 22, pp. 259-286, 2011.


  \bibitem {NWY11}
  {\sc M. K.\ Ng, F.\ Wang and X. M.\ Yuan},
  {\em Inexact alternating direction methods for image recovery},
  SIAM Journal on Scientific Computing, vol. 33, pp. 1643-1668, 2011.

  \bibitem{Powell69}
  {\sc M.\ Powell},
 {\em A method for nonlinear constraints in minimization problems}, in Optimization,
  R. Fletcher, ed., Academic Press, 1969, pp. 283-298.


  \bibitem{RW98}
  {\sc R. T.\ Rockafellar and R. J-B.\ Wets}, {\em Variational Analysis}, Springer, 1998.


   \bibitem{Roc70}
  {\sc R. T.\ Rockafellar},
  {\em Convex Analysis}, Princeton University Press, Princeton, NJ, 1970.


   \bibitem{Roc76}
  {\sc R. T.\ Rockafellar},
  {\em Augmented Lagrangians and applications of the proximal point algorithm in convex programming},
  Mathematics of Operations Research, vol. 1, pp. 97-116, 1976.


   \bibitem{ROF92}
  {\sc L.\ Rudin, S. J.\ Osher and E.\ Fatemi},
  {\em Nonlinear total variation based noise removal algorithms},
    Physica D, vol. 60, pp. 259¨C268, 1992.

  \bibitem{SMG10}
  {\sc K.\ Scheinberg, S. Q.\ Ma and D.\ Goldfarb},
  {\em Sparse inverse covariance selection via alternating linearization methods},
  in Advances in Neural Information Processing Systems, 2010.




  \bibitem {STY14}
 {\sc D. F.\ Sun, K. C.\ Toh and L. Q.\ Yang},
 {\em A convergent proximal alternating direction method of multipliers for conic programming with $4$-block constraints},
  arXiv preprint arXiv:1404.5378, 2014.


 \bibitem {TYuan11}
 {\sc M.\ Tao and X. M.\ Yuan},
 {\em  Recovering low-rank and sparse components of matrices from incomplete and noisy observations},
  SIAM Journal on Optimization, vol. 21, pp. 57-81, 2011.


  \bibitem {WGY10}
 {\sc Z. W.\ Wen, D.\ Goldfarb and W. T.\ Yin},
 {\em Alternating direction augmented Lagrangian methods for semidefinite programming},
  Mathematical Programming Computation, vol. 12, pp. 203-230, 2012.


  \bibitem{WGRPM09}
 {\sc J.\ Wright, A.\ Ganesh, S.\ Rao, Y.\ Peng and Y.\ Ma},
 {\em Robust principle compoenent analysis: exact recovery of corrupted low-rank matrices by convex optimization},
 in Proceeding of Neural Information Processing Systems, 3(2009).

   \bibitem{WY12}
  {\sc X. F.\ Wang and X. M.\ Yuan},
  {\em The linearized alternating direction method for Dantzig selector},
  SIAM Journal on Scientific Computing, vol. 34, pp. A2792-A2811, 2012.


  \bibitem {WXL13}
 {\sc Y.\ Wang, H.\ Xu and C.\ Leng},
 {\em  Provable subspace clustering: when LRR meets SSC}, in NIPS 2013, Lake Tahoe, 2013.

   \bibitem{WHML2013}
   {\sc X. F.\ Wang, M. Y.\ Hong, S. Q.\ Ma and Z. Q.\ Luo},
   {\em Solving multiple-block separable convex minimization problems using two-block alternating direction method of multipliers},
   arXiv preprint arXiv:1308.5294, 2013.


 \bibitem {XW11}
 {\sc M. H.\ Xu and T.\ Wu},
 {\em  A class of linearized proximal alternating direction methods},
  Journal of Optimization Theory and Applications, vol. 151, pp. 321-327, 2011.


  \bibitem{Yuan12}
  {\sc X. M.\ Yuan},
  {\em Alternating direction methods for sparse covariance selection},
   Journal of Scientific Computing, vol. 51, pp. 261-273, 2012.


  \bibitem {ZJD13}
  {\sc Y. M.\ Zhang, Z. L.\ Jiang and L. S.\ Davis},
  {\em  Learning structured low-rank representations for image classification},
   IEEE Conference on Computer Vision and Pattern Recognition, 2013.


   \bibitem{ZBO11}
  {\sc X.\ Zhang, M.\ Burger and S.\ Osher},
  {\em A unified primal-dual algorithm framework based on Bregman iteration},
  Journal of Scientific Computing, vol. 46, pp. 20-46, 2011.

 \end{thebibliography}
\end{document}